\documentclass[10pt, reqno]{amsart} 
\usepackage{amsmath,amssymb,amsfonts,amscd}
\usepackage[margin = 1.4in]{geometry}
 \usepackage[all]{xy}
\usepackage[usenames,dvipsnames]{xcolor}
\usepackage{mathrsfs}
\usepackage{tikz-cd} 
\usepackage{comment}
\usepackage{stmaryrd}
\usepackage{hyperref}
\usepackage{framed}
\usepackage{enumerate}

\DeclareFontFamily{T1}{cbgreek}{}
\DeclareFontShape{T1}{cbgreek}{m}{n}{<-6>  grmn0500 <6-7> grmn0600 <7-8> grmn0700 <8-9> grmn0800 <9-10> grmn0900 <10-12> grmn1000 <12-17> grmn1200 <17-> grmn1728}{}
\DeclareSymbolFont{quadratics}{T1}{cbgreek}{m}{n}
\DeclareMathSymbol{\qoppa}{\mathord}{quadratics}{19}
\DeclareMathSymbol{\Qoppa}{\mathord}{quadratics}{21}

\usepackage{amssymb}

 \numberwithin{equation}{section}
 \newcommand{\ab}{\mathrm{ab}}
\newcommand{\Spec}{\mathrm{Spec}\,}
\newcommand{\Z}{\mathbf{Z}}
\newcommand{\C}{\mathbf{C}} 
\newcommand{\N}{\mathbf{N}} 
\newcommand{\Ext}{\mathrm{Ext}}
\newcommand{\Pic}{\mathrm{Pic}}

\newcommand{\Hom}{\mathrm{Hom}}
\newcommand{\Q}{\mathbf{Q}}
\newcommand{\Gal}{\mathrm{Gal}}
\newcommand{\univ}{\mathrm{univ}}
\newcommand{\Sing}{\mathrm{Sing}}
\newcommand{\ST}{\mathsf{ST}}
\newcommand{\PP}{\mathcal{P}}
\newcommand{\PPminus}{\mathcal{P}^-_E}
\newcommand{\PPplus}{\mathcal{P}^+_E}
\newcommand{\PPqplus}{\mathcal{P}_{K_q}^{+}}
\newcommand{\PPqminus}{\mathcal{P}_{K_q}^{-}}

\newcommand{\Utop}{\mathcal{U}(\C^{\mathrm{top}})}

 \newcommand{\OO}{\mathcal{O}}

\newcommand{\F}{\mathbf{F}}

\newcommand{\PS}{\mathbf{P}}
\newcommand{\KSp}{\mathrm{KSp}}

\newcommand{\CC}{\mathbf{C}}

\newcommand{\tr}[0]{\operatorname{tr}}

\newcommand{\ul}[1]{\underline{#1}}
\newcommand{\ol}[1]{\overline{#1}}

\newcommand{\Cal}[1]{\mathcal{#1}}

\newcommand{\mbf}[1]{\mathbf{#1}}

\newcommand{\mrm}[1]{\mathrm{#1}}
\newcommand{\co}{\colon}
\newcommand{\mf}[1]{\mathfrak{#1}}
\newcommand{\G}{\mbf{G}}
\newcommand{\rsHom}{\mrm{R}\kern -.5pt\mathscr{H}\kern -.7pt om}

\DeclareMathOperator{\Id}{Id}
\DeclareMathOperator{\Art}{Art}

\DeclareMathOperator{\Ker}{Ker}
\DeclareMathOperator{\pt}{pt}

\DeclareMathOperator{\Ind}{Ind}
\DeclareMathOperator{\Aut}{Aut}
\DeclareMathOperator{\Et}{\acute{E}t}

\DeclareMathOperator{\et}{\acute{e}t}

\DeclareMathOperator{\cyc}{cyc}

\DeclareMathOperator*{\holim}{\mathrm{holim}}
\DeclareMathOperator*{\colim}{\mathrm{colim}}

\usepackage{color}

\newcommand{\TFadd}[1]{{\color{blue} #1}}

\newcommand{\quadr}[1]{\mathcal{Q}(#1)}  
\newcommand{\KH}{\mathit{KH}}  
\newcommand{\KBott}{K^{(\beta)}} 
 \newcommand{\Rtop}{\R^{\mathrm{top}}} 
\newcommand{\Ctop}{\C^{\mathrm{top}}} 

\newcommand{\proj}[1]{\mathcal{P}(#1)}  
\newcommand{\pic}[1]{\mathrm{Pic}(#1)}  
\newcommand{\SP}{\mathcal{SP}} 

\newcommand{\R}{\mathbf{R}}

\newcommand{\Sp}{\mathrm{Sp}}
\newcommand{\K}{\mathrm{K}}
\newcommand{\GL}{\mathrm{GL}}

\newtheorem{theorem}{Theorem}[section]
\newtheorem{lemma}[theorem]{Lemma} 
\newtheorem{corollary}[theorem]{Corollary} 
 
\newtheorem{proposition}[theorem]{Proposition} 

\theoremstyle{remark}
\newtheorem{remark}[theorem]{Remark}
\newtheorem{definition}[theorem]{Definition}
\newtheorem{example}[theorem]{Example}

\makeatletter
\def\th@remark{%
  \thm@headfont{\bfseries}%
  \normalfont 
  \thm@preskip \thm@preskip 
  \thm@postskip\thm@preskip
}
\def\imod#1{\allowbreak\mkern5mu({\operator@font mod}\,\,#1)}
\makeatother

\begin{document} 

 \title{The Galois action on symplectic K-theory}
 \author{Tony Feng, Soren Galatius, Akshay Venkatesh}
 
\begin{abstract} 
We study a symplectic variant of algebraic $K$-theory of the integers, which comes equipped with a canonical action of the absolute Galois group of $\Q$. 
We compute this action explicitly.
The representations we see are extensions of Tate twists $\Z_p(2k-1)$ by a trivial representation, and we characterize them by a universal property among such extensions.

The key tool in the proof is the theory of complex multiplication for abelian varieties. 
   \end{abstract}

\maketitle
\tableofcontents


\section{Introduction} 
\subsection{Motivation and results}
Let $\mathrm{Sp}_{2g}(\Z)$ be the group of automorphisms of $\Z^{2g}$ preserving the standard symplectic form $\langle x, y \rangle = \sum_{i=1}^g (x_{2i-1} y_{2i} - x_{2i} y_{2i-1})$. The group homology
\begin{equation}\label{eq:20}
  H_i(\Sp_{2g}(\Z);\Z_p)
\end{equation}
with coefficients in the ring of $p$-adic numbers, carries a natural action of the group $\Aut(\C)$ which comes eventually from the relationship between $\Sp_{2g}(\Z)$ and $\mathcal{A}_g$, the moduli stack of principally polarized abelian varieties; we discuss this in more detail in \S \ref{elab}.  It is a natural question to understand this action; 
indeed, studying the  actions of Galois automorphisms
on (co)homology of arithmetic groups has been a central concern of number theory.

It was proved by Charney (\cite[Corollary 4.5]{Cha87}) that the homology groups~(\ref{eq:20}) are independent of $g$, as long as $g \geq 2i+5$, in the sense that  the evident inclusion $\Sp_{2g}(\Z) \hookrightarrow \Sp_{2g+2}(\Z)$   induces an isomorphism in group homology.  These maps are also equivariant for $\Aut(\C)$,
and so it is sensible to ask how  $\Aut(\C)$ acts on the {\em stable homology} $$H_i(\Sp_{\infty}(\Z); \Z_p) :=   \varinjlim_g H_i(\Sp_{2g}(\Z); \Z_p).$$

The answer to this question with rational $\Q_p$-coefficients is straightforward. The homology in question has
an algebra structure induced
by $\Sp_{2g_1} \times \Sp_{2g_2} \hookrightarrow\Sp_{2(g_1+g_2)}$, and is isomorphic to a polynomial algebra:
$$ H_*(\Sp_{\infty}(\Z); \Q_p) \simeq \Q_p[x_2, x_6, x_{10}, \dots]$$
and $\Aut(\C)$ acts on $x_{4k-2}$ by the $(2k-1)$st power of the cyclotomic character. 
The elements $x_2, x_6, \dots, $
can be chosen primitive with respect to the coproduct on homology.

With $\Z_p$ coefficients,  it is not  simple even to describe the stable
homology as an abelian group.  However, the situation looks much more elegant
after passing to a   more homotopical invariant---the \emph{symplectic $K$-theory} $\KSp_i(\Z;\Z_p)$---which
can be regarded as a distillate of the stable homology.
We recall the definition in \S \ref{elab}; for the moment we just  note that
 $\Aut(\C)$ also acts on the symplectic $K$-theory and 
there is an equivariant morphism
\begin{equation}
\KSp_i(\Z; \Z_p) \rightarrow H_i(\Sp_{\infty}(\Z); \Z_p)\label{eq:9}
\end{equation}
which, upon tensoring with $\Q_p$, identifies the left-hand side
with the primitive elements in the right-hand side. 
In particular,
\begin{equation} \label{kspq} 
\KSp_i(\Z; \Z_p) \otimes \Q_p \cong \begin{cases} \Q_p(2k-1), & i = 4k-2 \in \{2, 6, 10, \dots\}, \\ 
0 & \mbox{else}, \end{cases} \end{equation}
where $\Q_p(2k-1)$ denotes $\Q_p$ with the $\Aut(\C)$-action given by the
$(2k-1)$st power of the cyclotomic character. 

The identification of \eqref{kspq} can be made very explicit.
The moduli stack of principally polarized abelian varieties carries a Hodge vector bundle (see \S \ref{elab})
whose Chern character classes induce maps $\mathrm{ch}_{2k-1}: H_{4k-2}(\Sp_{2g}(\Z); \Z_p) \rightarrow \Q_p$.  Passing to the limit $g \to \infty$ and composing with~\eqref{eq:9} gives rise to homomorphisms $c_H:  \KSp_{4k-2}(\Z; \Z_p) \rightarrow  \Q_p$ for all $k \geq 1$;
 then $c_H \otimes \Q_p$ recovers \eqref{kspq}.

\subsubsection{Statement of main results}  For each $n$, let $\Q(\zeta_{p^n})$ be the cyclotomic field
obtained by adjoining $p^n$th roots of unity,
and let $H_{p^n}$ be the maximal everywhere unramified abelian
extension of $\Q(\zeta_{p^n})$ of $p$-power degree; put $H_{p^{\infty}} = \bigcup H_{p^n}$.  We regard these as subfields of $\C$.

\medskip

\begin{quote}
\textbf{Main theorem} (see Theorem~\ref{Zpuniversal}).  Let $p$ be an odd prime.
 \begin{itemize}
 \item[(i)] The map $c_H: \KSp_{4k-2}(\Z; \Z_p) \rightarrow \Z_p(2k-1)$  is surjective and equivariant for the $\Aut(\C)$ actions;
\item[(ii)]  The kernel of $c_H$ is a finite $p$-group with trivial $\Aut(\C)$ action;
\item[(iii)] The action of $\Aut(\C)$ factors through the Galois group
 $\Gamma$ of $H_{p^{\infty}}$ over $\Q$.  The short exact sequence 
 \begin{equation} \label{introseq} \Ker(c_H) \hookrightarrow \KSp_{4k-2}(\Z; \Z_p) \stackrel{c_H}{\longrightarrow} \Z_p(2k-1) \end{equation} 
 is initial among
 all such extensions of $\Z_p(2k-1)$ by a $\Gamma$-module with trivial action
 (all modules being $p$-complete and equipped with continuous $\Gamma$-action). 
\end{itemize}
\end{quote}

\medskip

In particular,  the extension of $\Aut(\C)$-modules $\Ker(c_H) \rightarrow \KSp_{4k-2}(\Z; \Z_p) \rightarrow \Z_p(2k-1)$
is not split if $\Ker(c_H)$ is nontrivial,  and in this case the $\Aut(\C)$-action
does {\em not} factor through the cyclotomic character.
In fact $\Ker(c_H)$ is canonically isomorphic to the $p$-completed algebraic $K$-theory $K_{4k-2}(\Z;\Z_p)$ which, through the work of Voevodsky and Rost, and Mazur and Wiles, we know is non-zero precisely when $p$ divides the numerator of $\zeta(1-2k)$ (see \cite[Example 44]{Wei05}).  The first example is $k=6, p=691$. 
 The group $\Gamma = \Gal(H_{p^{\infty}}/\Q)$ itself is a central object of Iwasawa theory;
it surjects onto $ \Z_p^{\times}$ via the cyclotomic character,
with abelian kernel.  In general $\Gamma$ 
is non-abelian, with its size is controlled 
 by the $p$-divisibility of $\zeta$-values.
 
  \begin{remark} The theorem addresses degree $4k-2$; this is the only interesting case.  For $i = 4k$ or $4k+1$ with $k>0$, we explain in \S \ref{sec:symplectic-K} that $\KSp_i(\Z; \Z_p) = 0$. For $i = 4k+3$, $\KSp_{4k+3}(\Z; \Z_p) \cong K_{4k+3}(\Z; \Z_p)$ is a finite group, and we establish in \S \ref{sec:degree-4k-1} that the $\Aut(\CC)$-action on $\KSp_i(\Z; \Z_p)$ is trivial in those cases. 
 \end{remark}

\subsubsection{Other formulations} There are other, equally reasonable, universal properties that can be formulated. 
 For example---and perhaps more natural from the point of view of number theory---$\KSp_{4k-2}(\Z; \Z_p)$ can be considered as the fiber, over $\Spec \C$, of an {\'e}tale sheaf on $\Z[1/p]$;
 then it is (informally) the largest split-at-$\Q_p$ extension of $\Z_p(2k-1)$ by a trivial \'etale sheaf. 
  See \S \ref{univ2} for more discussion of this and other universal properties.

 \subsubsection{} The prediction of the Langlands program is---informally---that
``every Galois representation that looks like it could arise in the cohomology of arithmetic groups,
in fact does so arise.'' In the cases at hand there is no more exact  conjecture available; but we regard the universality
statement above as fulfilling the spirit of this prediction.   The occurrence of extensions
as in \eqref{introseq} is indeed familiar from the Langlands program, where
they arise (see e.g. \cite{Ribet}) in the study of congruences between Eisenstein series and cusp forms. 
They arise in our context in a very direct way, and our methods are also quite different. 

It would  be of interest to relate our results to the study of the action of Hecke operators on stable cohomology;
the latter has been computed for $\GL_n$ by Calegari and Emerton \cite{CE}. 
 
\subsubsection{Consequences} Before we pass to a more detailed account,  let us indicate a geometric  implication of this result (which is explained in more detail in \S \ref{sec: stable homology}).

If $\Cal{A} \rightarrow S$ is an principally polarized abelian scheme over $\Q$ with fiber dimension $g$
then one has a classifying map $S \rightarrow \Cal{A}_g$. If $S$ is projective over $\Q$ of odd dimension $(2k-1)$, then we get a cycle class $[S] \in H_{4k-2}(\Cal{A}_g; \Z_p)$ which transforms according under $\Gal(\ol{\Q}/\Q)$ by the $(2k-1)$st power of the cyclotomic character.  (Examples of this situation can be constructed arising from a Shimura variety, or from the relative Jacobian of a family of curves.)
By pairing $[S]$ with the Chern character of the  Hodge bundle, we get a characteristic number $$c_{\mrm{H}}([S]) 
\in  \Q$$
of the family.
If the numerator of $c_{\mrm{H}}([S])$ is not divisible by $p$  then
$[S]$ splits the analogue of the sequence \eqref{introseq}, but replacing
$\KSp_{4k-2}$ by $H_{4k-2}$. Now, in the range when  $p > 2k$
we may in fact identify $\KSp_{4k-2}$ as a quotient of $H_{4k-2}$
as a Galois module (see \S \ref{sec: stable homology}), and thereby the sequence \eqref{introseq} itself splits. 
Comparing with our theorem, we see that
\begin{equation} \label{bwmc}
\mbox{ $p >2k$ divides numerator of $\zeta(1-2k) \implies  p$ divides numerator of $c_{\mathrm{H}}([S]) $.}
\end{equation}

In other words, our theorem gives a universal divisibility for characteristic numbers
of families of abelian varieties over $\Q$.

 \subsection{Symplectic $K$-theory of $\Z$: definition, Galois action, relationship with usual $K$-theory.} \label{elab}
 We now give some background to the discussion of the previous section,
 in particular outlining the definition of symplectic $K$-theory and
 where the Galois action on it comes from.  
 For the purposes of this section we adopt
  a slightly {\em ad hoc} approach to $K$-theory that differs somewhat
  from the presentation in the main text (\S \ref{sec:symplectic-K}),
but is implicit in the later discussion where the Galois action is constructed (\S \ref{galois action construction}).    
  More detailed explanations
 are given in the  later text. 
  
 First let us explain in more detail the Galois action on the {\em homology} of $\Sp_{2g}(\Z)$
 with $\Z_p$ coefficients. 
 As usual in topology, the group homology of a discrete group such as $\Sp_{2g}(\Z)$
 can be computed
 as the singular homology of its classifying space 
 $ B \Sp_{2g}(\Z)$, which can be modeled by
 the quotient of a contractible $\Sp_{2g}(\Z)$-space with sufficiently free action. 
 In the case at hand,   there is a natural model for this classifying space that
 arises in algebraic geometry:

The group   $\mathrm{Sp}_{2g}(\Z)$ acts on the contractible Siegel upper half plane $\mathfrak{h}_g$  (complex symmetric $g \times g$ matrices
with positive definite imaginary part) and uniformization of abelian varieties identifies the quotient $\mathfrak{h}_g/\!\!/\mathrm{Sp}_{2g}(\Z)$, as a complex orbifold, with the complex points of $\Cal{A}_g$ in the analytic topology.  Since $\mathfrak{h}_g$ is contractible, we may identify the cohomology group $H^i(\Sp_{2g}(\Z);\Z/p^n\Z)$ with the sheaf cohomology of the constant sheaf $\Z/p^n\Z$ on $\mathcal{A}_{g,\C}$, which by a comparison theorem is identified with \'etale cohomology.  The fact that $\mathcal{A}_g$ is defined over $\Q$ associates a map of schemes $\sigma: \mathcal{A}_{g,\C} \to \mathcal{A}_{g,\C}$ to any $\sigma \in \Aut(\C)$, inducing a map on (\'etale) cohomology.  This is Pontryagin dualized to an action on $H_i(\Sp_{2g}(\Z);\Z/p^n\Z)$, for all $n$, and hence an action on~(\ref{eq:20}) by taking inverse limit.  (Here we used that arithmetic groups have finitely generated homology groups, in order to see that certain derived inverse limits vanish.)

\subsubsection{Definition of symplectic $K$-theory} Next let us outline one definition of symplectic $K$-theory.
 We will do so only with $p$-adic coefficients, and in a way that is adapted to discussing the Galois action; 
a more detailed exposition from a more sophisticated viewpoint is given in \S \ref{sec: symplectic K-theory}. 

 The first step is the insight, due to Sullivan, that there is an operation on spaces (or homotopy
 types) that  carries out $p$-completion at the level of homology.  In particular, there is a $p$-completion map $$B\Sp_{2g}(\Z) \to B\Sp_{2g}(\Z)^\wedge_p$$ inducing an isomorphism in mod $p$ homology and hence mod $p^n$ homology, and whose codomain turns out to be simply connected (at least for $g \geq 3$ where $\Sp_{2g}(\Z)$ is a perfect group).   Moreover, the $\Aut(\C)$
 action that exists on the mod $p^n$ homology of the left hand side
 can be promoted to an actual action  of $\Aut(\C)$ on the space $B\Sp_{2g}(\Z)^{\wedge}_p$. 
 
 Although the space $B\Sp_{2g}(\Z)$
 has no homotopy in degrees $2$ and higher, its $p$-completion {\em does}.
 As with \eqref{eq:20}, these homotopy groups are eventually independent of $g$;  
 the resulting stabilized groups are the ($p$-completed) \emph{symplectic $K$-theory} groups denoted
\begin{equation*}
  \KSp_i(\Z;\Z_p) := \varinjlim_g \pi_i(B\Sp_{2g}(\Z)^\wedge_p)
\end{equation*}
in analogy with the $p$-completed algebraic $K$-theory groups 
 $ K_i(\Z;\Z_p)$,
 which can be similarly computed as  $\mathrm{colim}_g \pi_i(B\GL_g(\Z)^\wedge_p)$.
 
 The action of $\Aut(\C)$ on the space $B\Sp_{2g}(\Z)^\wedge_p$
 now gives an action of $\Aut(\C)$ on $\KSp_i(\Z; \Z_p)$, for which the
 Hurewicz morphism  
\begin{equation} \label{eq:52} \KSp_i(\Z; \Z_p) \rightarrow H_i(\Sp_{\infty}(\Z); \Z_p)\end{equation}
 is equivariant.
 
\begin{remark} Although it is not obvious from the presentation above, these groups $\KSp_i(\Z;\Z_p)$ are in fact the $p$-completions of symplectic $K$-groups $\KSp_i(\Z)$ which are finite generated abelian groups (see \S \ref{sec: symplectic K-theory}).  However,  the $\Aut(\C)$ action 
exists only after $p$-adically completing.  \end{remark}

\subsubsection{} We also recall what is known about the underlying $\Z_p$-modules  (ignoring Galois-action).  These results are deduced from Karoubi's work on \emph{Hermitian $K$-theory} \cite{Karoubi80}, combined with what is now known about algebraic $K$-theory of $\Z$.  The upshot is isomorphisms for $k \geq 1$ and odd primes $p$,
\begin{align*}
  \KSp_{4k-2}(\Z;\Z_p) & \xrightarrow{(c_B,c_H)} K_{4k-2}(\Z;\Z_p) \times \Z_p\\
  \KSp_{4k-1}(\Z;\Z_p) & \xrightarrow{(c_B,c_H)} K_{4k-1}(\Z;\Z_p)
\end{align*}
and vanishing homotopy groups in degrees $\equiv 0,1 \mod 4$.  
Here:
\begin{itemize}
\item[-]
The homomorphism $c_B$ arises from the evident inclusion $\Sp_{2g}(\Z) \subset \GL_{2g}(\Z)$.
\item[-]  
 The homomorphism $c_H$ 
 is obtained as the composite
 $$ \KSp_{4k-2}(\Z; \Z_p) \rightarrow H_{4k-2}(B\Sp_{2g}(\Z); \Q_p) \stackrel{c_H}{\rightarrow} \Q_p.$$
 Here the final map is the Chern character of the $g$-dimensional (Hodge) vector bundle
 arising from  $\Sp_{2g}(\Z) \hookrightarrow \Sp_{2g}(\R) \simeq U(g)$;
 the composite map is valued in $\Z_p$ even though the Chern character involves denominators in general,  reflecting one advantage of homotopy over homology.

\end{itemize}

\subsection{Method of proof and outline of paper} \label{outline}
\label{sec:method-of-proof} 
For the present sketch we consider the reduction of symplectic $K$-theory modulo $q = p^n$;
one recovers the main theorem by passing to a limit over $n$.  

\begin{remark}
Rather than na\"{i}vely reducing homotopy groups modulo $q$, it is better to consider the groups $\KSp_i(\Z;\Z/q)$ which sit in a long exact sequence with the multiplication-by-$q$ map $\KSp_i(\Z;\Z_p) \to \KSp_i(\Z;\Z_p)$.  But that is the same in degree $4k-2$ since $\KSp_{4k-3}(\Z;\Z_p) = 0$.
\end{remark}

 The basic idea for proving the main theorem is to construct enough explicit classes on which one can compute the Galois action. In more detail,   the theory of complex multiplication (CM) permits us
to exhibit a large class of  complex principally polarized abelian varieties with actions of
a cyclic group $C$ with order $q$. If $A \to \Spec(\C)$ comes with such an action then the induced action on $H_1(A(\C)^\mathrm{an};\Z)$, singular homology of the complex points in the analytic topology, gives a homomorphism $C \to \Sp_{2g}(\Z)$. This gives a morphism
$$ \pi^s_i(BC; \Z/q) \longrightarrow \KSp_i(\Z; \Z/q)$$
from the stable homotopy groups of the classifying space $BC$ to symplectic $K$-theory. The left
hand side contains a polynomial ring\footnote{This is an advantage of stable homotopy over homology: the latter (in even degrees) is a divided power algebra.} on a degree $2$ element, the ``Bott element,''
and the image of this ring produces a supply of classes in $\KSp$.   

Let us call temporarily call classes in $\KSp_i(\Z;\Z/q)$ arising from this mechanism \emph{CM classes}.
We shall then show, on the one hand, that CM classes generate all of $\KSp_i(\Z; \Z/q)$.
On the other hand 
the Main Theorem of Complex Multiplication allows us to understand the action of $\Aut(\C)$ on CM classes.
 Taken together, this allows us to compute the $\Aut(\C)$ action on $\KSp_i(\Z; \Z/q)$.  

The contents of the various sections are as follows: 
\begin{itemize}
\item  \S  \ref{sec: alg K-theory}, {\em $K$-theory and its relation to algebraic number theory:}  We review   facts about homotopy groups, Bott elements,  $K$-theory,
 and the relation of $K$-theory and \'{e}tale cohomology.
From the point of view of the main proof, the main output here is
 Proposition \ref{shtuka}, which 
 identifies the transfer map from the $K$-theory of a cyclotomic
 ring to the $K$-theory of $\Z$ in terms of algebraic number theory:  namely, a transfer in the homology
 of corresponding Galois groups.

\item  \S \ref{sec:symplectic-K}, {\em Symplectic $K$-theory:} 
  We review the definition of symplectic $K$-theory, and recall the results of \cite{Karoubi80} 
  which, for odd $p$, lets us describe $\KSp_i(\Z;\Z_p)$ and $\KSp_i(\Z;\Z/q)$ in terms of usual algebraic $K$-theory.  The conclusions we need are summarized in Theorem~\ref{thm:KSp-of-Z}.

\item \S \ref{sec: cm Ab var}, {\em Construction of CM classes in symplectic $K$-theory:}
 The point of \S \ref{sec: cm Ab var}  is to set up the theory of CM in
a slightly unconventional form that allows the CM classes to be easily defined. 
One key output of the section is the sequence
 \eqref{both compositions}:
 it formulates the construction of CM abelian varieties
 as a functor between groupoids.
We also prove  in   Proposition  \ref{construction of CM structures} a technical result about the existence of ``enough'' CM abelian varieties
associated to cyclotomic fields.

\item  \S \ref{CMExhaust},  {\em CM classes exhaust all of symplectic $K$-theory:} 
We give the construction of CM classes   and prove that
that they exhaust all of symplectic $K$-theory
(see in particular Proposition \ref{prop:generators-for-KSp}). 
To prove the exhaustion one  must check both that 
$\KSp$ is not too large and that 
there are enough CM classes. These come, respectively, from
  the previously mentioned Proposition \ref{shtuka}
and  Proposition \ref{construction of CM structures}.

\item \S \ref{sec: galois action} {\em Computation of the action of $\Aut(\C)$ on CM classes:}
The action of $\Aut(\C)$ on CM classes can be deduced from  
the ``Main Theorem of CM,''  which computes how  $\Aut(\C)$ acts on moduli of CM abelian varieties.
(In its original form this is due to Shimura and Taniyama; we use the refined form  
due to Langlands, Tate, and Deligne.)
 We recall this theorem, in a language adapted to our proof,  in \S \ref{sec:main-thm}.

 \item \S \ref{maintheoremproof}, {\em Proof of the main theorem (Theorem \ref{mt2}).}   The results of the previous sections
 have already entirely computed the Galois action. More precisely,
 they allow one to explicitly give a cocycle 
 that describes the extension class of  \eqref{kspq}. In \S \ref{sec:universality}
 we explicitly compare this cocycle to one that describes the universal extension
 and show they are equal.  
 
 The remainder of \S \ref{maintheoremproof} describes
 variants on the universal property (e.g.\ passing between $\Z/q$ and $\Z_p$ coefficients,
 or a version for Bott-inverted symplectic $K$-theory which
 also sees extensions of {\em negative} Tate twists). 

\item \S \ref{sec: stable homology}, {\em Consequences in homology.}  The stable homology $H_i(\Sp_\infty(\Z);\Z_p)$ naturally surjects onto $\KSp_i(\Z;\Z_p)$, at least for $i \leq 2p-2$.  In this short section we use this to deduce divisibility of certain characteristic numbers of families of abelian varieties defined over $\Q$.
  
\end{itemize}

\begin{remark}\label{rem: intro K(1) local}
  Let us comment on the extent to which our result depends on the \emph{norm residue theorem}, proved by Voevodsky and Rost.    The $p$-completed homotopy groups $\KSp_*(\Z;\Z_p)$ in our main theorem may be replaced by groups we denote $\KSp^{(\beta)}_*(\Z;\Z_p)$ and call ``Bott inverted symplectic $K$-theory'' see Subsection~\ref{subsec:main-theorem-KSPBott}.  They agree with $\pi_*(L_{K(1)} \KSp(\Z))$, the so-called \emph{$K(1)$-local homotopy groups}.

  The norm residue theorem can be used to deduce that the canonical map $\KSp_i(\Z;\Z_p) \to \KSp_i^{(\beta)}(\Z;\Z_p)$ is an isomorphism for all $i \geq 2$.  Independently of the norm residue theorem, the main theorem stated above may be proved with $\KSp_{4k-2}^{(\beta)}(\Z;\Z_p)$ in place of $\KSp_i(\Z;\Z_p)$.  Besides the simplification of the proof, this has the advantage of giving universal extensions of $\Z_p(2k-1)$ for all $k \in \Z$, including non-positive integers.

  In our presentation we have chosen to work mostly with $\KSp_*(\Z;\Z_p)$ instead of $\KSp_*^{(\beta)}(\Z;\Z_p)$, for reasons of familiarity.  A more puritanical approach would have compared $\KSp_*(\Z;\Z_p)$ and $\KSp^{(\beta)}_*(\Z;\Z_p)$ at the very end, and this would have been the only application of the norm residue theorem.
\end{remark}

\subsection{Notation}

For $q$ any {\em odd} prime power,  
we denote:
\begin{itemize}
\item   $\OO_q$ the cyclotomic ring $\Z[e^{2 \pi i/q}]$ obtained by adjoining a primitive $q$th root of unity to $\Z$,
and $K_q = \OO_q \otimes \Q$ its quotient field. For us we shall always regard these as subfields of $\C$.
We denote by $\zeta_q \in \OO_q$ the primitive $q$th root of unit $e^{2 \pi i /q}$. 

\item $\Z' := \Z[\frac{1}{p}]$, and $\OO_q' := \OO_q[\frac{1}{p}]$.

\item We denote by $H_q$ the largest algebraic unramified extension of $K_q$ inside $\C$
whose Galois group is abelian of $p$-power order. Thus $H_q$
is a subfield of the Hilbert class field, and its Galois group is isomorphic
to the $p$-power torsion inside the class group of $\OO_q$.  

\item For a ring $R$, we denote by $\mathrm{Pic}(R)$ the groupoid
of locally free rank one $R$-modules, and by $\pi_0 \mathrm{Pic}(R)$
the group of isomorphism classes, i.e.\ the class group of $R$.  In particular, the
class group of $\OO_q$ is denoted $\pi_0 \Pic(\OO_q)$. 

\item  
There are ``Hermitian'' variants of the Picard groupoid that will play a crucial role for us. 
For the ring of integers $\mathcal{O}_E$ in a number field $E$,  $\mathcal{P}_E^{+}$
will denote the groupoid of  rank one locally free $\mathcal{O}_E$-modules
endowed with a $\mathcal{O}_E$-valued Hermitian form, 
and $\mathcal{P}_E^{-}$ will denote the groupoid of rank one locally free 
$\mathcal{O}_E$-modules endowed with a skew-Hermitian form valued in the inverse different. 
See   \S \ref{picgroup} for details of these definitions.
 \end{itemize} 

We emphasize that $q$ is assumed to be {\em odd}. Many of our statements remain valid for $q$ a power of $2$, and we attempt
to make arguments that remain valid in that setting, but for simplicity we prefer to impose $q$ odd as a standing assumption.  
 
\subsection*{Acknowledgements}
\label{sec:acknowledgements}

TF was supported by an NSF Graduate Fellowship, a Stanford ARCS Fellowship, and an NSF Postdoctoral Fellowship under grant No. 1902927. SG was supported by the European Research Council (ERC) under the European Union's Horizon 2020 research and innovation programme (grant agreement No.\ 682922), by the EliteForsk Prize, and by the Danish National Research Foundation (DNRF92 and DNRF151).  AV was supported by an NSF DMS grant as well as a Simons investigator grant. 

SG thanks Andrew Blumberg and Christian Haesemeyer for helpful discussions about $K$-theory and \'etale $K$-theory.  All three of us thank the Stanford mathematics department for providing a wonderful working environment. 


\section{Recollections on algebraic $K$-theory}\label{sec: alg K-theory}

This section reviews algebraic $K$-theory and its relation with   \'{e}tale cohomology.
Since it is somewhat lengthy we briefly outline the various subsections:
\begin{itemize}
\item \S \ref{homotopy recollections} is concerned with summarizing facts about stable
and mod $q$ homotopy groups; in particular we introduce the Bott element
in the stable homotopy of a cyclic group. 
\item  After a brief discussion of infinite loop space machines in \S \ref{sec:infinite-loop-space},
we review algebraic $K$-theory in \S \ref{sec:algebraic-k-theory-1}. 
In \S \ref{sec:picard groupoid}
 we discuss the Picard group and Picard groupoid, which are used to analyze a simple piece of algebraic $K$-theory.
 
 \item    
 A fundamental theorem of Thomason asserts that algebraic $K$-theory satisfies {\'e}tale descent
 after inverting a Bott element (defined in mod $q$ algebraic $K$-theory in \S \ref{sec:bott-element}). We review this theorem and its consequences in  \S \ref{sec:bott-inverted-k}.

\item  \S \ref{Bott Etale} uses Thomason's results to compute Bott-inverted $K$-theory of $\Z$ and of $\mathcal{O}_q$
in terms of {\'e}tale (equivalently, Galois) cohomology.

\item Finally, in Proposition  \ref{shtuka} we rewrite some of the results of \S \ref{Bott Etale} in terms of homology of Galois groups, which 
is most appropriate for our later applications.  Specifically, the Proposition
 identifies the transfer map from the $K$-theory of a cyclotomic
 ring to the $K$-theory of $\Z$ in a corresponding transfer in the  group homology.
  \end{itemize}

 We do not claim any substantial original results in this section. Most of the statements in this section follow quickly from work of Thomason and Voevodsky, but do not appear in the literature in the form written here, so we take the opportunity to spell them out.
  

\subsection{Recollections on stable  and mod $q$ homotopy} \label{homotopy recollections}
Recall that, for a topological space $Y$, the notation $Y_+$ means the space $Y \coprod \{*\}$ consisting of $Y$ together with a disjoint basepoint.
Each space gives rise to a spectrum $\Sigma^\infty_+ Y$, namely the suspension spectrum on $Y_+$, and consequently we can
freely specialize constructions for spectra to those for spaces. 
 In particular, the \emph{stable homotopy groups} of $Y$ are, by definition, the homotopy groups of the associated spectrum:   $$\pi_k^s(Y) = \pi_k(\Sigma^\infty_+ Y) := \varinjlim_{n} [S^{k+n}, \Sigma^n Y_+],$$
 where $[-,-]$ denotes homotopy classes of based maps. 
We emphasize that $\pi_*^s$ is defined for an {\em unpointed} space $Y$.
(In some references, it is defined for a based space, and in those references the definition does not involve an added disjoint basepoint). 

\begin{remark} One could regard $\pi^s_*(Y)$  as being the ``homology of $Y$ with coefficients in the sphere spectrum,'' and   
it enjoys the properties of any generalized homology theory. 
There is a Hurewicz map $\pi^s_*(Y) \rightarrow H_*(Y)$, which is an isomorphism in degree $* = 0$ and a surjection in degree 1.
\end{remark}

We will be interested in the corresponding notion with $\Z/q$ coefficients. 
 For any spectrum $E$, the homomorphisms $\pi_i(E) \to \pi_i(E)$ which multiply by $q \in \Z$ fit into a long exact sequence
\begin{equation*}
  \dots \to   \pi_i(E) \xrightarrow{q} \pi_i(E) \to \pi_i(E \wedge (\mathbb{S}/q)) \to \pi_{i-1}(E) \xrightarrow{q} \pi_i(E) \to \dots,
\end{equation*}
where the spectrum $\mathbb{S}/q$ is the mapping cone of a degree-$q$ self map of the sphere spectrum.  For $q > 0$ we write $$\pi_i(E;\Z/q) := \pi_i(E \wedge (\mathbb{S}/q))$$ for these groups, the \emph{homotopy groups of $E$ with coefficients in $\Z/q$}. 
   Correspondingly, we get stable homotopy groups $\pi^s_*(Y; \Z/q)$ for a space $Y$.
These have
the usual properties of a homology theory.

\subsubsection{The Bott element in the stable homotopy of a cyclic group} \label{Bottcyclic}  The stable homotopy of the classifying space of a cyclic group contains a polynomial algebra on a certain ``Bott element'' in degree $2$.  This  
will be a crucial tool in our later arguments, and we review it now.

We recall (see \cite{MR760188}) that for $q = p^n > 4$ the spectrum $\mathbb{S}/q$ has a product which is unital, associative, and commutative up to homotopy.  It makes $\pi_*(E;\Z/q)$ into a graded ring when $E$ is a ring spectrum, graded commutative ring when the product on $E$ is homotopy commutative.  In the rest of this section we shall tacitly assume $q > 4$ in order to have such ring structures available. (In fact everything works also in the remaining case $p = q = 3$ with only minor notational updates: see Remark \ref{remark:p=q=3}.)

For the current subsection \S \ref{Bottcyclic} set $Y := B(\Z/q)$, the classifying space of a cyclic group of order $q$. 
This $Y$ has the structure of $H$-space, in fact a topological abelian group, and correspondingly $\pi^s_*(Y)$ 
has  the structure of a graded commutative ring. 

Recall that $q$ is supposed odd. Then there is a unique element (the ``Bott class'')   $\beta \in \pi_2^s(Y; \Z/q)$ 
such that, in the diagram
 \begin{equation}\label{eq: diag defining Bott}
\begin{tikzcd}
& \beta \in \pi_2^s(Y; \Z/q) \ar[r] \ar[d, "\mathrm{Hur}", "\sim"'] & \pi_1^s(Y; \Z) \ar[d]  \\
0 \ar[r] & H_2(Y; \Z/q) \ar[r]  & H_1(Y; \Z)[q] \ar[r, equals] &  \Z/q
\end{tikzcd}
\end{equation} 
the image of $\beta$ in the bottom right $\Z/q$ is the canonical generator $1$ of $\Z/q$.  In fact, the  
Hurewicz map $\mathrm{Hur}$ above is an isomorphism.

\begin{remark}
The diagram \eqref{eq: diag defining Bott} exists for all $q$, but for $q$ a power of $2$ the map $\pi_2^s(Y; \Z/q) \rightarrow H_2(Y; \Z/q)$ is not an isomorphism. There is a class in $\pi_2^s(Y; \Z/q)$ fitting into \eqref{eq: diag defining Bott} when $q$ is a power of $2$, but the diagram above does not characterize it. To pin down the correct $\beta$ in that case, note that the map $S^1 \rightarrow Y$ inducing $\Z \rightarrow \Z/q$ on $\pi_1$
extends to $M(\Z/q, 1)=  S^1 \cup_q D^2$, the  pointed mapping cone of the canonical degree $q$ map $S^1 \to S^1$,
and one can construct $\beta$ starting from the identification of  $\pi_k^s(Y;\Z/q) = \varinjlim_{n} [\Sigma^{n} M(\Z/q, k-1), \Sigma^n Y_+]$.
\end{remark}

 \begin{lemma} \label{split injection}
 The induced map $\Z/q[\beta] \rightarrow \pi_*^s(Y; \Z/q)$
 is a split injection of graded rings. 
 \end{lemma} 

 \proof  
The map $a \to e^{2 \pi i a/q}$ is a homomorphism from $\Z/q$ to $S^1$ and it gives rise to a line bundle $\underline{L}$ on $Y$.
In turn this  
 induces a map  from the suspension spectrum of $Y_+$
 to the spectrum $ku$ 
 representing topological $K$-theory, and thereby induces 
 on homotopy groups a map $$\pi_*^s(Y; \Z/q) \rightarrow \pi_{*}(ku 
 ; \Z/q).$$
 This map is in fact a ring map. 
 
 We claim that the class $\beta$ is sent to the reduction of the usual Bott class $\mathrm{Bott} \in \pi_2(ku 
 )$;  
 this implies that $\Z/q[\beta] \rightarrow \pi_*^s(Y; \Z/q)$ is indeed split injective, because
 $\pi_*(ku 
; \Z/q)$ is, in non-negative degrees, a polynomial algebra on this reduction. 
 The Bott class in $\pi_2(ku 
;\Z) \simeq \pi_2(\mathrm{BU}, \Z)$ is characterized  (at least up to sign,
 depending on normalizations)
 by having pairing $1$ with the  first Chern class of the line bundle $\underline{L}$ arising from $\det: U \rightarrow S^1$. 
 It sufficies then to show that
 $$ \langle \bar{\beta}, c_1(\underline{L}) \rangle = 1 \in \Z/q,$$
 where $\bar{\beta} \in  H_2(Y; \Z/q)$ is the image of $\beta$ in 
 by the Hurewicz map,
 and $c_1(\underline{L})$ is the first Chern class of $\underline{L}$ considered
 as a line bundle on $Y$. 
 
  This Chern class is the image of   $j \in H^1(Y; \R/\Z)$ by the connecting homomorphism $H^1(Y, \R/\Z) \rightarrow H^2(Y, \Z)$
arising from the map $\Z \rightarrow \R \stackrel{e^{2 \pi i x}}{\rightarrow} \R/\Z$. 
Therefore $c_1(\underline{L}) \in H^2(Y, \Z)$ is 
obtained from the tautological class $\tau \in  H^1(Y, \Z/q)$
by the connecting map $\delta$ associated to $\Z \rightarrow \Z \rightarrow \Z/q$,  
and the reduction of $c_1(\underline{L})$ modulo $q$ is simply the Bockstein of $\tau$. 
Therefore, the pairing of $c_1(\underline{L})$ with $\overline{\beta}$ 
 is the same as the pairing of $\tau$ with the Bockstein of $\overline{\beta}$; 
 this last pairing is $1$, by definition of $\beta$. 
     \qed 
 
\begin{remark}\label{rem: cyc action on bott} The reasoning of the proof also shows the following:
had we replaced  the morphism $j: \Z/q \rightarrow S^1$ 
by $j^a$ (for some $a \in \Z$), 
then  
  the corresponding element in $\pi_2(ku; \Z/q)$ is also multiplied by $a$. 
\end{remark}
  
\subsection{Infinite loop space machines}
\label{sec:infinite-loop-space}

Recall that associated to a small category $\Cal{C}$, there is a \emph{classifying space} $|\Cal{C}|$, which is the geometric realization of the nerve of $\Cal{C}$ (a simplicial set).  In particular $\pi_0(|\mathcal{C}|)$ is the set of isomorphism classes.
A symmetric monoidal structure on $\mathcal{C}$ induces in particular a ``product'' $|\mathcal{C}| \times |\mathcal{C}| \to |\mathcal{C}|$ which is associative and commutative up to homotopy.  The theory of \emph{infinite loop space machines} associates to the symmetric monoidal category $\mathcal{C}$ a \emph{spectrum} $K(\mathcal{C})$ and a map
\begin{equation}\label{eq:54}
  |\mathcal{C}| \to \Omega^\infty K(\mathcal{C}).
\end{equation}
Up to homotopy this map preserves products, and the induced monoid homomorphism $\pi_0(|\mathcal{C}|) \to \pi_0(\Omega^\infty K(\mathcal{C}))$ is the universal homomorphism to a group, namely, the ``Grothendieck group'' of the monoid.  The map~\eqref{eq:54} can be viewed as a derived version of the universal homomorphism from a given monoid to a group.


\subsection{Algebraic $K$-theory: definitions} 
\label{sec:algebraic-k-theory-1}

   For a ring $R$, let $\proj{R}$ denote the symmetric monoidal groupoid whose objects are finitely generated projective $R$-modules, morphisms are $R$-linear isomorphisms, and with the Cartesian symmetric monoidal structure (i.e.,\ direct sum of $R$-modules).  The set $\pi_0(\proj{R})$ of isomorphism classes in $\proj{R}$ then inherits a commutative monoid structure.
  Write $|\proj{R}|$ for the associated topological space (i.e.\ geometric realization of the nerve of $\proj{R}$).  Direct sum of projective $R$-modules is a symmetric monoidal structure on $\proj{R}$ and induces a map $\oplus: |\proj{R}| \times |\proj{R}| \to |\proj{R}|$.

  As recalled above, 
  there is a canonically associated spectrum $K(R) := K(\Cal{P}(R))$ and a ``group completion'' map
  \begin{equation*}
    |\proj{R}| \to \Omega^\infty K(R).
  \end{equation*}
  The algebraic $K$-groups of $R$ are defined as the homotopy groups of $K(R)$. Alternately, for $i=0$,  it is the projective class group $K_0(R)$ while for $i > 0$ it may be defined as
\begin{equation*}
  K_i(R) := \pi_i B\GL_\infty(R)^+,
\end{equation*}
the homotopy groups of the Quillen plus construction applied to the commutator subgroup of $\GL_\infty(R) = \varinjlim_n \GL_n(R)$.\footnote{
The group completion theorem can be used to induce a comparison between $K_0(R) \times B\GL_\infty(R)^+$ and $\Omega^\infty K(R)$, roughly speaking by taking direct limit over applying $[R] \oplus -: |\proj{R}| \to |\proj{R}|$ infinitely many times and factoring over the plus construction. }

When $R$ is commutative, we also have product maps
\begin{equation*}
  K_i(R) \otimes K_j(R)  \to K_{i + j}(R),
\end{equation*}
induced from tensor product of $R$-modules,
making $K_*(R)$ into a graded commutative ring. 

\begin{definition}
 We define the \emph{mod $q$ algebraic $K$-theory groups of $R$} to be 
\[
K_i(R; \Z/q) := \pi_i(K(R); \Z/q).
\]
In the case $R=\Z$ we define the \emph{$p$-adic algebraic $K$-theory groups} via
\[
 K_i(\Z;\Z_p) := \varprojlim_n K_i(\Z;\Z/p^n).
\] 
(This is the correct definition because of finiteness properties of $K_*(\Z; \Z/p^n)$; in general, we should work with ``derived inverse limits.'') 
\end{definition}

\subsubsection{Adams operations} Finally, let us recall that (again for $R$ commutative, as shall be the case in this paper) there are Adams operations $\psi^k: K_i(R) \to K_i(R)$ for $k \in \Z$ satisfying the usual formulae.  We shall make particular use of $\psi^{-1}$, which in the above model is induced by the functor $\proj{R} \to \proj{R}$ sending a module $M$ to its dual $D(M) := \Hom_R(M,R)$ and an isomorphism $f: M \to M'$ to the inverse of its dual $D(f): D(M') \to D(M)$.

\subsection{Picard groupoids} \label{sec:picard groupoid}

We now define certain spaces which can be understood explicitly and used to probe algebraic $K$-theory. They are built out of categories that we call Picard groupoids. 
 
\begin{definition}\label{def:pic} (The Picard groupoid.) 
  For a commutative ring $R$, let $\pic{R} \subset \proj{R}$ be the subgroupoid whose objects are the rank 1 projective modules, with the symmetric monoidal structure given by $\otimes_R$.

The associated space $|\pic{R}|$ inherits a group-like product 
  \[
  \otimes_R: |\pic{R}| \times |\pic{R}| \to |\pic{R}|,
  \]
and there are canonical isomorphisms of abelian groups $\pi_0(|\pic{R}|) = H^1(\Spec(R);\mathbf{G}_m)$ (the classical Picard group) and $\pi_1(|\pic{R}|,x) = H^0(\Spec(R);\mathbf{G}_m) = R^\times$ for any object $x \in \pic{R}$.  The higher homotopy groups are trivial.
\end{definition}

When $R$ is a ring of integers, $\pic{R}$ is equivalent to the groupoid whose objects are the invertible fractional ideals $I \subset \mathrm{Frac}(R)$ and whose set of morphisms $I \to I'$ is $\{x \in R^\times \mid xI = I'\}$.

The tensor product of rank 1 projective modules gives a product on the space $|\pic{R}|$ and makes the stable homotopy groups $\pi^s_*(|\pic{R}|)$ into a graded-commutative ring.  We have a canonical ring isomorphism $\Z[\pi_0(\pic{R})] \to \pi_0^s(|\pic{R}|)$ from the group ring of the abelian group $\pi_0(\pic{R}) \cong H^1(\Spec(R);\G_m)$.
The fact that stable homotopy (being a homology theory) takes disjoint union to direct sum implies that the product map
\begin{equation}\label{eq:8}
  \pi_*^s(BR^\times) \otimes \Z[\pi_0(\pic{R})] \xrightarrow{\cong} \pi_*^s(|\pic{R}|).
\end{equation}
is an isomorphism. 

The inclusion functor induces maps $|\pic{R}| \to |\proj{R}| \to \Omega^\infty K(R)$ preserving $\otimes_R$, at least up to coherent homotopies.  The adjoint map $\Sigma^\infty_+|\pic{R}| \to K(R)$ is then a map of ring spectra, and we get a ring homomorphism
\begin{equation}\label{eq:26}
  \pi_*^s(BR^\times) \otimes \Z[\pi_0(\pic{R})] \xrightarrow{\cong} \pi_*^s(|\pic{R}|) \to K_*(R).
\end{equation}

\subsection{Bott elements in $K$-theory with mod $q$ coefficients}
\label{sec:bott-element}

\begin{definition}
The \emph{algebraic $K$-theory of $R$ with mod $q$ coefficients} is defined as $K_i(R;\Z/q) := \pi_i(K(R);\Z/q)$.
\end{definition}
As discussed earlier, $K_*(R; \Z/q)$ has the structure of a graded-commutative ring  for $q = p^n > 4$.

Let us next recall the construction of a canonical \emph{Bott element} in $K_2(R;\Z/q)$ associated to a choice of primitive $q$th root of unity $\zeta_q \in R^\times$.
The choice of $\zeta_q$  induces a homomorphism $\Z/q \to \GL_1(R)$.  Regarding $\GL_1(R)$ as the automorphism group of the object $R \in \pic{R}$ gives a map $B(\Z/q) \to |\pic{R}|$. 
Now we previously produced a ``Bott element'' $\beta \in \pi_2^s(B (\Z/q))$; 
under the maps~(\ref{eq:26}) we have  
\begin{equation*}
  \beta \in \pi_2^s(B(\Z/q)) \to \pi_2^s(|\pic{R}|;\Z/q) \to K_2(R;\Z/q).
\end{equation*}
The image is the \emph{Bott element} and shall also be denoted $\beta \in K_2(R;\Z/q)$.
More intrinsically,   this discussion gives a homomorphism 
\begin{equation} \label{Bottbeta} \beta: \mu_q(R) \to  K_2(R;\Z/q)
\end{equation}
 which is independent of any choices;
since our eventual application is to subrings of $\C$ where we will take $\zeta = e^{2 \pi i/q}$,
we will not use this more intrinsic formulation. 

\subsection{Bott inverted $K$-theory and Thomason's theorem}
\label{sec:bott-inverted-k}

The element $\beta \in K_2(R;\Z/q)$ may be inverted in the ring structure (when $q > 8$), leading to a 2-periodic $\Z$-graded ring $K_*(R;\Z/q)[\beta^{-1}]$ called \emph{Bott inverted} $K$-theory of $R$, when $R$ contains a primitive $q$th root of unity.  As explained in \cite[Appendix A]{Thomason85} we can still make sense of this functor when $R$ does not contain primitive $q$th roots of unity: the power $\beta^{p-1} \in K_{2p-2}(\Z[\mu_p];\Z/p)$ comes from a canonical element in $K_{2p-2}(\Z;\Z/p)$, also denoted $\beta^{p-1}$ (even though it is not the $(p-1)$st power of any element of $K_*(\Z;\Z/p)$), whose  $p^{n-1}$st power lifts to an element of $K_{2p^{n-1}(p-1)}(\Z;\Z/p^n)$. Inverting the image of these elements gives a functor
\begin{equation*}
  X \mapsto K_*(X;\Z/q)[\beta^{-1}]
\end{equation*}
from schemes to $\Z$-graded $\Z/q$-modules (graded commutative $(\Z/q)$-algebras when $q > 4$), where $q = p^n$ as before.  For typographical ease, we will denote this via $\KBott$: 
$$ \KBott_*(X;\Z/q) =  K_*(X;\Z/q)[\beta^{-1}].$$ In the case $X =\Spec \Z$, we also define the \emph{$p$-adic Bott-inverted  $K$-theory} groups 
\[
 \KBott_*(\Z;\Z_p) := \varprojlim_n \KBott_*(\Z;\Z/p^n).
\] 
 
\begin{remark}\label{remark:telescope}   
As also recalled in \cite[Appendix A]{Thomason85} this may be implemented on the spectrum level as follows: Adams constructed spectrum maps $\Sigma^m (\mathbb{S}/p^n) \to (\mathbb{S}/p^n)$ for $m = 2p^{n-1}(p-1)$ when $p$ is odd, with the property that it induces isomorphisms $\Z/q = \pi_0(ku;Z/q) \to \pi_m(ku;\Z/q) = \Z/q$, where $ku$ is the topological $K$-theory spectrum, and we can let $T$ be the homotopy colimit of the infinite iteration $\mathbb{S}/q \to \Sigma^{-m} (\mathbb{S}/q) \to \Sigma^{-2m}(\mathbb{S}/q) \to \dots$.  Then $\KBott_*(X;\Z/q)$ is canonically the homotopy groups of the spectrum $K(X) \wedge T$. We will on occasion denote this spectrum as $\KBott(X; \Z/q)$. 
\end{remark}

\subsubsection{\'{E}tale descent and Thomason's spectral sequence}

The main result of \cite{Thomason85} is an \'etale descent property for the Bott inverted $K$-theory functor.  (Because of this, Bott-inverted $K$-theory is essentially the same as Dwyer--Friedlander's ``\'etale $K$-theory'' \cite{MR805962}, at least in positive degrees.  See also \cite{clausen2019hyperdescent} for a recent perspective.)

For a scheme $X$ over $\Spec(\Z[1/p])$, Thomason constructs a convergent spectral sequence
\begin{equation} \label{TSS}
  E^2_{s,t} = H^{-s}_\mathrm{et}(X;\mu_q^{\otimes (t/2)}) \Rightarrow \KBott_{t+s}(X;\Z/q),
\end{equation}
concentrated in degrees $s \in \Z_{\leq 0}$ and $t \in 2\Z$.  (Existence and convergence of the spectral sequence requires mild hypotheses on $X$, satisfied in any case we need.)  The spectral sequence arises as a hyperdescent spectral sequence for $\KBott$, regarded as a sheaf of spectra on the \'etale site of $X$.  

Since the Adams operations $\psi^a$ act on $\KBott$ through maps of sheaves of spectra when $a \not \equiv 0 \mod p$, there are compatible actions of Adams operations on the spectral sequence.  The operation $\psi^a$ acts by multiplication by $a^{t/2}$ on $E^2_{s,t}$ and in particular $\psi^{-1}$ acts as $+1$ on the rows with $t/2$ even and as $-1$ on the rows where $t/2$ is odd.

\subsubsection{Comparison with algebraic $K$-theory}
This {\'e}tale descent property makes Bott-inverted $K$-theory amenable to computation. 
 On the other hand, it is a well known consequence of the \emph{norm residue theorem} (due to Voevodsky and Rost) that when $X$ is a scheme over $\Spec(\Z[1/p])$ satisfying a mild hypothesis, the localization homomorphism $K_*(X;\Z/q) \to K_*(X;\Z/q)[\beta^{-1}]$ is an isomorphism in sufficiently high degrees.
 We briefly spell out how this comparison between $K$-theory and Bott inverted $K$-theory follows from the norm residue theorem  (see \cite{MR3931681} for a textbook account of the latter)
 in the cases of interest:
\begin{proposition} \label{normresidueconsequence}
   For $X = \Spec(\Z')$ or $X = \Spec(\mathcal{O}'_q)$, the localization map
  \begin{equation*}
    K_i(X;\Z/q) \to \KBott_i(X;\Z/q) 
  \end{equation*}
  is an isomorphism for all $i  > 0$ and a monomorphism for $i=0$.  (It is in fact also an isomorphism for $i=0$,
  as will be proved in \S \ref{Bott Etale}). The same assertion
  holds for $\Spec(\Z)$ or $\Spec(\mathcal{O}_q)$ if we suppose $i \geq 2$.
\end{proposition}
 
\begin{proof}[Proof sketch]  
  For any field $k$ of finite cohomological dimension (and admitting a ``Tate-Tsen filtration'', as in \cite[Theorem 2.43]{Thomason85}), there are spectral sequences converging to both domain and codomain of the map $K_*(k;\Z/p) \to \KBott(k;\Z/p)$.  In the codomain it is the above-mentioned spectral sequence of Thomason, applied to $X = \Spec(k)$, and in the domain it is the motivic spectral sequence.  There is a compatible map of spectral sequences, which on the $E^2$ page is the map from motivic to \'etale cohomology
  \begin{equation*}
    H^{-s}_\mathrm{mot}(\Spec(k);(\Z/p)(t/2)) \to H^{-s}_\mathrm{et}(\Spec(k);\mu_p^{\otimes t/2}).
  \end{equation*}
  induced by changing topology from the Nisnevich to \'etale topology.  The norm residue theorem implies that this map is an isomorphism for $t/2 \geq -s$.  Below this line the motivic cohomology vanishes but the \'etale cohomology need not.  If $\mathrm{cd}_p(k) = d$ we may therefore have non-trivial \'etale cohomology in $E^2_{-d,2d-2}$ which is not hit from motivic cohomology, and the total degree $d-2$ of such elements is the highest possible total degree in which this can happen.  By convergence of the spectral sequences, the  map    $$ K_i(k; \Z/p) \rightarrow \KBott_i(k; \Z/p)$$
  is an isomorphism for $i \geq d-1$ and an injection for $i=d-2$; 
  the same conclusion follow with $\Z/q$ coefficients by induction using the long exact sequences. 
  
   This applies to $k = \Q$ which has $p$-cohomological dimension 2 (we use here that $p$ is odd) and $k = \F_\ell$ which has $p$-cohomological dimension 1 for $\ell \neq p$, as well as finite extensions thereof.  Finally,  Quillen's localization sequence
  \begin{equation*}
    \bigvee_{\ell \neq p} K(\F_\ell) \to K(\Z') \to K(\Q)
  \end{equation*}
  and its Bott-inverted version imply that $K_i(\Z';\Z/q) \to \KBott_i(\Z';\Z/q)$ is an isomorphism for $i \geq 1$ and a monomorphism for $i = 0$, and a similar argument applies when $X = \Spec(\mathcal{O}'_q)$.  

  The final assertion results from  using Quillen's localization sequence to compare $\Z$ and $\Z'$, plus Quillen's computation of the $K$-theory of finite fields \cite{Qui72}. For reference we state this as Lemma \ref{Quillenexact}, and expand on the proof below.
  \end{proof}

\begin{lemma} \label{Quillenexact}
  The map $\Z \rightarrow \Z'$ induces an isomorphism on mod $q$
  $K$-theory in all degrees except 1, where $K_1(\Z'; \Z/q) \cong \Z/q \oplus K_1(\Z; \Z/q)$.    The same assertion holds true for $\mathcal{O}_q \to \mathcal{O}'_q$.
In particular, the maps $\KBott_*(\Z;\Z/q)) \to \KBott_*(\Z';\Z/q)$ and $\KBott_*(\mathcal{O}_q) \to \KBott_*(\mathcal{O}_q';\Z/q)$ are both isomorphisms in all degrees. 
 \end{lemma}
 
 \begin{proof}
  Quillen's devissage and localization theorems \cite[Section 5]{MR0338129} gives fiber sequences
\begin{align*}
  &K(\F_p) \to K(\Z) \to K(\Z')\\
  &K(\F_p) \to K(\mathcal{O}_q) \to K(\mathcal{O}_q').
\end{align*}
His calculation \cite{Qui72} of $K$-theory of finite fields implies $K_i(\F_p;\Z/q) = 0$ for $i \neq 0$, while $K_0(\F_p;\Z/q) = \Z/q$.  Finall we note the homomorphisms $K_0(\Z) \to K_0(\Z')$ and $K_0(\mathcal{O}_q) \to K_0(\mathcal{O}_q')$ are injective -- the latter because the prime above $p$ in $\mathcal{O}_q$ is principal. 
\end{proof}

\subsection{Some computations of Bott-inverted $K$-theory in terms of {\'e}tale cohomology} \label{Bott Etale}
 
In this section, we shall use Thomason's spectral sequence \eqref{TSS} 
\begin{equation*} 
  E^2_{s,t} = H^{-s}_\mathrm{et}(X;\mu_q^{\otimes (t/2)}) \Rightarrow \KBott_{t+s}(X;\Z/q),
\end{equation*}
 to compute Bott-inverted $K$-theory of number rings in terms of {\'e}tale cohomology. 
By Proposition~\ref{normresidueconsequence}, many of the results can be directly
stated in terms of $K$-theory. 
Through this use of Proposition \ref{normresidueconsequence},
 our main result -- in the form
 stated in the introduction --  depends on the norm residue theorem; but that dependence is easily avoided by replacing $\KSp_{4k-2}(\Z;\Z_p)$ by 
its Bott-inverted version, see Subsection \ref{subsec:main-theorem-KSPBott}.

We recall that we work under the standing assumption that $q$ is odd.  

\begin{lemma} \label{KOq}
 We have the following isomorphisms, for all $k \in \Z$: 
   \begin{align*}
    K^{(\beta)}_{4k-2}(\mathcal{O}'_q;\Z/q)^{(+)} & \cong H^2(\Spec(\mathcal{O}_q');\mu_q^{\otimes 2k})\\
    K^{(\beta)}_{4k-2}(\mathcal{O}'_q;\Z/q)^{(-)} & \cong H^0(\Spec(\mathcal{O}_q');\mu_q^{\otimes (2k-1)})\\
    K^{(\beta)}_{4k}(\mathcal{O}'_q;\Z/q)^{(+)} & \cong H^0(\Spec(\mathcal{O}_q');\mu_q^{\otimes 2k})\\
    K^{(\beta)}_{4k}(\mathcal{O}'_q; \Z/q)^{(-)} & \cong H^2(\Spec(\mathcal{O}_q');\mu_q^{\otimes 2k+1}).
  \end{align*}

In odd degrees we have an isomorphism
  \begin{equation}\label{eq:75}
    \KBott_{2k-1}(\mathcal{O}'_q;\Z/q) \cong H^1_\mathrm{et}(\Spec(\mathcal{O}'_q);\mu_q^{\otimes k})
  \end{equation}
  and $\psi^{-1}$ acts by $(-1)^k$. 
  
  Finally, the map 
   $K_i(\mathcal{O}'_q;\Z/q) \to \KBott_i(\mathcal{O}'_q;\Z/q) $
is an isomorphism  for all $i \geq 0$,
and the map 
   $K_i(\mathcal{O}_q;\Z/q) \to \KBott_i(\mathcal{O}_q;\Z/q) $
is an isomorphism for $i=0$ or $i \geq 2$.  
 \end{lemma}
\begin{proof} 
We apply  \eqref{TSS} to  $X = \Spec(\mathcal{O}'_q)$.  This scheme has \'etale cohomological dimension 2, so the spectral sequence is further concentrated in the region $-2 \leq s \leq 0$.  The spectral sequence must collapse for degree reasons, since no differential goes between two non-zero groups
  (since only $t \in 2\Z$ appears).   Convergence of the spectral sequence gives in odd degrees 
  \eqref{eq:75}.
  
In even degrees we obtain  a short exact sequence
  \begin{equation*}
    0 \to H^2_\mathrm{et}(\Spec(\mathcal{O}'_q);\mu_q^{\otimes k}) \to \KBott_{2k-2}(\mathcal{O}'_q;\Z/q) \to H^0_\mathrm{et}(\Spec(\mathcal{O}'_q);\mu_q^{\otimes k-1}) \to 0.
  \end{equation*}
  For odd $q$ this sequence splits canonically, using the action of the Adams operation $\psi^{-1}$ on the spectral sequence: it acts as $(-1)^k$ on the kernel and as $(-1)^{k-1}$ on the cokernel in the short exact sequence.  
  
For the final assertion for $\mathcal{O}_q'$: by Proposition \ref{normresidueconsequence} we need only consider $i=0$, and by injectivity in degree $0$ it follows in that case from a computation of orders:
both sides have order $q \cdot \# (\mathrm{Pic}(\mathcal{O}_q)/q)$.  (Alternatively prove surjectivity as in Corollary~\ref{KOq-Bott} below.)
The  version for $\mathcal{O}_q$ follows from Lemma \ref{Quillenexact}.
  \end{proof}

  The isomorphisms in different degrees in Lemma \ref{KOq} are intertwined through the action of $\beta$ in an evident way; this switches between $+$ and $-$ eigenspaces. 
    For example, 
the group $\KBott_{4k-2}(\mathcal{O}_q;\Z/q)^{(-)}$ is isomorphic to $\Z/q$ for any $k \in \Z$, generated by $\beta^{2k-1}$. 
We want to make the isomorphism on the $+$ eigenspace in  degree $4k-2$ more explicit.  

\begin{corollary} \label{KOq-Bott} The map 
  \begin{align*}
    \pi_0(\pic{\mathcal{O}_q})/q & \to \KBott_{4k-2}(\mathcal{O}_q;\Z/q)^{(+)}\\
    [L] & \mapsto \beta^{2k-1} \cdot ([L]-1)
  \end{align*}
  is an isomorphism of groups (where the group operation is induced by tensor product in the domain and direct sum in the codomain).   More invariantly,
  in the notation of \eqref{Bottbeta},  the isomorphism may be written
  \begin{equation}\label{eq:47}
    \begin{aligned}
    \pi_0(\pic{\mathcal{O}_q}) \otimes \mu_q(\mathcal{O}_q)^{\otimes (2k-1)} & \to \KBott_{4k-2}(\mathcal{O}_q;\Z/q)^{(+)}\\
    [L] \otimes \zeta^{\otimes (2k-1)}\quad\quad & \mapsto \beta(\zeta)^{2k-1} \cdot ([L]-1),
    \end{aligned}
  \end{equation}
  valid for any $L \in \pic{\mathcal{O}_q}$ and any $\zeta \in \mu_q(\mathcal{O}_q)$.  In this formulation the isomorphism is equivariant for the evident action of $\Gal(K_q/\Q)\cong (\Z/q)^\times$ on both sides.
  
  A similar result holds for $\KBott_{4k}(\mathcal{O}_q;\Z/q)$, except the roles of positive and negative eigenspaces for $\psi^{-1}$ are reversed.
\end{corollary}

\begin{proof}
  Multiplication by $\beta^{2k-1}\colon \KBott_0(\mathcal{O}'_q;\Z/q)^{(-)} \to \KBott_{4k-2}(\mathcal{O}'_q;\Z/q)^{(+)}$ is an isomorphism which under the isomorphisms of Lemma~\ref{KOq} corresponds to multiplication by $\zeta_q^{\otimes (2k-1)}\colon  H^2(\Spec(\mathcal{O}'_q);\mu_q) \to H^2(\Spec(\mathcal{O}'_q);\mu_q^{\otimes(2k-1)})$, so it suffices to prove that the composition
  \begin{equation}\label{eq:80}
    \begin{aligned}
      \pi_0(\pic{\mathcal{O}'_q})/q & \to K_0(\mathcal{O}'_q;\Z/q)^{(-)} \to
      \KBott_0(\mathcal{O}'_q;\Z/q)^{(-)} \xrightarrow{\text{Lem.~\ref{KOq}}}
                                  H^2_\mathrm{et}(\Spec(\mathcal{O}'_q);\mu_q)
    \\
    [L] & \mapsto [L] -1
  \end{aligned}
  \end{equation}
  is an isomorphism.  The mod $q$ \'etale Chern class $[L] \mapsto c_1(L)$ induces an isomorphism between the same two groups, so it suffices to identify~\eqref{eq:80} with $c_1$.  This identification is well known\footnote{In a preprint version of our paper we outlined a proof \cite[Proof of Corollary 2.12]{FGV} of this well known fact, since we were not able to locate a proof in the literature.}, and follows by tracing through the isomorphism between Bott inverted $K$-theory and \'etale cohomology induced by Thomason's spectral sequence.
  See also Remark~\ref{rem:preposterous-shortcut} for a shortcut.
\end{proof}

\begin{remark}\label{rem:preposterous-shortcut}
  For the reader who prefers to keep \emph{both} the norm residue theorem and Thomason's spectral sequence as black boxes not to be opened, it may be shorter to consider the two maps
  \begin{equation*}
    \pi_0(\pic{\mathcal{O}'_q})/q \to K_0(\mathcal{O}'_q;\Z/q)^{(-)} \to \KBott_0(\mathcal{O}'_q;\Z/q)^{(-)}    
  \end{equation*}
  separately.  The first is an isomorphism by the usual splitting $K_0(\mathcal{O}'_q;\Z/q) \cong \Z \oplus \pi_0(\pic{\mathcal{O}'_q})$ and the second by the final part of Lemma~\ref{KOq}.  That route gives a proof that~\eqref{eq:80} is an isomorphism without inspecting what the map is, at the cost of appealing to the norm residue theorem, thus invalidating Remark~\ref{rem: intro K(1) local}.
\end{remark}

  \begin{lemma}\label{K-beta-theoryeval} 
  For all $k \in \Z$ we  have
  \begin{align*}
    K^{(\beta)}_{4k-2}(\Z';\Z/q)^{(+)} & \cong H^2(\Spec(\Z');\mu_q^{\otimes 2k})\\
    K^{(\beta)}_{4k-2}(\Z';\Z/q)^{(-)} & \cong H^0(\Spec(\Z');\mu_q^{\otimes (2k-1)})\\
    K^{(\beta)}_{4k}(\Z';\Z/q)^{(+)} & \cong H^0(\Spec(\Z');\mu_q^{\otimes 2k})\\
    K^{(\beta)}_{4k}(\Z';\Z/q)^{(-)} & \cong H^2(\Spec(\Z');\mu_q^{\otimes 2k+1}).
  \end{align*}
  In odd degrees we have $K^{(\beta)}_{2k-1}(\Z'; \Z/q) \cong H^1(\Spec(\Z');\mu_q^{\otimes k})$ for all $k$, on which $\psi^{-1}$ acts as $(-1)^k$. 
 
    Finally, the map 
   $K_i(\Z';\Z/q) \to \KBott_i(\Z';\Z/q) $
is an isomorphism  for all $i \geq 0$
and the map 
   $K_i(\Z;\Z/q) \to \KBott_i(\Z;\Z/q) $
   is an isomorphism for for $i=0$ or $i \geq 2$.  
\end{lemma}
Recall our standing assumption that $q$ is odd. 

 \begin{proof}
  \label{ZBottinverted}  Similarly to the prior analysis we get  canonical isomorphisms
  \begin{equation*}
    \KBott_{2k-1}(\Z';\Z/q) \cong H^1_\mathrm{et}(\Spec(\Z');\mu_q^{\otimes k})
  \end{equation*}
  in odd degrees, and in even degrees we have short exact sequences
  \begin{equation}\label{eq:27}
    0 \to H^2_\mathrm{et}(\Spec(\Z');\mu_q^{\otimes k}) \to \KBott_{2k-2}(\Z';\Z/q) \to H^0_\mathrm{et}(\Spec(\Z');\mu_q^{\otimes  k-1}) \to 0,
  \end{equation}
  canonically split into positive and negative eigenspaces for $\psi^{-1}$ when $q$ is odd.  The periodicity of these groups has longer period though: multiplying with $\beta^{p^{n-1}(p-1)}$ increases $k$ by $p^{n-1}(p-1)$.
  
  As before the asssertion comparing $K$-theory and Bott-inverted $K$-theory of $\Z'$ follows from Proposition \ref{normresidueconsequence} by computing orders,
  and the assertion for $\Z$ uses Lemma \ref{Quillenexact}.
\end{proof}

\begin{proposition}\label{propcor:transfer-surjective}
Suppose that $q$ is odd.   Let $\Gal(K_q/\Q) \cong (\Z/q)^\times$ act on $K_*(\mathcal{O}'_q;\Z/q)$ by functoriality of algebraic $K$-theory.  Then the homomorphisms
  \begin{align*}
    \big(\KBott_{4k-2} (\mathcal{O}'_q;\Z/q)^{(+)}\big)_{\Gal(K_q/\Q)} &\to \KBott_{4k-2}(\Z';\Z/q)^{(+)}\\
    \big(\KBott_{4k}(\mathcal{O}'_q;\Z/q)^{(-)}\big)_{\Gal(K_q/\Q)} &\to \KBott_{4k}(\Z';\Z/q)^{(-)},
  \end{align*}
  induced by the transfer map $K_*(\mathcal{O}'_q;\Z/q) \to K_*(\Z';\Z/q)$, are both isomorphisms. (Here $(-)_{\Gal(K_q/\Q)}$ denotes coinvariants for $\Gal(K_q/\Q)$.) 
\end{proposition}

\begin{remark}\label{rem:transfer-Adams}    It will follow implicitly from the proof that the transfer map behaves in the indicated way with respect to eigenspaces for $\psi^{-1}$, but let us give an independent explanation for why the transfer map $K_*(\mathcal{O}_q;\Z/q) \to K_*(\Z;\Z/q)$ commutes with the Adams operation $\psi^{-1}$.  This may seem surprising at first, since the forgetful map from $\mathcal{O}_q$-modules to $\Z$-modules does not obviously commute with dualization.  The ``correction factor'' is the dualizing module $\omega$, isomorphic to the inverse of the \emph{different} $\mathfrak{d}$, which will play an important role later in the paper.  In this case the different is principal, and any choice of generator leads to a functorial isomorphism between the $\Z$-dual and the $\mathcal{O}_q$-dual.
\end{remark}

\begin{proof}
  The argument is the same in both cases, and uses naturality of Thomason's spectral sequence with respect to transfer maps: there is a map of spectral sequences which on the $E_2$ page is given by the transfer in \'etale cohomology and on the $E_\infty$ page by (associated graded of) the transfer map in $K$-theory.  This naturality is proved in Section 10 of \cite{BlumbergMandell}, the preprint version of \cite{BlumbergMandell-published}.  In our case the spectral sequences collapse, and identify the two homomorphisms in the corollary with the maps on $E_2^{-2,4k}$ and $E_2^{-2,4k+2}$, respectively.  Hence we must prove that the transfer maps
  \begin{equation*}
    \big(H^2_\mathrm{et}(\Spec(\mathcal{O}'_q);\mu_q^{\otimes t})\big)_{\Gal(K_q/\Q)} \to H^2_\mathrm{et}(\Spec(\Z');\mu_q^{\otimes t})
  \end{equation*}
  are isomorphisms for all $t$ or, equivalently, that their Pontryagin duals are isomorphisms.  By Poitou--Tate duality, the Pontryagin dual map may be identified with
  \begin{equation*}
    \pi^*: H^1_c(\Spec(\Z');\mu_q^{\otimes (1-t)}) \to \big(H^1_c(\Spec(\mathcal{O}'_q);\mu_q^{\otimes (1-t)})\big)^{\Gal(K_q/\Q)},
  \end{equation*}
  where the ``compactly supported'' cohomology is taken in the sense of \cite[Appendix]{GalatiusVenkatesh}, i.e.,\ defined as cohomology of a mapping cone. In this context
  we may apply a relative Hochschild-Serre spectral sequence\footnote{The relative Leray spectral sequence is noted in a topological context,  for example, in Exercise 5.6 of \cite{McCleary}. This implies such a spectral sequence for pairs of finite groups,
  and then for profinite groups by a limit argument.} to give an exact sequence 
  \begin{align*}
    0  &\to H^1((\Z/q)^*;H^0_c(\Spec(\mathcal{O}'_q);\mu_q^{\otimes(1-t)})) \to H^1_c(\Spec(\Z');\mu_q^{\otimes (1-t)}) \\
    & \to H^1_c(\Spec(\mathcal{O}'_q);\mu_q^{\otimes (1-t)})^{(\Z/q)^*} \to H^2((\Z/q)^*;H^0_c(\Spec(\mathcal{O}'_q);\mu_q^{\otimes (1-t)})).
  \end{align*}

  Now, the compactly supported cohomology group $H^0_c(\Spec(\mathcal{O}'_q);\mu_q^{\otimes(1-t)})$ is the kernel of the restriction map $\mu_q(\mathcal{O}'_q)^{\otimes(1-t)} \to \mu_q(\Q_p[\mu_q])^{\otimes(1-t)}$, which is an isomorphism.  The exact sequence then precisely becomes the desired isomorphism.
\end{proof}

\subsection{Bott inverted algebraic $K$-theory and homology of certain Galois groups}
\label{sec:algebraic-k-theory}

In this subsection we express Bott inverted algebraic $K$-theory of cyclotomic rings of integers in terms of certain Galois homology groups. This will be useful later one, when trying to relate $K$-theory to extensions of Galois modules.

Let $\widetilde{H}_q \subset \C$ be the Hilbert class field of $K_q = \Q[\zeta_q] \subset \C$, the maximal abelian extension unramified at all places.  Class field theory asserts an isomorphism $\pi_0(\pic{\mathcal{O}_q}) \cong \Gal(\widetilde{H}_q/K_q)$, given by the Artin symbol.  Let $H_q \subset \widetilde{H}_q$ be the largest extension with  $p$-power-torsion Galois group,  so that the Artin symbol factors over an isomorphism
\begin{equation} \label{Artmap2}
  \begin{aligned}
  \pi_0(\pic{\mathcal{O}_q}) \otimes \Z_p &\xrightarrow{\cong} \Gal(H_q/K_q)\\
  [\mathfrak{p}] \quad& \mapsto \left(\frac{H_q/K_q}{\mathfrak{p}}\right).
\end{aligned}
\end{equation}
It is easy to check that this map is equivariant for the action of $\Gal(K_q/\Q)$ which acts in the evident way on the domain, and on the codomain the action is induced by the short exact sequence
\begin{equation}\label{eq:48}
  \Gal(H_q/K_q) \to \Gal(H_q/\Q) \to \Gal(K_q/\Q).
\end{equation}

The following diagram   gives  the main tool  through which we will understand the transfer map $\mathrm{tr}: \KBott_{4k-2}(\mathcal{O}_q;\Z/q)^{(+)} \to \KBott_{4k-2}(\Z;\Z/q)^{(+)}$.
To keep typography simple,  we write (in the statement and its proof) $\mu_q$ for $\mu_q(\C)$, 
and for a Galois extension $E/F$ of fields, we write $H_*(E/F, -)$ for the group homology of the group $\Gal(E/F)$.

\begin{proposition} \label{shtuka}
  For all $k \in \Z$ there is a commutative diagram, with all horizontal maps isomorphisms
\begin{equation}
  \label{eq:50}
  \begin{aligned}
    \xymatrix{
      H_1(H_q/K_q;\mu_q^{\otimes (2k-1)}) \ar@{->>}[d]^{i_*} &
      \pi_0(\pic{\mathcal{O}_q}) \otimes \mu_q^{\otimes (2k-1)} \ar[l]_-{\mathrm{Art}}^-\cong  \ar[r]^-\cong
      & \KBott_{4k-2}(\mathcal{O}_q;\Z/q)^{(+)} \ar@{->>}[d]^{\mathrm{tr}}\\
      H_1(H_q/\Q;\mu_q^{\otimes (2k-1)}) \ar[rr]^-{\cong} & & \KBott_{4k-2}(\Z;\Z/q)^{(+)},
    }
  \end{aligned}
\end{equation}
where:
\begin{itemize}
\item  the map denoted $i_*$ is induced by the inclusion $\Gal(H_q/K_q) \subset \Gal(H_q/\Q)$;
\item the map denoted $\mathrm{Art}$ is induced by the Artin map \eqref{Artmap2},
together with the identification  $H_1(H_q/K_q, \mu_q^{\otimes (2k-1)}) \simeq \mu_q^{\otimes (2k-1)}
\otimes \Gal(H_q/K_q)$; 
\item the top arrow labeled ``$\cong$'' is the map of \eqref{eq:47}, i.e.\ the product of the map  $[L] \mapsto [L]-1 \in K_0(\mathcal{O}_q;\Z/q)^{(-)}$ composed with $K_0(\mathcal{O}_q;\Z/q)^{(-)} \to \KBott_0(\mathcal{O}_q;\Z/q)^{(-)}$
  and
    $\beta^{2k-1}: \mu_q(\C)^{\otimes{(2k-1)}} \to \KBott_{4k-2}(\mathcal{O}_q;\Z/q)^{(-)}$;
\item  the bottom arrow labeled ``$\cong$'' is induced by the rest of the diagram.
\end{itemize}
The same assertion holds without Bott-inversion of the $K$-theory for $k \geq 1$. 
\end{proposition}

\begin{proof} 
That the right top arrow is an isomorphism was already proved in   Corollary \ref{KOq-Bott}.
We have also seen that $\KBott_{4k-2}(\mathcal{O}_q;\Z/q)^{(+)} \to \KBott_{4k-2}(\Z;\Z/q)^{(+)}$ induces an isomorphism from the $\Gal(K_q/\Q)$ coinvariants on the source: see Proposition~\ref{propcor:transfer-surjective}, Lemma \ref{Quillenexact}
and Proposition \ref{normresidueconsequence}.

Therefore, we need only verify the corresponding property for $i_*$: it induces an isomorphism from the $\Gal(K_q/\Q)$-coinvariants on the source. 
 This follows from the  Hochschild--Serre spectral sequence for the extension~(\ref{eq:48}), which gives an exact sequence
\begin{multline} \label{HS5}
  H_2(K_q/\Q;\mu_q^{\otimes (2k-1)})  \to 
  H_0(K_q/\Q;H_1(H_q/K_q; \mu_q^{\otimes (2k-1)}))  \to H_1(H_q/\Q;\mu_q^{\otimes (2k-1)})  \\
  \to H_1(K_q/\Q;\mu_q^{\otimes (2k-1)}).
\end{multline}  Considering the action of the central element $c \in \Gal(K_q/\Q)$ given by complex conjugation 
we see that the two outer terms vanish (the ``center kills'' argument).

For the last sentence use Lemma  \ref{KOq} and Lemma
 \ref{K-beta-theoryeval}.
\end{proof}

\begin{remark}
Let $c \in \Gal(H_q/\Q)$ be complex conjugation. Then $H_0(\langle c\rangle;\mu_q^{\otimes (2k-1)}) = 0 = H_1(\langle c\rangle;\mu_q^{\otimes (2k-1)})$. Therefore the  map
\begin{equation*}
  H_1(H_q/\Q;\mu_q^{\otimes (2k-1)}) \to H_1(\Gal(H_q/\Q),\langle c\rangle;\mu_q^{\otimes (2k-1)})
\end{equation*}
is an isomorphism; on the right we have ``relative'' group homology, i.e.\ relative homology of classifying spaces.  This relative group homology  may therefore be substituted in place of the lower left corner of~(\ref{eq:50}). This observation will be significant later. 
\end{remark}

\begin{remark}
  \label{remark:p=q=3}  
  The case $p = q = 3$ is anomalous in that the Moore spectrum $\mathbb{S}/3$ does not admit a unital multiplication which is associative up to homotopy.  It does admit a unital and homotopy commutative multiplication though, which induces graded commutative---but a priori possibly non-associative---ring structures on $K_*(\mathcal{O}_3;\Z/3)$ and $\K_*(\Z;\Z/3)$.  There is no problem in defining Bott inverted $K$-theory, e.g.\ as in Remark~\ref{remark:telescope}, and according to \cite[A.11]{Thomason85} the construction of the spectral sequence holds also in this case.  The Bott element $\beta \in K_2(\mathcal{O}_3;\Z/3)$ is defined as before, and multiplication by $\beta$ defines an endomorphism of $K_*(\mathcal{O}_3;\Z/3)$.  Iterating this endomorphism $2k-1$ times gives a homomorphism $K_0(\mathcal{O}_3;\Z/3) \to K_{4k-2}(\mathcal{O}_3;\Z/3)$, which we use to give meaning to expressions like $\beta^{2k-1} ([L]-1)$ in this section.

  In this interpretation the results of this section hold also in the case $p = q = 3$.  Multiplication by powers of a Bott element also appear in Section 5, we leave it to the diligent reader to verify that similar remarks apply there.
\end{remark}

\section{Symplectic $K$-theory}\label{sec: symplectic K-theory}
\label{sec:symplectic-K}

In this section, we define the symplectic $K$-theory of the integers. Our main goal is to state
and prove Theorem \ref{thm:KSp-of-Z},
which shows that this symplectic $K$-theory,  with $\Z/q$-coefficients,
splits into two parts: one arising from the $+$ part
of the algebraic $K$-theory of $\Z$, and the other from the $-$
part of topological $K$-theory.

\subsection{Definition of symplectic $K$-theory}
Just as $K$-theory arises from the symmetric monoidal category of
projective modules, 
symplectic $K$-theory arises from the symmetric monoidal category of
{\em symplectic} modules:

Consider the groupoid whose objects are pairs $(L,b)$, where $L$ is a finitely generated free  $\Z$-module and $b: L \times L \to \Z$ is a skew symmetric pairing whose adjoint $L \to L^\vee$ is an isomorphism, and whose morphisms are $\Z$-linear isomorphisms $f: L \to L'$ such that $b'(fx,fy) = b(x,y)$ for all $x,y \in L$.  This groupoid becomes symmetric monoidal with respect to orthogonal direct sum $(L,b) \oplus (L',b') = (L \oplus L',b + b')$, and we shall denote it $\SP(\Z)$.  The corresponding space $|\SP(\Z)|$ then inherits a product structure, and as before 
we get a spectrum $\KSp(\Z)$ and a group-completion map
\begin{equation*}
  |\SP(\Z)| \to \Omega^\infty \KSp(\Z).
\end{equation*}
The positive degree homotopy groups of $\KSp(\Z)$ can be computed via the Quillen plus construction (with respect to the commutator subgroup of $\Sp_{\infty}(\Z) = \pi_1(B \Sp_{\infty}(\Z))$). 

\begin{definition}
The \emph{mod $q$ symplectic $K$-theory groups of $\Z$} are defined as 
\[
\KSp_i(\Z; \Z/q) := \pi_i(\KSp(\Z); \Z/q)
\]
and 
the \emph{$p$-adic symplectic $K$-theory groups} can be defined\footnote{Or equivalently as the homotopy groups of the $p$-completion of the spectrum $\KSp(\Z)$.  These agree because the homotopy groups of $\KSp(\Z)$ are finitely generated abelian groups.} as
\[
 \KSp_i(\Z;\Z_p) := \varprojlim_n \KSp_i(\Z;\Z/p^n).
\] 
 \end{definition}

\subsection{Hodge map and Betti map}

The groups $\KSp_i(\Z;\Z/q)$ are described in Theorem~\ref{thm:KSp-of-Z} below. The result is stated in terms of two homomorphisms, the \emph{Hodge map} and the \emph{Betti map}, which we first define. 

\subsubsection{The Betti map}

\begin{definition}\label{def:Soule-map}
  Let $c_B: \KSp(\Z) \to K(\Z)$ be the spectrum map defined by the forgetful functor $\SP(\Z) \to \proj{\Z}$.     We shall use the same letter $c_B$ to denote the induced homomorphism on mod $q$ homotopy groups
  \begin{equation*}
    c_B: \KSp_i(\Z;\Z/q) \to K_i(\Z;\Z/q). 
  \end{equation*}
\end{definition}

\subsubsection{The Hodge map}The Hodge map is  more elaborate. It arises from the functors of
\begin{equation} \label{Hodgezigzag} \underbrace{ \mbox{symplectic $\Z$-modules}}_{\SP(\Z)} \rightarrow \underbrace{  \mbox{symplectic $\R$-modules}}_{\SP(\R^{\mathrm{top}})} \leftarrow \underbrace{ \mbox{Hermitian $\C$-vector spaces}}_{\Utop}. \end{equation} 
where the entries are now regarded as symmetric monoidal groupoids that are enriched in topological spaces. In more detail: 
\begin{itemize}
\item 
   Let $\SP(\R^\mathrm{top})$ be the groupoid (enriched in topological spaces) defined as $\SP(\Z)$ but with $\R$-modules $L$ and $\R$-bilinear symplectic pairings $b: L \times L \to \R$.  
   We regard it as a  groupoid enriched in topological spaces, where morphism spaces are topologized in their Lie group topology, inherited from the topology on $\R$ (the superscript ``top'' signifies that we remember the topology, as opposed to considering $\R$ as a discrete ring).  
  
  \item Write $\mathcal{U}(\Ctop)$ for the groupoid (again enriched in topological spaces) whose objects are finite dimensional $\C$-vector spaces $L$ equipped with a positive definite Hermitian form $h: L \otimes_\C \overline{L} \to \C$, and morphisms the unitary maps topologized in the Lie group topology. 
  
  \item   The functor $\mathcal{U}(\Ctop) \to \SP(\R^\mathrm{top})$ is obtained by sending a unitary space $(L, h)$,  as in (ii), to 
  the underlying real vector space $L_{\R}$, equipped with the symplectic form $\mathrm{Im} \ h$.
  This functor induces a bijection on sets of isomorphism classes and homotopy equivalences on all morphisms spaces, 
  because $U(g) \subset \Sp_{2g}(\R)$ is a homotopy equivalence. 

  \end{itemize}
 We equip these categories with the symmetric monoidal structures given by direct sum. Then, as discussed in the Appendix, the diagram \eqref{Hodgezigzag} gives rise to 
  a diagram of $\Gamma$-spaces and thereby to 
  a diagram of spectra:
\begin{equation} \label{kspkspku} \KSp(\Z) \rightarrow \KSp(\Rtop) \leftarrow ku,\end{equation}
  where we follow standard notation in using $ku$ (connective $K$-theory) to
  refer to the spectrum associated to $\mathcal{U}(\Ctop)$. 
   The last arrow here is a weak equivalence, i.e.\ induces an isomorphism on all homotopy groups.
   Indeed, as noted above,
  $\coprod_g B\mathrm{U}(g) \simeq \coprod_g B\Sp_{2g}(\R)$
  is  a weak equivalence,  therefore the group completions are weakly equivalent, 
  therefore
   $\Omega^{\infty} (ku) \rightarrow \Omega^{\infty}(\KSp(\Rtop))$ is a weak equivalence, 
 and so the map
  $ku \rightarrow \KSp(\Rtop)$  of connective spectra is  a weak equivalence. 
  
  In the homotopy category of spectra, weak equivalences become invertible, and so 
  the diagram \eqref{kspkspku}
  induces there a map 
  $ \KSp(\Z) \rightarrow ku.$

  \begin{definition}\label{def:Hodge-map}
The {\em Hodge map} is the morphism  \begin{equation*}
    c_H: \KSp(\Z) \to ku
  \end{equation*}
  in the homotopy category of spectra that has just been constructed.
The map $c_H$ induces a homomorphism 
  \begin{equation*}
    \KSp_i(\Z;\Z/q) \to \pi_i(ku;\Z/q)
  \end{equation*}
  which we shall also call the Hodge map.  By Bott periodicity, the target is $\Z/q$ when $i$ is even and 0 when $i$ is odd.
\end{definition}

\begin{remark}[Explanation of terminology]
With reference to the relationship between symplectic $K$-theory
and moduli of principally polarized abelian varieties (\S \ref{elab}) the Hodge map  is related to the Chern classes of the \emph{Hodge bundle} 
whose fiber over $A$ is  $H^0(A, \Omega^1)$, and thus to the ``Hodge realization'' of $A$. On the other hand, the Betti map is related to the ``Betti realization'' $H_1(A, \Z)$. 
\end{remark}

\subsection{Determination of symplectic $K$-theory in terms of algebraic $K$-theory}

As explained above, the Adams operation $\psi^{-1}$ induces an involution of $K_*(\Z;\Z/q)$ which gives a splitting for odd $q$ into positive and negative eigenspaces.  There are also Adams operations on $ku$, and their effect on homotopy groups are very easy to understand.  In particular, $\psi^{-1}$ acts as $(-1)^k$ on $\pi_{2k}(ku) \cong \Z$.  The main goal of this section is to explain the following result.
\begin{theorem}\label{thm:KSp-of-Z}
  For odd $q = p^n$, the homomorphism
  \begin{equation}\label{eq:30}
    \KSp_i(\Z;\Z/q) \to (K_i(\Z;\Z/q))^{(+)} \oplus (\pi_i(ku;\Z/q))^{(-)}
  \end{equation}
  defined by the Betti and Hodge maps, composed with the projections onto the indicated eigenspaces for $\psi^{-1}$, is an isomorphism. 
  (We will refer later to the induced isomorphism as the Betti-Hodge map).   In particular we get for $k \geq 1$
  \begin{equation*}
    \KSp_{4k-2}(\Z;\Z/q) \cong H^2(\Spec(\Z');\mu_q^{\otimes 2k}) \oplus (\Z/q).
  \end{equation*}
\end{theorem}
The latter statement follows from the first using Corollary \ref{K-beta-theoryeval}. 
Using the other statements of that Corollary, 
taking the inverse limit over $n$, and using that $K_i(\Z)$ and $\KSp_i(\Z)$ are finitely generated for all $i$ to see that the relevant derived inverse limits vanish, we deduce the following.

\begin{corollary}\label{cor: explicit determination of KSp}
  For odd $p$ and $i > 0$, the groups $\KSp_i(\Z; \Z_p)$ are as in the following table, with the identifications given explicitly by the Betti-Hodge map:
\begin{center}
\begin{tabular}{|c|c|c|c|c|}
$i \mod{4}$ &  0 & 1 & 2 & 3 \\ 
\hline 
$\KSp_i(\Z; \Z_p)$ & 0 & 0 & $K_i(\Z; \Z_p) \oplus \Z_p$ & $K_i(\Z; \Z_p) $
\end{tabular}
\end{center}
\end{corollary}

\begin{remark}
  The relationship between $K$-theory and Hermitian $K$-theory is more complicated when the prime 2 is not inverted, and is well understood only quite recently.  See \cite{nineauthorI, nineauthorII, nineauthorIII}\footnote{The second author wishes to thank Fabian Hebestreit, Markus Land, Kristian Moi, and Thomas Nikolaus for helpful conversations.}
 as well as \cite{hesselholt2015real}, \cite{1911.11682} and \cite{MR4015334}.

  We consider only odd primes in this paper, where the isomorphism~\eqref{eq:30} can be deduced from the work of Karoubi \cite{Karoubi80}.
\end{remark}

\subsection{Symplectic $K$-theory and Grothendieck--Witt theory}

Symplectic $K$-theory as defined above is a special case of Grothendieck--Witt theory, introduced by Karoubi and Villamayor \cite{MR279163} under the name of \emph{Hermitian $K$-theory}.  In the original definition the input is a ring $A$ equipped with an anti-involution $x \mapsto \overline{x}$ and an element $\epsilon \in A$ such that $\epsilon \overline{\epsilon} = 1$.  These groups were denoted ${_\epsilon \mathcal{L}_n(A)}$, and for $n > 0$ defined as
\begin{equation*}
  {_\epsilon \mathcal{L}_n(A)} = \pi_n(B{_\epsilon O}(A)^+),
\end{equation*}
where ${_\epsilon O(A)}$ is the direct limit as $g \to \infty$ of certain subgroups ${_\epsilon O_{g,g}(A)} \subset \mathrm{GL}_{2n}(A)$.  The special case $A = \Z$ and $\epsilon = -1$ is closely related to symplectic $K$-theory, because ${_\epsilon O_{g,g}(\Z)}$ is a subgroup of $\Sp_{2g}(\Z)$ of finite index $2^{g-1}(2^g + 1)$, and a transfer argument can be used to show that the inclusion induces a homomorphism
\begin{equation*}
  {_{-1} \mathcal{L}_n(\Z)} \to \KSp_n(\Z)
\end{equation*}
which becomes an isomorphism after inverting 2. 

After inverting 2, we have a splitting $K_n(\Z)[\frac12] = K_n(\Z)[\frac12]^{(+)} \oplus K_n(\Z)[\frac12]^{(-)}$ into eigenspaces of the Adams operation $\psi^{-1}$, and it follows\footnote{In an earlier version of the paper, we explained in more detail how to deduce this isomorphism from Karoubi's main result.  We decided this was unnecessary, not least in light of the thorough treatment in \cite{nineauthorI,nineauthorII,nineauthorIII,HebestreitSteimle}.} from the main result of \cite{Karoubi80} that there is an isomorphism
\begin{equation*}
  \KSp_n(\Z)[\tfrac12] \xrightarrow{\cong} \big(K_n(\Z)[\tfrac12]\big)^{(+)} \oplus \big({_{-1} W}_n(\Z)[\tfrac12]\big),
\end{equation*}
for certain groups ${_{-1} W}_n(\Z)[\tfrac12]$ which vanish when $n \not\equiv 2 \mod 4$ and are isomorphic to $\Z[\tfrac12]$ when $n \equiv 2$.  Moreover, the projection $\KSp_n(\Z)[\tfrac12] \to K_n(\Z)[\frac12]^{(+)}$ is induced by the inclusions of discrete groups $\Sp_{2g}(\Z) \to \mathrm{GL}_{2g}(\Z)$ and the inclusion $K_n(\Z)[\frac12]^{(+)} \to \KSp_n(\Z)[\tfrac12]$ is induced by the hyperbolic construction $\mathrm{GL}_g(\Z) \hookrightarrow \Sp_{2g}(\Z)$.

In the notation of \cite{nineauthorII,nineauthorIII} and \cite{HebestreitSteimle}, our $\KSp(\Z)$ agrees by definition with what is denoted $\mathrm{GW}^{-s}_\mathrm{cl}(\Z)$ and $\mathrm{GW}^s_\mathrm{cl}(\Z;-\Z)$ there.  The subscript ``cl'' is short for ``classical'' and denotes that these are defined as homotopy groups of a Quillen plus-construction, as in Karoubi's work and in our discussion of $\KSp(\Z)$.  The main result of \cite{HebestreitSteimle} compares that with a spectrum $\mathrm{GW}^{-gs}(\Z) = \mathrm{GW}(\mathcal{D}^p(\Z),\Qoppa^{-gs})$, whose definition is based on chain complexes instead of discrete abelian groups, which by an instance of the main theorem of \cite{nineauthorII} fits into a cofiber sequence with algebraic $L$-theory and the homotopy orbits of the involution on algebraic $K$-theory.  The upshot is a fiber sequence of spectra of the form
\begin{equation}\label{eq:3}
  K(\Z)_{hC_2} \to \KSp(\Z) \to \tau_{\geq 0} L^{-s}(\Z),
\end{equation}
an instance of \cite[Corollary 8.3.5]{HebestreitSteimle} for example.  After inverting 2, we obtain a spectrum level splitting $K(\Z)[\frac12] \simeq K(\Z)[\frac12]^{(+)} \oplus K(\Z)[\frac12]^{(-)}$, and the cofiber sequence~\eqref{eq:3} becomes
\begin{equation}
  \label{eq:5}
  K(\Z)[\tfrac12]^{(+)} \to \KSp(\Z)[\tfrac12] \to \tau_{\geq 0} L^{-s}(\Z)[\tfrac12],
\end{equation}
where the first map is induced by the hyperbolic construction (i.e., sending a finitely generated abelian group $M$ to $M \oplus M^\vee$ equipped with the standard symplectic form).  The sequence is canonically split by the spectrum map $\KSp(\Z) \to K(\Z)[\frac12]^{(+)}$ induced by the inclusion $\Sp_{2g}(\Z) \to \mathrm{GL}_{2g}(\Z)$.
Finally, $\pi_n L^{-s}(\Z)[\tfrac12]$ coincides\footnote{We point out that Karoubi's notation $\mathcal{L}$ is a special case of the $\mathrm{GW}$ from \cite{nineauthorI, nineauthorII, nineauthorIII}.  The $L$-theory in the latter papers are more similar to Karoubi's $W$-groups} with Karoubi's ${_{-1} W}_n(\Z)[\tfrac12]$.

\subsection{Proof of Theorem \ref{thm:KSp-of-Z}.}

The proof requires  the following non-triviality result about the Hodge map in degrees $\equiv 2 \mod 4$, proved in Section~\ref{prop:generators-for-KSp}.
\begin{proposition}\label{prop:postponed-prop}
  The homomorphism
  \begin{equation*}
    K\Sp_{4k-2}(\Z;\Z/p) \to \pi_{4k-2}(ku;\Z/p) \cong \Z/p
  \end{equation*}
  induced by the Hodge map is non-zero for all $k\geq 1$.
\end{proposition}
This proposition, valid under our standing assumption that $p$ is odd, implies that the homomorphism $\KSp_{4k-2}(\Z) \to \pi_{4k-2}(ku) \cong \Z$ is non-zero, and in fact that it becomes surjective after inverting 2.  
The proof (in Proposition \ref{prop: hodge map surjects})
amounts to constructing a spectrum map $\Sigma^\infty_+ B(\Z/p) \to K\Sp(\Z)$ whose composition with the Hodge map is nonzero in $\pi_{4k-2}(-;\Z/p)$.

\begin{proof}[Proof of Theorem~\ref{thm:KSp-of-Z}, assuming  Proposition~\ref{prop:postponed-prop}]
  Isomorphism with $p$-local coefficients implies isomorphism with mod $q = p^n$ coefficients, so we have isomorphisms
  \begin{equation*}
    \pi_i(\KSp(\Z);\Z/q) = \pi_i(\mathrm{GW}^{-gs}(\Z);\Z/q) \xrightarrow{\cong} K_i(\Z;\Z/q)^{(+)} \oplus \pi_i(L^{-s}(\Z);\Z/q),
  \end{equation*}
  while we wish to show that 
  \begin{equation}\label{eq:copy}
    \pi_i(\KSp(\Z);\Z/q) \xrightarrow{(c_B,c_H)} K_i(\Z;\Z/q)^{(+)} \oplus \pi_i(ku;\Z/q)^{(-)}
  \end{equation}
  is an isomorphism.  By inspection the
  groups $\pi_i(ku;\Z/q)^{(-)}$ and $\pi_i(L^{-gs}(\Z);\Z/q)$ are abstractly isomorphic for all $i \geq 0$.  Since all the groups involved here are finite, it suffices to show that~\eqref{eq:copy} is surjective.   Composing the hyperbolization map $K_i(\Z;\Z/p) \to K\Sp_i(\Z;\Z/p)$ with the Betti-Hodge map induces
  \[
  (1 + \psi^{-1},0): K_i(\Z;\Z/p) \to K_i(\Z;\Z/p) \times \pi_i(ku;\Z/p)^{(-)},
  \]
  because the composition with the Hodge map may be identified with the map $K(\Z) \to ku$ arising from the inclusion $\Z \to \C^\mathrm{top}$ composed with $1 + \psi^{-1}$, which lands in the positive eigenspace.  Hence the image of \eqref{eq:copy} contains the first summand. The second summand is nonzero only for $i$ congruent to $2$ modulo $4$, and the claim follows from Proposition~\ref{prop:postponed-prop}.
\end{proof}

\begin{remark} A similar argument shows that
  on the level of spectra,
  there is a weak equivalence
 $    K\Sp(\Z)[\tfrac12] \xrightarrow{\simeq} K(\Z)[\tfrac1{1 + \psi^{-1}}] \times ku[\tfrac1{1 - \psi^{-1}}].$
\end{remark} 

\subsection{Bott inverted symplectic $K$-theory} 
\label{sec:bott-invert-sympl}

This subsection is in fulfillment of Remark~\ref{rem: intro K(1) local}, but is not logically necessary for the main presentation of our results.

Recall from Remark~\ref{remark:telescope} that there is a spectrum $T$ such that $\KBott(X;\Z/q) = \pi_*(K(X) \wedge T)$, and that $T$ is defined as a mapping telescope of a self-map $\Sigma^m \mathbb{S}/q \to \mathbb{S}/q$ with $m = 2p^{n-1}(p-1)$, chosen with the property that it induces an isomorphism $\pi_i(K) \to \pi_{i + m}(K)$ for all $i$ when $K$ denotes periodic complex $K$-theory.  Such a map is often called a \emph{$v_1$ self-map}, and serves as a replacement for multiplication by the Bott element.  We may then define ``Bott inverted homotopy groups'' of any spectrum $E$ as the homotopy groups of $E \wedge T$, although this is more commonly called \emph{$v_1$-inverted homotopy groups} and denoted
\begin{equation*}
  \pi_i(E;\Z/q)[v_1^{-1}] =  \pi_i(E \wedge T) = \colim \bigg(\pi_i(E;\Z/q) \to \pi_{i + m}(E;\Z/q) \to \dots \bigg),
\end{equation*}
where the maps in the direct limit are induced by the chosen $v_1$ self-map.

For example, the natural map $ku \to K$ from connective to periodic complex $K$-theory induces isomorphisms
\begin{equation*}
  \pi_i(ku;\Z/q)[v_1^{-1}] \to \pi_i(K;\Z/q)[v_1^{-1}] \leftarrow \pi_i(K;\Z/q),
\end{equation*}
for all $i \in \Z$.  In this notation we have
\begin{equation*}
  K_i^{(\beta)}(X;\Z/q) = \pi_i(K(X);\Z/q)[v_1^{-1}],
\end{equation*}
and we may completely similarly define
\begin{equation}
  \label{eq:66}
  \KSp^{(\beta)}_i(\Z;\Z/q) := \pi_i(\KSp(\Z);\Z/q)[v_1^{-1}].
\end{equation}
Also, we define the \emph{$p$-adic Bott-inverted symplectic $K$-theory groups}
\[
\KSp^{(\beta)}_i(\Z;\Z_p)  := \varprojlim_n  \KSp^{(\beta)}_i(\Z;\Z/p^n).
\]

Since colimits preserve isomorphisms, we immediately deduce the following.
\begin{corollary}\label{cor: bott inverted symplectic K}
  The Bott inverted symplectic $K$-theory groups of $\Z$ are given by isomorphisms
  \begin{equation*}
    \KSp^{(\beta)}_i(\Z;\Z/q) \xrightarrow{(c_B,c_H)} K^{(\beta)}_i(\Z;\Z/q)^{(+)} \times \pi_i(K;\Z/q)^{(-)}
  \end{equation*}
  for all $i \in \Z$.\qed
\end{corollary}
These groups are periodic in $i$ and in particular they are likely non-zero in negative degrees.  
We then obtain isomorphisms
\begin{equation*}
  \KSp^{(\beta)}_i(\Z;\Z_p) \xrightarrow{(c_B,c_H)} K^{(\beta)}_i(\Z;\Z_p)^{(+)} \times \pi_i(K;\Z_p)^{(-)}
\end{equation*}
of $\Z_p$-modules.  These are still non-zero in many negative degrees, but are no longer periodic of any degree.

\begin{remark}
To elaborate upon Remark \ref{rem: intro K(1) local}, we explain that these inverse limits of Bott inverted mod $p^n$ groups may be re-expressed using \emph{$K(1)$-localization}. The $K(1)$-localization of a spectrum $E$ consists of another spectrum $L_{K(1)} E$ and a map $E \to L_{K(1)} E$ with various good properties. (The functor $L_{K(1)}$ depends on $p$, which is traditionally omitted from the notation.) The defining properties include that the induced homomorphism in $K/p$-homology $(K/p)_*(E) \to (K/p)_*(L_{K(1)} E)$ is an isomorphism, where $K$ denotes periodic complex $K$-theory and $K/p = (\mathbb{S}/p) \wedge K$.  More relevant for us is that it ``implements inverting $v_1$'', see \cite[Theorem 10.12]{Ravenel}, and we have canonical isomorphisms
  \begin{align*}
    \KSp^{(\beta)}_i(\Z;\Z/q) & \cong \pi_i(L_{K(1)} \KSp(\Z);\Z/q)\\
    \KSp^{(\beta)}_i(\Z;\Z_p) & \cong \pi_i(L_{K(1)} \KSp(\Z)).
  \end{align*}
\end{remark}


\newcommand{\Tr}{\mathrm{Tr}}
\newcommand{\moddd}{/\!\!/}

\section{Review of the theory of CM abelian varieties}\label{sec: cm Ab var}

In the main part of this section, we  discuss the theory of abelian varieties with \emph{complex multiplication} (CM). 
In order to motivate why we are doing this, let us 
first explain how $\KSp$ is related to abelian varieties,
and then outline how the theory of complex multiplication
can be used to produce classes in  $\KSp(\Z)$.

\subsection{Abelian varieties, symplectic $K$-theory, and the construction of CM classes}
\label{sec:abel-vari-sympl}

As discussed in the outline in \S \ref{outline}, 
we  are going to  construct  certain ``CM classes'' in the symplectic $K$-theory of $\Z$. 
 Let us first go from abelian varieties to $K$-theory, before considering how CM enters the picture. 
 
 There is a functor between groupoids
\begin{equation} \label{eq:64}
  \mathcal{A}_g(\C) \to \SP(\Z).
\end{equation}
Here the domain $\mathcal{A}_g(\C)$ denotes the groupoid of principally polarized abelian varieties and isomorphisms between such.  In particular, we do not take the topology of $\C$ into account at this moment. 

An object in $\mathcal{A}_g(\C)$ consists of an abelian variety $\mathcal{A} \to \Spec(\C)$ together with a polarization, which is given by a line bundle $\mathcal{L} \to A \times_{\Spec(\C)} A$, rigidified by a non-zero section of $\mathcal{L}$ over $(e,e)$.  The reference map $A \to \Spec(\C)$ allows us to take ``Betti'' homology $H_*(A(\C);\Z)$ and cohomology, and $c_1(\mathcal{L}) \in H^2(A(\C) \times A(\C);\Z)$ defines a skew symmetric pairing on $L = H_1(A(\C);\Z)$ which is perfect because the polarization is principal.  We therefore have an object
 $
  (H_1(A;\Z),c_1(\mathcal{L})) \in \SP(\Z).
 $
Similarly isomorphisms in $\mathcal{A}_g(\C)$ are sent to isomorphisms in $\SP(\Z)$, so \eqref{eq:64} induces
 $
  |\mathcal{A}_g(\C)| \to |\SP(\Z)| \to \Omega^\infty \KSp(\Z)
 $
and then by adjunction a map of spectra
\begin{equation}\label{eq:51}
 \Sigma^{\infty}_+ |\mathcal{A}_g(\C)|  \rightarrow \KSp(\Z). 
\end{equation}

Next we explain what CM classes are. We take a principally polarized abelian variety $A$ that admits an action of the cyclotomic ring $\mathcal{O}_q = \Z[e^{2 \pi i /q}] \subset \C$.
In particular, the cyclic group
$\Z/q$ acts on $A$ (where $1 \in \Z/q$ acts via $e^{2 \pi i/q}$)
giving rise to a morphism of groupoids
$ B(\Z/q) \rightarrow \mathcal{A}_g(\C)$,
whence
$$ \Sigma^{\infty}_+ (B(\Z/q)) \rightarrow \Sigma^{\infty}_+ |\mathcal{A}_g(\C)| \rightarrow \KSp(\Z).$$
Now take homotopy with mod $q$ coefficients. On the left, 
we get  the stable homotopy of $B(\Z/q)$ with mod $q$ coefficients;
in \S \ref{Bottcyclic} we described a polynomial algebra $\Z/q[\beta]$
inside this homotopy ring. The image of powers of $\beta$ under the composite
in $\KSp_*(\Z;\Z/q)$ are, by definition,  the ``CM classes'' of \S \ref{outline}.
(A more precise version of this discussion is given after 
 Proposition \ref{prop:generators-for-KSp}). 

Now let us review how   principally polarized abelian varieties
with an action of $\mathcal{O}_q$ are parameterized.
We will work a little more generally:
 for any $2g$-dimensional abelian variety $A$,
 the dimension of any commutative $\Q$-subalgebra of
  $\mathrm{End}(A) \otimes \Q$ is
at most $g$. If equality holds, then $A$ is said
to have ``complex multiplication'' (or CM for short), and the ring $ 
  \mathrm{End}(A)$ is necessarily a  {\em CM order}:

  \begin{definition} 
  \begin{enumerate}
  \item
A {\em CM field} is, by definition, a field extension of $\Q$ which is a totally imaginary extension 
of a totally real field $E^+$, i.e.\ $E \cong E^+(\sqrt{d})$
where all embeddings $E^+ \rightarrow \C$ have real image,
and all take $d$ to negative real numbers.\footnote{In the literature one sometimes sees a slightly broader definition of CM fields, including totally real fields.} 
  A {\em CM algebra} is a product of CM fields.

 \item A {\em CM order} is an 
order in a CM algebra $E$ stable by conjugation.

Here ``order''
means a subring $\mathcal{O} \leqslant E$ which is free as a $\Z$-module and for which $\mathcal{O} \otimes_{\Z} \Q \rightarrow E$ is an isomorphism;
 and ``conjugation''  is the unique automorphism $x \mapsto \bar{x}$ of $E$ 
 which induces conjugation in any homomorphism $E \hookrightarrow \C$.  
\end{enumerate}
\end{definition}

 We can construct CM abelian varieties as follows:
 taking $\mathcal{O}$ a CM-order, let
 $\mathfrak{a} \leqslant \mathcal{O}$
 be an ideal, and $\Phi: \mathcal{O} \otimes \R \simeq \C^g$
 an isomorphism. Then
 $(\mathcal{O} \otimes \R)/\mathfrak{a}$
has the structure of complex analytic torus. To give it an algebraic
structure, one must polarize the resulting torus: 
 one needs a symplectic $\Z$-valued pairing on the first homology group $\mathfrak{a}$.
To get it one chooses a  suitable purely imaginary element $u \in \mathcal{O} \otimes \Q$
and considers the symplectic form $(x,y) \in \mathfrak{a} \times \mf{a} \mapsto \mathrm{Tr}(x u \bar{y}).$
All CM abelian varieties over $\C$ arise from this construction. 
  
 The resulting construction produces a complex abelian variety $A$
 from the data $\mathcal{O},  \mathfrak{a}, \Phi, u$. 
 For any automorphism $\sigma \in \mathrm{Aut}(\C)$
 the twist $\sigma(A)$, i.e.\ the abelian variety obtained
 by applying $\sigma$ to a system of equations defining $A$,
 necessarily arises from some other data $(\mathcal{O}', \mathfrak{a}', \Phi', u')$. 
 The {\em Main Theorem of Complex Multiplication} in its sharpest form, describes how to compute this new data. This theorem (in a slightly weaker form) 
is due to Shimura and Taniyama, and it will eventually be used by us to compute the action of $\Aut(\C)$
on CM classes in $\KSp(\Z;\Z/q)$.

In our presentation --  designed to  simplify the interface  with algebraic $K$-theory --
we will regard the basic object as the $\mathcal{O}$-module $\mathfrak{a}$ together with the skew-Hermitian form
 $x,y \mapsto  x u \bar{y},$
 valued in $\mathfrak{a} u \overline{\mathfrak{a}}$.
 We will
 interpret the construction sketched above
 as a functor of groupoids  
  $$  \mathcal{P}_{E}^{-} = \mbox{groupoid of skew-Hermitian $\mathcal{O}$-modules} \stackrel{\ST}{\longrightarrow} \mathcal{A}_g(\C).$$
The composition of this functor with  \eqref{eq:64} $\mathcal{A}_g(\C) \rightarrow \mathcal{SP}(\Z)$
associates to a skew-Hermitian module an underlying symplectic $\Z$-module.

\begin{remark}
 The appearance of Hermitian forms is   quite natural from the point of view of the theory of Shimura varieties: indeed,
 the set of abelian varieties with CM by a given field $E$ is related
 to the Shimura variety for an associated unitary group. 
 \end{remark}

\subsection{Picard groupoids} \label{picgroup}

Recall from Section~\ref{def:pic} that for a commutative ring
$R$ we have defined $\mathrm{Pic}(R)$ as the groupoid whose objects are rank 1 projective $R$-modules and whose morphisms are $R$-linear isomorphisms between them.  For a ring with involution, there is a version of this groupoid where the objects are equipped with 
perfect sesquilinear forms.

 \begin{definition}
   For a commutative ring $\mathcal{O}$ with involution $x \mapsto \overline{x}$ and an $\mathcal{O}$-module $L$ we shall write $\overline{L}$ for the module with the same underlying abelian group but $\mathcal{O}$-action changed by the involution.
  For a rank 1 projective $\mathcal{O}$-module $\omega$ equipped with an $\mathcal{O}$-linear involution $\iota:\omega \to \overline{\omega}$ we shall write $\PP(\mathcal{O},\omega,\iota)$ for the following groupoid:
  \begin{itemize}
  \item[-] Objects are pairs $(L,b)$ where $L$ is a rank 1 projective $\mathcal{O}$-module and $b: L \otimes_\mathcal{O} \overline{L} \to \omega$ an isomorphism satisfying $b(x \otimes y) = \iota(b(y \otimes x))$. 
  
  We may equivalently view $b$ as a function $L \times L \rightarrow \omega$ which is $\Cal{O}$-linear in the first variable and conjugate $\Cal{O}$-linear in the second variable, and we will frequently do this below. 
  \item[-] Morphisms $(L,b) \to (L',b')$ are $\mathcal{O}$-linear isomorphisms $\phi: L \to L'$ such that $b'(\phi(x),\phi(y)) = b(x,y)$ for all $x,y \in L$. 
\end{itemize}
\end{definition}

 There are some instances of this construction of particular interest for us.
 Take $E$ to be a CM field and $\mathcal{O}$ to be its ring of integers
  (i.e.,\ the integral closure of $\Z$ in $E$), with involution the conjugation $x  \mapsto \bar{x}$.  
  
 \begin{itemize}
\item[(i)] $\PPplus$, the groupoid of Hermitian forms on $\mathcal{O}$: 

Take  $\omega = \mathcal{O}$ with the conjugation involution and
set $\PPplus = \PP(\mathcal{O}, \omega, \iota)$. 
  \item[(ii)]  $\mathcal{P}_{E \otimes \R}^+$, the groupoid of Hermitian forms on $E \otimes \R$: 
 
 As in (i), but now replacing $\mathcal{O}$ by $\mathcal{O} \otimes \R$ and
 $\omega$ by $\mathcal{O} \otimes \R = \omega \otimes \R$, i.e.\ $\mathcal{P}_{E \otimes \R} = \PP(\mathcal{O} \otimes \R, \mathcal{O} \otimes \R, \iota \otimes \R)$.
 
\item[(iii)]  $\PPminus$, the groupoid of skew-Hermitian forms on $\mathcal{O}$ valued in the inverse different:

 Take  $\omega = \mathfrak{d}^{-1}$ the inverse different\footnote{
The inverse different $\mathfrak{d}^{-1}$ is, by definition,  
\begin{equation*}
\mathfrak{d}^{-1} = \{ y \in E \mid \text{$\mathrm{Tr}^E_\Q(xy) \in \Z$ for all $x \in \mathcal{O}$}\}.
\end{equation*}   
which is canonically isomorphic to $\Hom_\Z(\mathcal{O},\Z)$, with module structure defined by $(a.f)(x) = f(ax)$; under this identification the trace $\mathfrak{d}^{-1} \to \Z$ is sent to the functional on $\Hom_\Z(\mathcal{O},\Z)$ given by precomposition with $\Z \to \mathcal{O}$. } for $E$, with the {\em negated} conjugation involution $-\iota: z \mapsto -\bar{z}$
 and set $\PPminus = \PP(\mathcal{O}, \mathfrak{d}^{-1}, -\iota)$.

\item[(iv)]  $\mathcal{P}_{E \otimes \R}^{-}$, the groupoid of skew-Hermitian forms on $E \otimes \R$: 

 As in (iii), but tensoring with $\R$. \end{itemize}

Now, given $(L, b) \in \PPminus$, we shall write $L_\Z$ for the $\Z$-module underlying $L$.  It is a free $\Z$-module of rank $2g = \dim_\Q(E)$, and inherits a bilinear pairing
\begin{equation} \label{smpl}
  \begin{aligned}
    L_\Z \times L_\Z & \to \Z\\
    (x,y) & \mapsto - \mathrm{Tr}^E_\Q (b(x,y)).
  \end{aligned}
\end{equation}
This pairing is readily verified to be skew-symmetric and perfect (i.e.,\ the associated map $L_{\Z} \rightarrow L_{\Z}^{\vee}$ is an isomorphism)
so that 
associating to $(L,b) \in \PPminus$ the free $\Z$-module with the pairing above defines a \emph{functor} 
\begin{equation} \label{PPSP} \PPminus \to \mathcal{SP}(\Z).\end{equation} We shall return to this in Section~\ref{ssec:constr-abel-vari} below.

Finally, we comment on monoidal structure.
Unlike $\pic{R}$, we do not have a symmetric monoidal structure on $\PP(\mathcal{O},\omega,\iota)$ in general. 
However, if we take $\omega=\mathcal{O}$ equipped with the involution  on $\mathcal{O}$, then    $\PPplus = \PP(\mathcal{O},\mathcal{O},\iota)$ has the structure
of a symmetric monoidal groupoid, and more generally: 
 \begin{definition}
  Let $(\mathcal{O}, \omega,\iota)$ and $(\mathcal{O}, \omega',\iota')$ be as above (same underlying ring with involution, two different invertible modules with involution).  Define a functor
  \begin{equation}\label{eq:4}
    \PP(\mathcal{O},\omega,\iota) \times  \PP(\mathcal{O},\omega',\iota') \xrightarrow{\otimes} \PP(\mathcal{O},\omega \otimes_{\mathcal{O}} \omega',\iota \otimes \iota')
  \end{equation}
  as $(L,b) \otimes (L',b') = (L \otimes_{\mathcal{O}} L',b \otimes b')$, where $(b\otimes b')(x\otimes x',y \otimes y') = b(x,y) b(x',y')$.
\end{definition}

In particular this construction gives a symmetric monoidal structure on $\PPplus$ and an ``action'' bifunctor $\PPplus \times \PPminus \rightarrow \PPminus$.

\subsection{Construction of CM abelian varieties}
\label{ssec:constr-abel-vari}
Let $E$ be a CM field. 
We will now construct the map $\ST: \PPminus \to \mathcal{A}_g(\C)$
promised in 
 \S \ref{sec:abel-vari-sympl}. In fact this factors   the functor $\PPminus \to \mathcal{SP}(\Z)$ of \eqref{PPSP}.
 \begin{equation} \label{both compositions}
\PPminus \stackrel{\ST}{ \to} \mathcal{A}_g(\C) \xrightarrow{\eqref{eq:64}} \mathcal{SP}(\Z).
\end{equation} 

To construct the functor $\ST$, start with an object $(L, b) \in \PPminus$.
We shall equip
$ L_{\R}/L_{\Z}$
with the structure of a principally polarized abelian variety.
In order to do so it is necessary to specify, firstly, a complex structure $J$
on $L_{\R}$, and secondly a Hermitian form on $L_{\R}$
whose imaginary part is a perfect symplectic pairing $L_{\Z} \times L_{\Z} \rightarrow \Z$.
(This data can be used, as in \cite[Section I.2]{Mum08}, to construct
 an explicit ample line bundle on $L_{\R}/L_{\Z}$
 whose first Chern class is the specified symplectic pairing.)
 
 We begin by specifying the symplectic pairing: it is given by
 the expression of \eqref{smpl}, i.e.
\begin{equation} \label{simpl2}
  \begin{aligned}
    L_{\Z} \times L_\Z &\to \Z\\
    (x,y) & \mapsto - \mathrm{Tr}^{E}_{\Q} b(x,y).
  \end{aligned}
\end{equation}
(The sign is a purely a convention---the opposite convention would lead to other signs elsewhere, e.g.\ the inequality in~(\ref{PhiLbdef}) below would be the other way around.)  The definition of $\mathfrak{d}^{-1}$ makes this form $\Z$-valued and perfect, by the corresponding properties of $b$.  The real-linear extension of this symplectic form is the imaginary part of a Hermitian form on $L_\R$ in a complex structure; we specify this complex structure and Hermitian form next.

A {\em   CM type $\Phi$} for $E$ is, by definition, 
  a subset 
$\Phi \subset \Hom_{\mathrm{Rings}}(E,\C)$ with the property that the induced map
\begin{equation} \label{Phiiso} E \otimes \R \rightarrow \C^{\Phi}\end{equation}
is an isomorphism; equivalently, $\Phi$ contains 
 precisely one element in each conjugacy class $\{j, \overline{j}\}$.
 Such a $\Phi$ determines a complex structure on $L_{\R}$,
for \eqref{Phiiso} gives $E \otimes \R$ the structure of $\C$-algebra.
 
If $\Phi$ is a CM type, then
\begin{align*}
  \mathrm{Tr}^E_\Q b(x,y) 
                          & = 2\mathrm{Re} \bigg(\sum_{j \in \Phi} j(b(x,y)) \bigg),
\end{align*}
where we used that  $\mathrm{Tr}^E_\Q(x) = \sum_{j: E \to \C} j(x) \in \Q \subset \C$, where the sum is over all ring homomorphisms $E \to \C$.
 In particular, the function $  L_{\R} \times L_{\R} \rightarrow \C$  given by
\begin{equation} \label{bform} \langle x,y \rangle_b = -2i \bigg(\sum_{j \in \Phi} j(b(x,y)) \bigg),\end{equation} 
has \eqref{simpl2}  for  imaginary part. Moreover, $\langle -, - \rangle_b$
is Hermitian with respect to the complex structure on $L_{\R}$ induced by $\Phi$. 
Finally, $\langle -, - \rangle_b$   is positive definite precisely for the unique CM-type $\Phi = \Phi_{(L,b)}$, defined as
\begin{equation} \label{PhiLbdef} 
  \Phi_{(L,b)} := \{j : E \to \C \mid \text{$\mathrm{Im}(jb(x,x)) \geq 0$ for all $x \in L_\R$}\}
\end{equation}
i.e.\ the embeddings sending $b(x,x) \in E \otimes \R$ to the upper half-plane for all $x \in L_\R$.  We shall say that $\Phi_{(L,b)}$ is the CM structure on $E$ \emph{associated} to the object $(L,b) \in \PPminus$.  Evidently, it depends only on the image of $(L,b)$ under the base change functor
\begin{equation*}
\PPminus \xrightarrow{ - \otimes_\Z \R}  \mathcal{P}_{E \otimes \R}^{-}
\end{equation*}

To summarize, to $(L,b)$ we have associated:
\begin{itemize}
\item[-] a complex structure on $L_{\R}$
(the one induced from $\Phi_{(L,b)}$ via \eqref{Phiiso});
\item[-] a positive definite Hermitian form $\langle -, - \rangle_b$
  on this complex vector space; the imaginary part of this form 
  restricts to the symplectic form~(\ref{simpl2}).
\end{itemize}
The quotient $L_{\R}/L_{\Z}$ thus has the structure of a principally polarized
abelian variety over $\C$; we denote it by $\ST(L,b)$.

   The $\mathcal{O}$-module structure on $L$ gives a homomorphism $$\mathcal{O} \to \mathrm{End}\ \ST(L,b),$$ which is a homomorphism of rings with involution when the target is given the Rosatti involution induced by the polarization of $\ST(L,b)$.  Acting by an element of $a \in \mathcal{O}$ gives an endomorphism of $\ST(L,b)$, which will be an automorphism if $a \in \mathcal{O}^\times$, but not necessarily one that preserves the polarization: the polarization is given by a line bundle $\mathscr{L}$ on $\ST(L,b) \times \ST(L,b)$, and the correct statement is that 
   \begin{equation}\label{eq: respect polarization}
   (a,1)^*(\mathscr{L}) = (1, \overline{a})^*(\mathscr{L}).
   \end{equation}
 However, acting by an element of the subgroup $\{a \in \mathcal{O}^\times \mid x \overline{x} = 1\}$ does preserve the polarization.

The association $(L,b) \mapsto \ST(L,b)$ defines the desired functor
\begin{equation}\label{eq:1}
\ST: \PPminus \to \mathcal{A}_g(\C).
\end{equation}

\begin{remark} \label{equiv-of-cat} 
In fact, the association $(L,b) \mapsto \ST(L,b)$ can be made into an equivalence by modifying the target category. Namely, consider principally polarized abelian variety $A$ over the complex numbers
together with a map $\iota: \mathcal{O} \rightarrow \mathrm{End}(A)$ that respects the polarization in the sense of \eqref{eq: respect polarization}.
Such form a groupoid in an evident way (the morphisms
being isomorphisms of abelian varieties respecting polarization and $\mathcal{O}$-action);
call this groupoid $\mathcal{A}_g^{\mathcal{O}}(\C)$. 
The functor $\ST$ defined in \eqref{eq:1} factors through
$$ \PPminus \longrightarrow \mathcal{A}_g^{\mathcal{O}}(\C)$$
and this is an equivalence: an inverse functor 
sends $(A, \iota)$ to $(L, b)$, where
$L=H_1(A(\C); \Z)$, and
$b: L \times L \rightarrow \mathfrak{d}^{-1}$
is uniquely specified by the requirement that
$ - \mathrm{Tr}^{E}_{\Q} b(x,y)$ 
coincides with the skew-symmetric pairing on $L$
induced by the principal polarization. 
 \end{remark} 
 
\begin{remark}\label{rem:differ} 
As in (\ref{eq:4}), there is a tensor bifunctor $\PPminus \times \PPplus \rightarrow \PPminus$. 

  If $(X,q) \in \PP(\mathcal{O},\mathcal{O},\iota_+)$ is positive definite -- that is, $q(x,x) \in \mathcal{O}^+$ is totally positive for all $x \in X$ --  then this tensor operation can be described algebraically via ``Serre's tensor construction'' \cite{A-K}: if $(L',b') = (L,b) \otimes   (X,q) $ then 
\begin{equation} \label{Serretensor} \ST(L',b') \cong \ST(L,b) \otimes_\mathcal{O}  (X,q)  , \end{equation}
  the abelian variety representing the functor $R \mapsto  \Hom(\Spec(R),\ST(L,b))  \otimes_\mathcal{O} X$, equipped with a polarization induced by that of $A(L,b)$ and $q$. 

  If $q$ is not positive definite it seems difficult
  to give
  an explicit description such as \eqref{Serretensor}.    For example, tensoring with $(X,q) = (\mathcal{O},-1)$, where ``$-1$'' denotes the form $x\otimes y \mapsto -x\bar{y}$,  sends $A = \ST(L,b)$ to its ``complex conjugate'' variety $\overline{A}$. (In the discussion above,
  it replaces the CM type $\Phi_{(L,b)}$ with its complement.) 
\end{remark}

\subsection{Construction of enough objects of $\PPminus$ for a cyclotomic field}
\label{sec:constr-objects-ppoo}

We now specialize to the case when $ E = K_q  \subset \C$, the cyclotomic
field generated by the $q$th roots of unity.  We shall prove a slightly technical result
about the existence of enough objects in the groupoid $\PPqminus$;
this is the key setup   in our later verification (Proposition \ref{prop:generators-for-KSp})
that CM classes exhaust symplectic $K$-theory.

Recall that a CM structure on $\OO_q$, the ring of integers of $\Q(\mu_q)$, may be defined either as an $\R$-algebra homomorphism $\C \to \OO_q \otimes \R$, or as a set of embeddings $\OO_q \to \C$ containing precisely one element in each equivalence class $\{j ,\bar{j}\}$ under conjugation.  As in \eqref{PhiLbdef} each object $(L,b)$ 
 of the groupoid $\PPqminus$ picks out a CM type,
 which we denote as $\Phi_{(L,b)}$; explicitly, $(L \otimes \R, b \otimes \R)$
 is isomorphic to $\OO_q \otimes \R$ 
 with
Hermitian form given by $(x,y) \mapsto x u \bar{y}$ for some $u \in \OO_q \otimes \R$
purely imaginary, and the CM type is given by those embeddings for which the imaginary
part of $j(u)$ is positive.

\begin{proposition} \label{construction of CM structures}
  Let $\Phi$ be a CM structure on $\OO_q$ and let $L \in \pic{\OO_q}$.  Then there exist objects $(B_1,b_1)$ and $(B_2,b_2)$ of $\PPqminus$ such that
  \begin{enumerate}[(i)]
  \item\label{item:5} $[B_1][B_2]=[L][\overline{L}]^{-1} \in \pi_0(\pic{\OO_q})$,
  \item\label{item:6} $\Phi_{(B_1,b_1)} = \Phi_{(B_2,b_2)} = \Phi$.
  \end{enumerate}
\end{proposition}
\begin{proof}
  Let $\zeta_q = e^{2 \pi i/q} \in \OO_q$ as usual, and recall that the different $\mf{d} \subset \OO_q$ is principal and
  generated by\footnote{To justify this, see \cite[Proposition 2.7]{Wash} for the calculation of the discriminant, from which it's easy to deduce the statement about the different, using that $K_q/\Q$ is totally ramified over $p$.} $q/(\zeta_q^{q/p}-1)$.  The element $w = (1-\zeta^2)/(1-\zeta) = (1+\zeta)$ is a unit in $\OO_q$ and has the property that $\overline{w} = \zeta^{-1}_q w$.  If we set
  $$ \delta := w^{q/p} \frac{q}{\zeta_q^{q/p} - 1},$$
  it follows that $(\delta) = \mathfrak{d}$ and $\overline{\delta} = -\delta$, i.e.\ $\delta$ is purely imaginary.  The inverse different ideal $\mathfrak{d}^{-1} \subset K_q$ is therefore also principal, generated by the purely imaginary element $\delta^{-1}$. 
  
  It is now easy to satisfy~(\ref{item:5}): set
  \begin{alignat*}{2}
    B'_1 & = \OO_q &\quad b'_1(x,y) &= \delta^{-1} x \overline{y}\\
    B'_2 & = L \otimes \overline{L}^{-1} &\quad b'_2(x \otimes \phi, y \otimes \psi) &= \delta^{-1} \cdot \psi(x) \cdot \phi(y)
  \end{alignat*}
  where in the first line $x \in \OO_q$ and $y \in \ol{\OO}_q$, and in the second line   $x \otimes \phi \in B'_2 = L \otimes \overline{L}^{-1}$ and $y \otimes \psi \in \overline{B'_2} \cong \overline{L} \otimes L^{-1}$ (and the evaluation pairing between $L$ and $L^{-1}$ comes from viewing $L^{-1}$ as the dual of $L$).
   It is clear that the pairings $B'_i \otimes_{\OO_q} \overline{B}'_i \to K_q$ defined by the two formulae give isomorphisms onto $(\delta^{-1}) = \omega_{\OO_q}$, and the fact that $\delta$ is totally imaginary implies that $b_i(x,y) = - \overline{b_i(y,x)}$, so that we indeed have two objects $(B_i,b_i) \in \PPqminus$.  It is also obvious that $[B'_1][B'_2] = [L][\overline{L}]^{-1} \in \pi_0 \pic{\OO_q}$.  These objects do not necessarily satisfy~(\ref{item:6}) though: any complex embedding $j: K_q \to \C$ will take $b_1(x,x)$ and $b_2(x,x)$ to a non-negative real multiple of the imaginary number $j(\delta^{-1})$, so in fact $\Phi_{(B'_1,b'_1)} = \Phi_{(B'_2,b'_2)} = \Phi_0$, where
  \begin{equation*}
    \Phi_0 = \{j: K_q \to \C \mid \mathrm{Im}(j(\delta^{-1})) > 0\}.
  \end{equation*}

  To realize other CM structures we shall use the tensor product~\eqref{eq:4} and set
  \begin{align*}
    (B_1,b_1) & = (B'_1,b'_1) \otimes (X,q)\\
    (B_2,b_2) & = (B'_2,b'_2) \otimes (X,q)^{-1}
  \end{align*}
  for a suitable object $(X,q) \in \mathcal{P}_{K_q}^+$.
  We shall choose $(X,q)$ using the following Lemma:
  
  \begin{lemma}\label{lem: signs}
    Let $\OO_q^+ \subset K_q^+$ denote the ring of integers in $K_q^+ = \Q[\cos(2\pi/q)]$, the maximal totally real subfield of $K_q$, and let $S = \Hom(\OO_q^+,\R)$ be the set of real embeddings of $\OO_q$.  For any function $f: S \to \{\pm 1\}$ there exists a non-zero prime element $t \in \OO_q^+$ such that
    \begin{itemize}
    \item $\mathrm{sgn}(j(t)) = f(j)$ for all $j \in S$,
    \item $t \OO_q = \mathfrak{x} \overline{\mathfrak{x}}$ for a (prime) ideal $\mathfrak{x} \subset \OO_q$.
    \end{itemize}
  \end{lemma}
  We give the proof of the lemma below, but let us first explain why it permits us to conclude the proof. Let $X$ be the $\OO_q$-module underlying $\mathfrak{x}$ and define a sesquilinear pairing on $X$ by
  \begin{equation*}
    q(x,y) = t^{-1} x \overline{y}.
  \end{equation*}
  This defines an isomorphism $q: X \otimes_{\OO_q} \overline{X} \to \OO_q$, and $q(x,y) = \overline{q(y,x)}$ since $t$ is totally real.
  Hence we have an object $(X,q) \in \PPqplus$.
  Now the difference between
  $ \Phi_{(L, b) \otimes (X,q)}$ and $\Phi_{(L, b)}$  
  is precisely determined by the signs of $t$ under the real embeddings of $K_q^+$, which are controlled by the function $f$ in the lemma, which may be arbitrary.
\end{proof}
    
\begin{proof}[Proof of Lemma \ref{lem: signs}]
  This will be a consequence of the Chebotarev density theorem in algebraic number theory, which
  produces a prime ideal with a specified splitting behavior in a field extension; for us the extension is $H_q^+ K_q/K_q$,
  where $H_q^+$ is the narrow Hilbert class field of $K_q^+$, that is, the largest abelian extension of $K_q^+$
  that is unramified at all finite primes.
  
  Restriction defines an isomorphism
  \begin{equation} \label{GalR}
    \Gal(H_q^+ K_q/K_q^+) \stackrel{\sim}{\longrightarrow} \Gal(H_q^+/K_q^+) \times \Gal(K_q/K_q^+),
  \end{equation} 
(the map is surjective because $K_q/K_q^+$ is totally ramified at the unique prime above $q$ and $H_q^+/K_q^+$ is unramified,
so the inertia group at $q$ maps trivially to the first factor and surjects to the second factor).
Now class field theory defines an isomorphism
\begin{equation} \label{ICG}
  \mathrm{Art}: \frac{  \{\pm 1\}^{S} \times \mbox{fractional ideals}}{\mbox{principal signed ideals}}  \stackrel{\sim}{\longrightarrow} \Gal(H_q^+/K_q^+)
\end{equation} 
where the principal signed ideals are elements of the form $(\mathrm{sign}(\lambda), \lambda)$
for  $\lambda$   a nonzero element of $K_q^+$.  The map from left to right
is the Artin map on fractional ideals, and sends the $-1$ factor indexed by $j \in S$
to the complex conjugation above $j$. 

 By the Chebotarev density theorem, there exists a prime ideal $\mathfrak{t}$
 of $K_q^+$ whose image under \eqref{GalR} 
 is trivial in the second factor, and, in the first
 factor, coincides with $\mathrm{Art} (f \times \mathrm{trivial})$. 
Triviality in the second factor  forces  $\mathfrak{t}$ to be split in $K_q/K_q^+$;
the condition on the first factor forces $\mathfrak{t} = t \mathcal{O}_q^+$
where    the sign of $j(t)$ is given by $f(j)$, for each $j \in S$. 
 \end{proof}


\newcommand{\cyclicq}{\Z/q}
\newcommand{\Bott}{\mathrm{Bott}}

\section{CM classes exhaust symplectic $K$-theory} \label{CMExhaust}

The primary goal of this section is to verify that
the construction of classes
in symplectic $K$-theory sketched in \S  \ref{sec:abel-vari-sympl}
in fact
produces all of symplectic $K$-theory in the degrees of interest.

In more detail: we have  constructed a sequence \eqref{both compositions}
 $\PPminus \stackrel{\ST}{ \to} \mathcal{A}_g(\C) \to \mathcal{SP}(\Z)$
 associated to a CM field $E$; the functor $\ST$
 produces a CM abelian variety from a skew-Hermitian module over the ring of integers of $E$. 
  There are induced maps of spaces $|\PPminus| \to |\mathcal{A}_g| \to |\mathcal{SP}(\Z)| \to \Omega^\infty \KSp(\Z)$, where the last map is the group completion map.  By adjunction there are associated map of spectra
\begin{equation}\label{eq:main-map-to-KSp}
  \Sigma^\infty_+ |\PPminus| \to \Sigma^\infty_+ |\mathcal{A}_g(\C)| \to \KSp(\Z).
\end{equation} 
   {\em We emphasize that $\mathcal{A}_g(\C)$
is a discretely topologized groupoid, that is to say,  the topology on $\C$ plays no role.} 
This makes the middle term of \eqref{eq:main-map-to-KSp} rather huge.
In this section we show that the composition of \eqref{eq:main-map-to-KSp}
is surjective on homotopy, in the degrees of interest:

\ \begin{proposition}\label{prop:generators-for-KSp}
  Take $E = K_q$, the cyclotomic field. The composition
  \begin{equation} \label{KqexhaustsKSp}
    \pi_{4k-2}^s(|\PPqminus|;\Z/q) \to
    \KSp_{4k-2}(\Z;\Z/q).
  \end{equation}
  is surjective for all $k \geq 1$.  \end{proposition}

More precisely, we show that a certain natural supply of classes in the source already surject on the target. 
    All objects $(L,b) \in\PPminus$ have automorphism group the unitary group $U_1(\mathcal{O}) = \{x \in \mathcal{O} \mid x \overline{x} = 1\}$, so we get a homotopy equivalence
$|\PPminus| \simeq BU_1(\mathcal{O}) \times \pi_0(\PPminus)$ and since stable homotopy takes disjoint union to direct sum we get isomorphisms analogous to~(\ref{eq:8}) \begin{equation}\label{eq:2}
  \begin{aligned}
    \pi_*^s( |\PPminus|) &\cong \pi_*^s(BU_1(\mathcal{O})) \otimes \Z[\pi_0(\PPminus)]\\
    \pi_*^s( |\PPminus|;\Z/q) & \cong \pi_*^s(BU_1(\mathcal{O});\Z/q) \otimes \Z[\pi_0(\PPminus)]\\
\end{aligned}
\end{equation}
In the case $E=K_q$  with ring of integers $\mathcal{O}_q$, 
we get  a map
  $\Z/q  \to \mathcal{O}_q^{\times}$ sending $a$ to $e^{2 \pi i a/q}$, 
and thereby
$$ \pi_2^s(B(\Z/q);  \Z/q) \rightarrow  \pi_2^s(BU_1(\mathcal{O}_q);\Z/q),$$
The left-hand side contains a distinguished ``Bott element'' $\beta$, which generates a polynomial algebra in $\pi_2^s(B(\Z/q);  \Z/q) $, 
as discussed in   \eqref{Bottcyclic}. We denote by the same
letter its image inside the right-hand side. 

What we shall  show, in fact, is that 
  elements of the form
$  \beta^{2k-1} \otimes [(L, b)] \in \pi_{4k-2}^{s}(|\PPqminus|, \Z/q)$, 
with $(L, b) \in \pi_0 \PPqminus$,
generate the image of \eqref{KqexhaustsKSp}.
 To show this,   we use Theorem  \ref{thm:KSp-of-Z},
which provides a sufficient supply of maps {\em out of} $\KSp$,
namely the Hodge map $c_H$ and the Betti map $c_B$. 
In \S \ref{subsec:Hodge}, we compute    
$c_H \circ $(\ref{eq:main-map-to-KSp}),
and in \S \ref{subsec:Soule} we compute 
$c_B \circ $(\ref{eq:main-map-to-KSp}).
We them assemble the results in the final section \S \ref{sec:surjectivity}.  

\subsection{Hodge map for CM abelian varieties} \label{subsec:Hodge}

We first describe the composition
\begin{equation*}
  B\cyclicq \times \pi_0(  \PPqminus ) \simeq |\PPqminus| \to \Omega^\infty \KSp(\Z) \xrightarrow{c_H} \Z \times BU,
\end{equation*}
which is most conveniently expressed one path component at a time.

\subsubsection{Reminders on the Hodge map} Recall that the Hodge map $\KSp(\Z) \to ku$ arose from a zig-zag of functors $\SP(\Z) \to \SP(\R^\mathrm{top}) \xleftarrow{\simeq} \Utop$, as in \eqref{Hodgezigzag}.  Understanding the Hodge map $\KSp(\Z) \to ku$ therefore involves inverting the weak equivalence, which informally amounts reducing a structure group from $\Sp_{2g}(\R)$ to $U(g)$.  Roughly speaking, for a symplectic real vector space we must choose compatible complex structures and Hermitian metrics with the given symplectic form as imaginary part.  

\subsubsection{Computation of the Hodge map for $\PPqminus$} For the symplectic vector spaces arising from objects $(L,b) \in\PPminus$ by the construction in \S\ref{ssec:constr-abel-vari} above we already produced such a choice.  Indeed, the Hermitian inner product $\langle -,-\rangle_b$ from \eqref{bform} and the CM structure $\Phi_{(L,b)}$ on $E$ induces exactly this structure on $L_{\R} = L \otimes \R$.  This observation gives the diagonal arrow in the following diagram
\begin{equation*}
  \xymatrix{
    \PPqminus \ar[r] \ar[d] & \SP(\Z) \ar[d]\\
    \mathcal{P}_{K_q \otimes \R}^{-} \ar[r] \ar[dr] & \SP(\R^\mathrm{top})\\
    & \Utop \ar[u]_{\simeq}.
  }
\end{equation*}

Restricting the composition $\PPqminus \to \Utop$ to the object $(L,b)$ and its automorphism group $\mu_q = \mathrm{Aut}(L,b)$, we may describe the composition
\begin{equation}\label{eq:34}
  \{(L,b)\} \moddd \mu_q  \hookrightarrow \PPqminus \to \Utop,
\end{equation}
(where $\{(L,b)\} \moddd  \mu_q $ is shorthand for the full sub-groupoid of $\PPqminus$ on the object $(L,b)$) as follows. Giving a functor $\{(L,b)\} \moddd \mu_q   \to \Utop$ is equivalent to giving a unitary representation of $\mu_q$, and in these terms the composition \eqref{eq:34} corresponds to the unitary representation
\begin{equation}\label{eq:34ii}
  \bigoplus_{j \in \Phi_{(L,b)}} j\vert_{\mu_q}
\end{equation}
where we recall that $\Phi_{L,b}$ consists of various complex embeddings $K_q \hookrightarrow \C$,
and we may therefore regard  each restriction $j\vert_{\mu_q}: \mu_q \to U_1(\C) \subset \C^\times$ as a 1-dimensional (unitary) representation of $\mu_q$.
 The CM structure $\Phi_{(L,b)}$ depends only on the image of $(L,b)$ under the base change functor
 $
  \PPqminus \xrightarrow{- \otimes_\Z \R} \PP_{K_q \otimes \R}^{-}
 $
so the same is true for the functor~(\ref{eq:34}), up to natural isomorphism.

Finally, we use this discussion to compute the image of Bott elements under the Hodge map. The embeddings $\mathcal{O} \rightarrow \C$ are parameterized by $s \in (\Z/q)^{\times}$:
the $s$th embedding $j_s$ satisfies $j_s(e^{2 \pi i/q}) =  e^{2\pi is/q}$.
As discussed in \S \ref{Bottcyclic}, 
$j_1$ induces a homomorphism of graded rings $$(j_1)_*: \pi_*^s(B\cyclicq; \Z/q) \rightarrow \pi_*(ku, \Z/q),$$
and this sends the Bott element $\beta \in \pi_2^s((B\cyclicq);\Z/q)$ to the mod $q$ reduction of the usual Bott element -- we denote this by $\Bott$. The powers of $\Bott$ generate the mod $q$ homotopy groups of $ku$.  More generally we have
  $$(j_a)_*(\beta) = a \cdot  \Bott \in \pi_2(ku;\Z/q),$$
 and in particular  $(j_a)_*(\beta^i) = a^i \cdot (j_1)_*(\beta^i) = a^i \cdot \Bott^i$ for any $a \in (\Z/q)^\times$ (cf. Remark \ref{rem: cyc action on bott}).   
  Combining with \eqref{eq:34ii} we arrive at the following formula:

\begin{proposition}\label{prop: hodge map surjects}
As above, take $(L, b) \in \pi_0 \PPqminus$,
giving a class $\beta^i [L, b] \in \pi^s_{2i}(|\PPqminus|; \Z/q)$. 
The image of $\beta^i [L, b]$ under  the map of homotopy groups induced by (cf.  \eqref{eq:main-map-to-KSp})
 $ \Sigma^\infty_+ |\PPqminus| \to \Sigma^\infty_+ |\mathcal{A}_g(\C)| \to \KSp(\Z) \to ku$
  is given by
\begin{equation}\label{eq:35}
\bigg(\sum_{a \in (\Z/q)^*: j_a \in \Phi} a^i \bigg) \Bott^i \in \pi_{2i}(ku;\Z/q).
\end{equation}
Moreover,   
for any odd $i$ there exists a CM structure $\Phi$ on $K_q = \mathcal{O}_q \otimes \Q$ for which the element~(\ref{eq:35}) is a generator for $\pi_{2i}(ku;\Z/q) \cong \Z/q$.
\end{proposition}

\begin{proof}
The previous discussion already established \eqref{eq:35}, so we turn our attention to the last assertion. Since $\Bott^i$ generates,  we must find a CM structure satisfying
  \begin{equation*}
    \sum_{a \in (\Z/q)^*: j_a \in \Phi} a^i \in (\Z/q)^\times.
  \end{equation*}
  Equivalently, we must find a subset $X \subset (\Z/q)^\times$ containing precisely one element from each subset $\{a, -a\} \subset (\Z/q)^\times$, such that
  \begin{equation*}
    \sum_{a \in X} a^i \in (\Z/q)^\times.
  \end{equation*}
Choose such a set $X$ arbitrarily, and let $X'$ be  obtained
from $X$ by switching the element in which $X$ intersects $\{1, -1\}$. 
Then $\sum_{a \in X} a^i$ and $\sum_{a \in X'} a^i$
differ by $(-1)^i - 1 = -2$, and so at least one is a unit in $\Z/q$. 
\end{proof}

\begin{remark}
  For $p=2$ it seems a similar argument shows that there exists a $\Phi$ for which \eqref{eq:35} is twice a generator.
\end{remark}

\subsection{Betti map for CM abelian varieties} \label{subsec:Soule} 

Next we treat the composition of~(\ref{eq:main-map-to-KSp}) with the Betti map. The map $\PP \to \mathcal{A}_g(\C)$ sends $(L,b)$ to an abelian variety with underlying space $A = L_\R/L_\Z$, from which we read off $H_1(A;\Z) = L_\Z$.  This implies a diagram of functors
\begin{equation*}
  \xymatrix{
   \PPminus \ar[rr] \ar[d] & & \SP(\Z) \ar[d]\\
    \mathrm{Pic}(\mathcal{O}) \ar@{^{(}->}[r] & \proj{\mathcal{O}} \ar[r]^-{\text{forget}} & \proj{\Z},
    }
\end{equation*}
commuting up to natural isomorphism, where the vertical maps are induced by forgetting the pairings, i.e., $(L,b) \mapsto L$.  Passing to the associated spaces and composing with group completion maps we get a diagram of spectra
\begin{equation}\label{eq:32}
  \begin{aligned}
  \xymatrix{
    \Sigma^\infty_+|\PPminus| \ar[rr] \ar[d] & & \KSp(\Z) \ar[d]\\
    \Sigma^\infty_+ |\pic{\mathcal{O}}| \ar[r] & K(\mathcal{O}) \ar[r]^{\mathrm{tr}} & K(\Z),
    }
  \end{aligned}
\end{equation}
where the map $\mathrm{tr}: K(\mathcal{O}) \to K(\Z)$ is the ``transfer map'' induced by the functor $\proj{\mathcal{O}} \to \proj{\Z}$ sending a projective $\mathcal{O}$-module to its underlying (projective) $\Z$-module.

As explained in Section~\ref{def:pic}, the homotopy groups of the spectrum in the lower left corner are $\pi_*^s(|\pic{\mathcal{O}}|) = \pi_*^s(B \mathcal{O}^\times) \otimes \Z[\pi_0(\pic{\mathcal{O})}]$ and with mod $q$ coefficients $\pi_*^s(B \mathcal{O}^{\times};\Z/q) \otimes \Z[\pi_0(\pic{\mathcal{O}})]$, similar to~(\ref{eq:2}).
Commutativity of the induced diagram on homotopy groups gives the following,
after we specialize to the case of $E=K_q, \mathcal{O} = \mathcal{O}_q$:   
\begin{corollary}\label{cor:Soule-formula} 
Notation as above.   
The composition
  \begin{equation*}
    \pi^s_{4k-2}(|\PPminus|;\Z/q) \to \KSp_{4k-2}(\Z;\Z/q) \xrightarrow{c_B} K_{4k-2}(\Z;\Z/q)^{(+)}
  \end{equation*}
  sends the element $\beta^{2k-1} \cdot [(L,b)]$ (in the notation of~(\ref{eq:2})) to the element $\mathrm{tr}(\beta^{2k-1} \cdot ([L]-1)) \in K_{4k-2}(\Z;\Z/q)^{(+)}$.  Here $[(L,b)] \in \pi_0(\PPqminus)$ is any element, and $[L]-1 \in K_0(\OO)$ is the projective class associated to $[L] \in \pic{\OO}$.\qed
\end{corollary}
\begin{proof}
  Commutativity of the diagram \eqref{eq:32} yields
  \begin{equation*}
    \beta^{2k-1} \cdot [(L,b)] \mapsto \mathrm{tr}(\beta^{2k-1} \cdot [L]) = \mathrm{tr}(\beta^{2k-1} \cdot ([L]-1)) + \mathrm{tr}(\beta^{2k-1}).
  \end{equation*}
  Now  $\mathrm{tr}(\beta^{2k-1} \cdot ([L]-1)) \in K_{4k-2}(\Z;\Z/q)^{(+)}$ and $\mathrm{tr}(\beta^{2k-1}) \in K_{4k-2}(\Z;\Z/q)^{(-)}$, cf.\ Remark~\ref{rem:transfer-Adams}.  Since the image of $c_B: \KSp_{4k-2}(\Z;\Z/q) \to K_{4k-2}(\Z;\Z/q)$ is contained in the $(+1)$-eigenspace, commutativity of the diagram \eqref{eq:32} implies that $\mathrm{tr}(\beta^{2k-1}) = 0$, which proves the claim.
\end{proof}

\begin{remark}
  Alternatively, the vanishing of $\mathrm{tr}(\beta^{2k-1}) \in K_{4k-2}(\Z;\Z/q)$ may be seen by identifying $\mathrm{tr}$ with the transfer map in \'etale cohomology
  \begin{equation*}
    H^0(\mathcal{O}'_q;\mu_q^{2k-1}) \xrightarrow{\mathrm{tr}} H^0(\Z';\mu_q^{2k-1}),
  \end{equation*}
  which sends $\beta^{2k-1} \in \mu_q(\mathcal{O}'_q)^{\otimes (2k-1)}$ to the sum of all its Galois translates.  This vanishes for the same reason as
  \begin{equation*}
    \sum_{a \in (\Z/q)^\times} a^i = 0 \in \Z/q
  \end{equation*}
  when $p-1$ does not divide $i$, and in particular for any odd $i$.
\end{remark}

\subsection{Surjectivity}
\label{sec:surjectivity}
 Recall that in
Theorem~\ref{thm:KSp-of-Z} we proved that the combination of the Hodge and Betti maps define an
isomorphism 
\begin{equation}\label{eq: hodge-soule defn}
  \KSp_{4k-2}(\Z;\Z/q) \xrightarrow{(c_H,c_B)} \pi_{4k-2}(ku;\Z/q)^{(-)} \times K_{4k-2}(\Z;\Z/q)^{(+)}.
\end{equation}

\begin{proof}[Proof of Proposition \ref{prop:generators-for-KSp}] 
 The coordinates  of $\beta^{2k-1} [L,b]$, under the map above,  
 have been computed  in Proposition \ref{prop: hodge map surjects}
 and Corollary~\ref{cor:Soule-formula}.
 They are given by: 
 $$ \begin{cases} \mbox{Hodge: }   c_H(\beta^{2k-1} \cdot [L,b]) =
   \Bott^{2k-1} \sum_{a \in \Phi_{L,b}} a^{2k-1}  & \in \pi_{4k-2}(ku; \Z/q),  \\
  \mbox{Betti: }   
  c_B(\beta^{2k-1} \cdot [L,b])  =
  \mathrm{tr}(\beta^{2k-1}\cdot ([L]-1))  & \in K_{4k-2}(\Z;\Z/q).
  \end{cases} 
  $$
 By Proposition \ref{prop: hodge map surjects},
 there exists a CM structure $\Phi$ for which $\sum_{a \in \Phi} a^{2k-1}$
 is invertible in $(\Z/q)$. It therefore suffices to prove that for \emph{any} CM structure $\Phi_0$,
  \begin{equation} \label{Phigen} \{ \mathrm{tr}(\beta^{2k-1}\cdot ([L]-1)) \mid \Phi_{(L,b)} = \Phi_0\} \mbox{ generates } K_{4k-2}(\Z;\Z/q)^{(+)},
  \end{equation}
  which is what we shall do.
  
For any $[L] \in \pic{\mathcal{O}_q}$, there exist by
 Proposition \ref{construction of CM structures} two objects $(L_1,b_1), (L_2,b_2) \in \PPqminus$ satisfying $\Phi_{(L_1, b_1)} = \Phi_{(L_2, b_2)} = \Phi_0$, and whose images in $\pi_0(\pic{\mathcal{O}_q})$ satisfy 
  \[
  [L_1] [L_2] = [L] [\overline{L}]^{-1}.
  \]
 The corresponding elements in $K_0(\mathcal{O}_q)$ then satisfy $[L_1] + [L_2] = [L] - [\overline{L}] + 2$.
 Applying the same Proposition with $[L] = 1 = [\mathcal{O}_q]$ gives $(L_3,b_3), (L_4,b_4) \in \PPqminus$ with $[L_3] + [L_4] = 2 \in K_0(\mathcal{O}_q)$.
  We then have
  \begin{align*}
    {} & {}\mathrm{tr}\big(\beta^{2k-1} \cdot ([L_1] + [L_2] - [L_3] - [L_4])\big) \\
    ={} & {}
      \mathrm{tr}\big(\beta^{2k-1} \cdot ([L]-1))- \mathrm{tr}( \beta^{2k-1} \cdot([\overline{L}] - 1)\big)\\
    =
    {} & {}\mathrm{tr}\big(\beta^{2k-1} \cdot ([L]-1)\big) + \mathrm{tr}\big(\beta^{2k-1} \cdot ([L]-1)\big) = 2\mathrm{tr}\big(\beta^{2k-1} \cdot ([L]-1)\big)
  \end{align*}
  where the last line used that the automorphism of $K(\mathcal{O}_q)$ induced by the involution on $\mathcal{O}_q$ sends $\beta \mapsto -\beta$ and $[L] \mapsto [\overline{L}]$, and that the transfer map is invariant under this automorphism (as the underlying $\Z$-modules of $M$ and $\overline{M}$ are equal).

  Proposition \ref{shtuka} implies that the elements $\mathrm{tr}\big(\beta^{2k-1} \cdot ([L]-1)) \in K_{4k-2}(\Z;\Z/q)^{(+)}$ generate as $[L]$ range over all of $\pi_0(\pic{\OO_q})$, and since $q$ is odd the factor of 2 does not matter for surjectivity.
 \end{proof}

\begin{remark}
  The method used here to produce elements of $\KSp_{4k-2}(\Z;\Z/q)$ is very similar to the method used by Soul\'e \cite{Soule81} to produce elements in algebraic $K$-theory of rings of integers.  In our notation the elements he constructs in $K_{4k+1}(\Z;\Z/q)^{(-)}$ are of the form $\mathrm{tr}(\beta^{2k} \cdot u)$ with $u \in \mathcal{O}_q^\times/q = K_1(\mathcal{O}_q;\Z/q)^{(-)}$.  By a compactness argument he lifts his elements from the mod $q = p^n$ theory to the $p$-adic groups, which can also be done here.

  Related ideas were also used by Harris and Segal \cite{HarrisSegal}.
\end{remark}
 

\section{The Galois action on $\KSp$ and on CM abelian varieties}\label{sec: galois action}

Now that we understand the abstract $(\Z/q)$-module $\KSp_{4k-2}(\Z;\Z/q)$ and how to produce elements in it, we will study the Galois action on it.  
The first task is to define the action. We give the construction in
\S \ref{galois action construction}. In \S \ref{sec:main-thm} we compute the action of the $\Aut(\C)$ on CM classes. 

\subsection{Galois conjugation of complex varieties} \label{galoisvariety}
Given a $\C$-scheme $X$, we obtain 
a $\C$-scheme $\sigma X$ by ``applying $\sigma$ to all the
coefficients of the equations defining $X$.''  More formally, we are given a pair $(X,\phi)$ consisting of an underlying scheme $X$ and a reference map $\phi: X \to \Spec(\C)$, and we define
\begin{equation*}
  \sigma (X,\phi) = (X,\Spec(\sigma^{-1}) \circ \phi),
\end{equation*}
i.e.\ we simply postcompose the reference map with the map $\Spec(\sigma^{-1}): \Spec(\C) \to \Spec(\C)$ while the underlying schemes are equal (not just isomorphic).  The resulting $\C$-scheme $\sigma(X,\phi) =: (\sigma X, \sigma \phi)$ fits in a cartesian square
\begin{equation}\label{eq:65}
  \begin{tikzcd}
    \sigma X \ar[d, "\sigma \phi"]  \ar[r, "\Id"] & X \ar[d, "\phi"] \\
    \Spec \CC \ar[r, "\sigma"] & \Spec \CC.
  \end{tikzcd}
\end{equation}
The rule $(X, \phi) \mapsto (\sigma X, \sigma \phi) $ extends to a functor from $\C$-schemes to $\C$-schemes in an evident way.

Applying this construction when $X = A \to \Spec(\C)$ is a complex abelian variety gives a new complex abelian variety, 
which inherits a principal polarization from that of $A$.  We arrive at a functor
\begin{equation*}
  \mathcal{A}_g(\C) \xrightarrow{\sigma} \mathcal{A}_g(\C),
\end{equation*}
which agrees up to natural isomorphism with applying the ``functor of points'' $\mathcal{A}_g$ to $\Spec(\sigma): \Spec(\C) \to \Spec(\C)$, because coordinates on $\mathcal{A}_g$ are coefficients of the equations defining the abelian varieties (e.g.\ using the Hilbert scheme atlas on $\mathcal{A}_g$ as in \cite[Section 6]{Mum94}).
In this way we get an action\footnote{We prefer not to take the cartesian square~\eqref{eq:65} as the definition of $\sigma (X,\phi)$: with our definitions $(\sigma \circ \sigma') (X,\phi)$ is \emph{equal} to $\sigma(\sigma' (X,\phi))$, which ensures we get an actual action on $|\mathcal{A}_g(\C)|$.  This issue is mostly cosmetic, and could presumably alternatively be handled by ``keeping track of higher homotopies''.} of $\Aut(\C)$ on the groupoid $\mathcal{A}_g(\C)$ and hence on the space $|\mathcal{A}_g(\C)|$.

\subsection{Construction of the Galois action on homotopy of $\KSp$} \label{galois action construction}

Recall from Section~\ref{sec:abel-vari-sympl} that we consider the functor $\mathcal{A}_g(\C) \to \SP(\Z)$ induced by sending a principally polarized abelian variety $A$ to $\pi_1(A(\C)^\mathrm{an},e)$, equipped with the symplectic form induced from the polarization.
We emphasize that we here regard $\mathcal{A}_g(\C)$ as just a groupoid in sets, so the domain of this spectrum map is rather huge: for example $\pi_0^s(|\mathcal{A}_g(\C)|)$ is the free abelian group generated by $\pi_0(|\mathcal{A}_g(\C)|)$, the (uncountable) set of isomorphism classes of complex principally polarized abelian varieties.
\begin{proposition}\label{prop: disc to KSp surjective}
  For all $k \geq 1$ and odd $q = p^n$, the map
  \begin{equation*}
    \pi^s_{4k-2}(|\mathcal{A}_g(\C)|;\Z/q) \to \KSp_{4k-2}(\Z;\Z/q)
  \end{equation*}
  induced by~(\ref{eq:51}) is surjective, when $g \geq \varphi(q) = p^{n-1}(p-1)$.
\end{proposition}
\begin{proof}
  It suffices to consider $g = \phi(q)$ since otherwise we may use any 
  $A_0 \in \mathcal{A}_{g - \phi(q)}(\Q)$ to define a map $A_0 \times-: \mathcal{A}_{\phi(q)} \to \mathcal{A}_g$.
  We consider the spectrum maps  of \eqref{eq:main-map-to-KSp} 
  \begin{equation*}
    \Sigma^\infty_+ |\PPqminus| \to \Sigma^\infty_+| \mathcal{A}_g(\C)| \to \KSp(\Z).
  \end{equation*}
  Since the composition induces a surjection on mod $q$ stable homotopy, by Proposition \ref{prop:generators-for-KSp}, the same must be true for the second map alone.
\end{proof}

 As in \S \ref{galoisvariety} any $\sigma \in \Aut(\C)$ induces a functor $\mathcal{A}_g(\C) \to \mathcal{A}_g(\C)$ and hence an automorphism of the spectrum $\Sigma^\infty_+|\mathcal{A}_g(\C)|$ and in turn an action of $\mathrm{Aut}(\C)$ on $\pi_*^s(|\mathcal{A}_g(\C)|;\Z/q)$.  The following proposition
 characterizes  the Galois action on symplectic $K$-theory.
\begin{proposition}\label{prop:5.2}
  For any $k \geq 1$ and odd prime power $q = p^n$, there is a unique action of $\Aut(\C)$ on $\KSp_{4k-2}(\Z;\Z/q)$ for which the homomorphisms
  \begin{equation*}
    \pi_{4k-2}^s(|\mathcal{A}_g(\C)|;\Z/q) \to \KSp_{4k-2}(\Z;\Z/q)
  \end{equation*}
  are equivariant for all $g$.
\end{proposition}
\begin{proof}[Proof sketch]
  We have seen that these homomorphisms are surjective for sufficiently large $g$, so for $\sigma \in \Aut(\C)$ there is at most one homomorphism
  \begin{equation}
    \label{eq:12}
    \begin{aligned}
    \xymatrix{
      \pi_{4k-2}^s(|\mathcal{A}_g(\C)|;\Z/q) \ar[r]\ar[d]^{\sigma_*} & \KSp_{4k-2}(\Z;\Z/q) \ar@{-->}[d]^{\sigma_*}\\
      \pi_{4k-2}^s(|\mathcal{A}_g(\C)|;\Z/q) \ar[r] & \KSp_{4k-2}(\Z;\Z/q)
      }
    \end{aligned}
  \end{equation}
  making the diagram commute.  If these exist for all $\sigma$, uniqueness guarantees that composition is preserved, inducing an action.  It remains to see existence;
  this proof is somewhat technical and is given in the Appendix. 
   \end{proof}

\begin{remark}
  The spectrum level action constructed in the Appendix 
  should probably be viewed as more intrinsic than the particular statement of the proposition.  From an expositional point of view, the main advantage of the statement of the proposition is that it uniquely characterizes the action on homotopy groups which we are studying, at least in degrees 2 mod 4, while not making explicit reference to \'etale homotopy type.  This allows us to quarantine the fairly technical theory of \'etale homotopy type to the proof of Proposition~\ref{prop:5.2}.

  It will also be clear from the spectrum level construction that the actions of $\Aut(\C)$ on $\KSp_{4k-2}(\Z;\Z/p^n)$ are compatible over varying $n$, including in the inverse limit $n \to \infty$, so that the universal property for each $n$ also determines the action on the $p$-complete symplectic $K$-theory groups $\KSp_{4k-2}(\Z;\Z_p)$.  The spectrum level action also induces an action on homotopy groups in degrees $4k-1$, which by Corollary \ref{cor: explicit determination of KSp} is the only other interesting case when $p$ is odd.  In Subsection~\ref{sec:degree-4k-1} we prove that the action on $\KSp_{4k-1}(\Z;\Z_p)$ is trivial.
\end{remark}

\begin{lemma}\label{lem:Soule-complex-conjugation}
  The Betti map
  \begin{equation*}
    \KSp_{4k-2}(\Z;\Z/q) \xrightarrow{c_B} K_{4k-2}(\Z;\Z/q)
  \end{equation*}
  is equivariant for the subgroup $\langle c\rangle \subset \Aut(\C)$, where $c$ denotes complex conjugation, and $K_{4k-2}(\Z;\Z/q)$ is given the trivial action.
\end{lemma}
\begin{proof}
  The composite $\pi_{4k-2}^s(|\mathcal{A}_g(\C)|;\Z/q) \to \KSp_{4k-2}(\Z;\Z/q)$ is induced from the functor $\mathcal{A}_g(\C) \to \proj{\Z}$ sending an abelian variety $A \to \Spec(\C)$ to $H_1(A(\C)^\mathrm{an};\Z)$.  Complex conjugation induces a functor $\mathcal{A}_g(\C) \to \mathcal{A}_g(\C)$ which we'll denote $A \mapsto A^c$ on objects.  The fact that complex conjugation is continuous on $\CC$ implies that the induced bijection $A(\C) \to A^c(\C)$ is continuous in the analytic topology, and hence induces a canonical isomorphism $H_1(A(\C)^\mathrm{an};\Z) \to H_1(A^c(\C)^\mathrm{an};\Z)$.  Therefore the diagram
  \begin{equation*}
    \xymatrix{
      |\mathcal{A}_g(\C)| \ar[d]^c \ar[r] & |\proj{\Z}|\\
      |\mathcal{A}_g(\C)| \ar[ur] &
      }
  \end{equation*}
  commutes up to homotopy 
  It follows that the homomorphism 
  \[
  \pi_{4k-2}^s(|\mathcal{A}_g(\C)|;\Z/q) \to \pi_{4k-2}^s(B\GL_{2g}(\Z);\Z/q) \to K_{4k-2}(\Z;\Z/q)
  \]
  coequalizes $c_*$ and the identity.  The claim is then deduced from surjectivity of $\pi_{4k-2}^s(|\mathcal{A}_g(\C)|;\Z/q) \to \KSp_{4k-2}(\Z;\Z/q)$.
\end{proof}

\begin{remark}
  It may be deduced from our main theorem that $c_B\colon \KSp_{4k-2}(\Z;\Z_p) \to K_{4k-2}(\Z;\Z_p)$ is also equivariant for $\Gal(\overline{\Q}_p/\Q_p) \subset \Aut(\C)$ for suitable isomorphisms $\C \cong \overline{\Q}_p$, see Subsection~\ref{univ2}.  It would be interesting to understand whether that equivariance could be seen more geometrically.
\end{remark}

\subsection{Galois conjugation of CM abelian varieties}
\label{sec:main-thm}

Fix a CM field $E$. It follows from Remark \ref{equiv-of-cat}
 that there exists a functor  of groupoids
 making the following diagram commutative:
 \begin{equation*}
  \xymatrix{
    \PPminus \ar[r]^-{\ST} \ar@{-->}[d]_{F_{\sigma}} & \mathcal{A}_g(\C) \ar[d]^\sigma\\
    \PPminus \ar[r]^-{\ST} & \mathcal{A}_g(\C),
  }
\end{equation*}

 The main theorem of complex multiplication, originally due to Shimura and Taniyama
for automorphisms fixing the reflex field, and extended to the general case by Deligne and Tate, 
effectively provides a formula for $F_{\sigma}$.

Let $H$ be the Hilbert class field of the CM field $E$.
We will formulate the result only when $E$ (so also $H$) is Galois over $\Q$. 
  Let $\Phi = \Phi(L,b) \subset\mathrm{Emb}(E,\C)$ be the CM structure on $E$ determined by $(L,b) \in \PPminus$. 
  Let $c$ denote the complex conjugation on $E$ and choose for each $\tau \in \mathrm{Emb}(E,\C)$ an extension $w_\tau: H \rightarrow \C$ to a complex embedding of $H$, 
   such that 
\[
w_{\tau c} = w_{c \tau} = c w_{\tau}.
\]
Then for each $\sigma \in \Gal(H/\Q)$ and $\tau \in \mathrm{Emb}(E,\C)$, 
both $\sigma w_{\tau}$ and $w_{\sigma \tau}$ give 
embeddings $H \rightarrow \C$ extending $\sigma \tau$
and, therefore, 
$ w_{\sigma \tau}^{-1}\sigma w_{\tau} \in \Gal(H/K)$.

 The following theorem computes much of the action of $F_{\sigma}$ on the homotopy of $\PPminus$, in the
 cases of interest.
\begin{theorem} \label{MilneMainTheorem}
  \begin{enumerate}
\item[(i)]    The map $\pi_0(F_{\sigma}) \colon \pi_0 (\PPminus) \rightarrow \pi_0 (\PPminus)$
is given on each fiber of $\pi_0 \PPminus \rightarrow \pi_0 \mathcal{P}_{E \otimes \R}^{-}$
(i.e., upon fixing the CM type) 
by  tensoring, as in \eqref{eq:4},
 with a certain $[(X,q)]  \in \pi_0 \PPplus$ determined by $\sigma$ and the CM type.
  
 Moreover, the class of  $[X]$ under  the Artin map $\pi_0(\pic{\mathcal{O}_E})  \xrightarrow{\mathrm{Art}}  \Gal(H/K)^{\ab}$
 is given by
  \begin{equation}\label{eq: taniyama cocycle}
   \mathrm{Art} \left(X\right)=  \mbox{ class of } \left[ \sum_{\tau \in \Phi}  w_{\sigma \tau}^{-1}\sigma w_{\tau} \right] \mbox{ in } \Gal(H/K)^{\ab}.
 \end{equation}
\item[(ii)]  In the case $E=K_q$ the map on higher homotopy groups
 \begin{equation} \label{betalinear}  \pi_*(F_{\sigma}): \pi^s_*(|\PPminus|, \Z/q) \rightarrow \pi^s_*(|\PPminus|, \Z/q)\end{equation}
is $\Z/q[\beta]$-linear, that is to say, it sends 
 $[\beta^j(L,b) ]$ to $ \beta^j \sigma([L,b])$,
 with notation as described after Proposition \ref{prop:generators-for-KSp}.
 \end{enumerate}
\end{theorem}
For example, $(X,q) = (\mathcal{O}_q,-1)$ when $\sigma = c$ is complex conjugation, see Remark~\ref{rem:differ}.
\begin{proof}

(i)  We defined in Remark~\ref{rem:differ} a tensoring bifunctor $\PPminus \times \PPplus \rightarrow \PPminus$, such that, for each $(L,b) \in \PPminus$ and each
\emph{positive definite} $(Y,h) \in \PPplus$ we have 
 \[
 \ST( (L,b) \otimes (Y,h) ) \cong \ST(L,b) \otimes_{\mathcal{O}} (Y,h),
 \] 
 where on the right we have the Serre tensor construction, cf. Remark~\ref{rem:differ}. 
Because applying $\sigma$ commutes with the Serre tensor construction, this implies that
 $$ F_{\sigma}((L, b) \otimes (Y, h)) \cong F_{\sigma}(L, b) \otimes (Y,h),$$
 naturally in $(L,b)$ and $(Y,h)$. Now,  $\pi_0(\PPminus)$ is a torsor under the tensoring action of $\pi_0(\PPplus)$, 
 and  this induces on each fiber of $\pi_0 (\PPminus ) \rightarrow \pi_0 (\mathcal{P}_{E \otimes \R}^{-})$ the structure of  torsor under the positive definite subgroup of $\pi_0( \PPplus)$.
  Hence the action of $\pi_0(F_{\sigma})$ on any such fiber is through tensoring with the class of a particular $(X,q) \in \PPplus$. To complete the proof of (i), we need to pin down the explicit formula for $(X,q)$, which is given in  \cite[Theorem 4.2]{Mil07} except there Milne has replaced the Artin map by its refinement
 $
  \mathbf{A}_{f,E}^\times/E^\times \to \Gal(\ol{\Q}/E)^{\ab}.$

(ii) Writing ${\PPplus}^{\mathrm{pos.def.}} \subset \PPplus$ for the full subgroupoid on the positive definite $(Y,h)$, naturality implies that~\eqref{betalinear} is linear over the graded ring $\pi_*^s(|{\PPplus}^{\mathrm{pos.def.}}|;\Z/q)$, which contains $\Z/q[\beta]$ because $\Z/q \cong U_1(\mathcal{O}_q)$ is the automorphism group of $(Y_0,h_0) = (\mathcal{O}_q,1) \in {\PPplus}^{\mathrm{pos.def.}}$.
 \end{proof}  
 
  \section{The main theorem and its proof}  \label{maintheoremproof}
Recall from  
 Theorem \ref{thm:KSp-of-Z} that  there is an isomorphism
\[
\begin{tikzcd}
 \KSp_{4k-2}(\Z;\Z/q)   \ar[rr, "{(c_H, c_B)}", "\sim"'] &  &  \pi_{4k-2}(ku;\Z/q) \times K_{4k-2}(\Z;\Z/q)^{(+)}.
 \end{tikzcd}
 \]

 Let us recall that $\pi_{4k-2}(ku;\Z/q)$ is a cyclic of order $q$, generated by the $2k-1$st power of the Bott class 
 $\mathrm{Bott} \in \pi_2(ku;\Z/q)$. For purposes
 of making Galois equivariance manifest, we will in the current section identify
 $$ \pi_{4k-2}(ku; \Z/q) \stackrel{\sim}{\longrightarrow} \mu_q^{\otimes 2k-1}$$
 via $\Bott^{2k-1} \mapsto \zeta_q^{\otimes 2k-1}$.   By means of this identification, the target
 of the map $c_H$ can be considered
 to be $\mu_q^{\otimes 2k-1}$. 
  
 \begin{theorem} \label{mt2}
 Let $H_q \subset \C$  be the largest unramified extension of $K_q$ with  abelian $p$-power Galois group.
Let $G=\Gal(H_q/\Q)$, and let $\langle c \rangle \leqslant G$
 be the order $2$ subgroup  generated by complex conjugation. 
 \begin{itemize}
 \item[(i)]
 The Galois action on $\KSp_{4k-2}(\Z;\Z/q)$
 factors through $G$.
 \item[(ii)]
The sequence
\begin{equation} \label{seq}    \Ker(c_H)  \rightarrow  \KSp_{4k-2}(\Z;\Z/q)  \stackrel{c_H}{\rightarrow}   \mu_q^{\otimes 2k-1}\end{equation}
  is a short exact sequence of $G$-modules, where the $G$-action on $\Ker(c_H)$
  is understood to be trivial, and the action on $\mu_q$ is via the cyclotomic character.

    \item[(iii)] The sequence \eqref{seq} is universal for  extensions of $\mu_q^{\otimes (2k-1)}$
  by a trivial $\Z/q[G]$-module.
  \end{itemize}
\end{theorem}

In detail, the final assertion (iii) means that the sequence \label{cHsequence} is
the initial object of a category  $\Cal{C}_{\Z/q}(G;  \mu_q^{\otimes (2k-1)})$
of extensions of $G$-modules of $\mu_q^{\otimes (2k-1)}$ by a trivial $G$-module.
 This category and its basic properties are discussed
in \S \ref{sec:cocycl-univ-extens}.

\begin{remark}
It turns out to be technically more convenient to work in a more rigid category of
sequences equipped with splitting, and we will in fact prove the following statements:

{\em 
\begin{quote}
 \begin{itemize}
 \item[(ii')]
 There is a unique splitting of the sequence that    is equivariant for the action of $\langle c\rangle$; explicitly
  the kernel of $c_B$ maps isomorphically to $\mu_q^{\otimes 2k-1}$
  under $c_H$ and yields such a splitting.  
  
  \item[(iii')] The sequence
  \eqref{seq} is universal for extensions of $\mu_q^{\otimes 2k-1}$
  by a trivial $\Z/q[G]$-module that are equipped with a $\langle c\rangle$-equivariant splitting. 
  \end{itemize}
  \end{quote}
   }
 \end{remark}

 We will first give some generalities on universal extensions in 
\S \ref{sec:cocycl-univ-extens}. We then verify  (i) and (ii) in \S \ref{sec:galo-acti-qual},
and then (iii)  in \S \ref{sec:universality}. Finally, we give a number of
related universal properties in \S \ref{univ1}, \ref{univ2}.

\subsection{Cocycles and universal extensions}
\label{sec:cocycl-univ-extens}

\begin{definition} \label{catdef}
  Let $G$ be a discrete group, $H \leqslant G$ a subgroup, and $M$ a $\Lambda[G]$-module for some coefficient ring $\Lambda$. We consider a category $\Cal{C}_\Lambda(G, H; M)$ of ``extensions of $M$ by a trivial $G$-module, equipped with an $H$-equivariant splitting''.  More precisely, the objects of $\Cal{C}_\Lambda(G, H; M)$ are triples $(V,\pi,s)$ where $V$ is a $\Lambda[G]$-module, $\pi \in \Hom_{\Lambda[G]}(V,M)$ and $s \in \Hom_{\Lambda[H]}(M,V)$ satisfy $s \circ \pi = \mathrm{id}_M$, and the $\Lambda[G]$-module $T = \Ker(\pi)$ has trivial $G$-action; the morphisms $(V,\pi,s) \to (V',\pi',s')$ are those $\phi \in \Hom_{\Lambda[G]}(V,V')$ for which $\pi' \circ \phi = \pi$ and $\phi \circ s = s'$.

  We also consider the variant $\Cal{C}_\Lambda(G;M)$ where there is only given $(V,\pi)$ with $\Lambda[G]$-linear $\pi: V \to M$ and morphisms satisfy only $\pi' \circ \phi = \pi$.
  (As a warning, this is not the same category as $\Cal{C}_{\Lambda}(G,\{e\};M)$.)
\end{definition}

Objects of $\Cal{C}_\Lambda(G,H;M)$ may be depicted as short exact sequences of $\Lambda[G]$-modules
\begin{equation}\label{eq: object of C}
  \begin{tikzcd}
    T \arrow{r} & V \arrow{r}{\pi} & M. \arrow[bend left=33]{l}{s}
  \end{tikzcd}
\end{equation}
equipped with $\Lambda[H]$-equivariant splittings.  The identity map of $M$ evidently gives a terminal object in this category. 

We will show that the category $\Cal{C}_\Lambda(G,H;M)$ always has an initial object, 
which we call the \emph{universal extension} and denote
\begin{equation}\label{eq: universal extension}
  \begin{tikzcd}
    T^\mathrm{univ} \arrow{r} & V^\mathrm{univ} \arrow{r}{\pi} & M, \arrow[bend left=33]{l}{s}.
  \end{tikzcd}
\end{equation}
We will also see that there is a canonical isomorphism $H_1(G,H; M) \cong T^{\mrm{univ}}$ (where $H_1(G, H; -)$ is {\em relative} group homology). Any other object of $\Cal{C}_\Lambda(G,H;M)$ arises by pushout from the universal extension, so we think of \eqref{eq: universal extension} as being the ``most non-trivial'' object in $\Cal{C}_\Lambda(G,H;M)$.

 To  an object \eqref{eq: object of C} of $\Cal{C}_{\Lambda}(G,H; M)$, as above, we associate a function $\alpha: G \times M \to T$, by
\begin{equation}\label{eq:62}
  \alpha(g,m) = g.(s(m)) - s(g.m).
\end{equation}
This function satisfies
\begin{enumerate}[(i)]
\item for every $g \in G$, the function $\alpha(g,m)$ is $\Lambda$-linear in $m \in M$, 
\item\label{item:7} $\alpha(g,m) = 0$ when $g \in H$, 
\item\label{item:8} the cocycle condition
\begin{equation}
  \label{eq:19}
  \alpha(g g',m) = \alpha(g,g'.m) + \alpha(g',m).
\end{equation}
\end{enumerate}
Now, the rule $(m,t) \mapsto s(m) + t$ defines a $\Lambda[H]$-linear isomorphism $M \times T \stackrel{\sim}{\rightarrow} V$, 
   with respect to which
\begin{equation}\label{eq:58}
  g.(m,t) = (g.m,t + \alpha(g,m)),
\end{equation}
so the object~\eqref{eq: object of C} is described uniquely by the $\Lambda$-module $T$ and the function $\alpha$.  This defines an equivalence of categories between $\mathcal{C}_\Lambda(G,H;M)$ and a category whose objects are pairs $(T,\alpha)$ and whose morphisms are $\Lambda$-linear maps $f: T \to T'$ such that $f(\alpha(g,m)) = \alpha'(g,m)$.

Recall also that group homology $H_*(G;M)$ is calculated by a standard ``bar'' complex with
\begin{equation*}
  C_i(G;M) = \Z[G^i] \otimes_\Z M \cong \Lambda[G^i] \otimes_\Lambda M.
\end{equation*}
The inclusion $i: H \subset G$ gives an injection $C_*(H;\mathrm{Res}^G_H M) \to C_*(G;M)$ and we let $C_*(G,H;M)$ be the quotient; in particular $C_0(G,H;M) = 0$.  Its homology is the relative group homology $H_*(G,H;M)$, which sits in a long exact sequence with $i_*: H_*(H;\mathrm{Res}^G_H M) \to H_*(G;M)$.
For an object~(\ref{eq: object of C}) the cocycle $\alpha$ defines a $\Lambda$-linear map
\begin{equation}\label{eq:63}
  C_1(G;M) = \Z[G] \otimes M \cong \Lambda[G] \otimes_\Lambda M \xrightarrow{\alpha} T,
\end{equation}
and the conditions~(\ref{item:7}) and~(\ref{item:8}) say that this map has both $C_1(H;\mathrm{Res}^G_H M)$ and $\partial C_2(G;M)$ in its kernel.  Therefore $\alpha$  gives a $\Lambda$-linear map
\begin{equation}\label{eq:57}
  H_1(G,H;M) \xrightarrow{[\alpha]} T.
\end{equation}

\begin{lemma}\label{lem: initial object exists}
  The rule associating $[\alpha]$ of (\ref{eq:57}) to the extension ~\eqref{eq: object of C} defines an equivalence of categories between $\mathcal{C}_\Lambda(G,H;M)$ and the category of $\Lambda$-modules under $H_1(G,H;M)$. 
In particular, the category $\Cal{C}(G, H; M)$ has an initial object
  $$T^{\univ} \rightarrow V^{\univ} \twoheadrightarrow M$$ 
wherein $T^{\univ} \cong H_1(G, H; M)$,  the relative group homology.
\end{lemma}

\begin{proof}
Indeed, the functor in the other direction is described as follows: given $f: H_1(G,H;M) \to T$, define $\alpha: G \times M \to T$ by composing $f$ with the canonical maps $G \times M \to C_1(G,H;M) \to H_1(G,H;M)$, and set $V = M \times T$ with $G$-action given by~(\ref{eq:58}).
  The identity map of $H_1(G,H;M)$ then corresponds to an initial object with $T^\mathrm{univ} = H_1(G,H;M)$.
\end{proof}

\begin{remark}
The proof above also gives an explicit description of the map $T^\mathrm{univ} \to T$ arising from the universal map to another object~\eqref{eq: object of C}: first extract $\alpha: G \times M \to T$ as in~\eqref{eq:62}, extend to an additive map~\eqref{eq:63} and factor as in~\eqref{eq:57}.
\end{remark}

There are natural situations where one can drop $H$.

\begin{lemma}\label{lem: vanishing homology implies same universal object} 
\begin{itemize}
\item[(a)]
  If $H_0(H;M) = 0 = H_1(H;M)$, then the forgetful functor
  \begin{equation*}
    \mathcal{C}_\Lambda(G,H;M) \to \mathcal{C}_\Lambda(G;M)
  \end{equation*}
  is an equivalence.  In particular the image of the initial object of $\mathcal{C}_\Lambda(G,H;M)$ is initial in $\mathcal{C}_\Lambda(G;M)$.
  \item[(b)] If $H_0(H, M) = 0$, then the forgetful functor
    \begin{equation*}
    \mathcal{C}_\Lambda(G,H;M) \to \mathcal{C}_\Lambda(G;M)^{H\text{-split}}
  \end{equation*}
  is an equivalence, where, on the right, we take the full subcategory of $\mathcal{C}_{\Lambda}(G; M)$
consisting of sequences which admit splittings as sequences of $H$-modules.
  \end{itemize}
  
\end{lemma}
\begin{proof}
 If we regard an object~\eqref{eq: object of C} as an extension of $\Lambda[H]$-modules, it is classified by an element of
$ 
    \Ext^1_{\Lambda[H]}(M,T) 
 $ and two splittings differ by an element of
  $
    \Hom_{\Lambda[H]}(M,T).
 $
  Under the vanishing assumption of (a), both these groups vanish;
  so the splitting is unique and hence the forgetful functor
  is an equivalence. In the setting of (b) only the
  latter group vanishes, which still implies that the stated
 forgetful functor is an equivalence. 
  \end{proof}

In the absense of a specified $H$, it can be shown that $\mathcal{C}_\Lambda(G;M)$ admits an initial object if and only if $H_0(G;M) = 0$, and in this case the kernel is $H_1(G;M)$.  (Note that $\mathcal{C}_\Lambda(G;M)$ is not the same as $\mathcal{C}_\Lambda(G,\{e\};M)$.)

\subsection{Proof of (i) and (ii) of the main theorem}  \label{sec:galo-acti-qual}

  We briefly recall some of the prior results before proceeding to the proof. We have constructed maps    
 $$\label{eq: multiline}
 \pi_{4k-2}^s(|\PPqminus|;\Z/q)   \rightarrow  \KSp_{4k-2}(\Z;\Z/q)   \stackrel{(c_H, c_B)}{\longrightarrow}  \pi_{4k-2}(ku;\Z/q) \times K_{4k-2}(\Z;\Z/q)^{(+)}. 
 $$

Recall from \eqref{eq:2} that each class
$[(L, b)] \in  \pi_0(\PPqminus)$ gives a class
 $\beta^{2k-1} [(L,b)] \in  \pi_{4k-2}^s(|\PPqminus|;\Z/q)$;
 here $\beta \in \pi_2^s(|\PPqminus|; \Z/q)$ is
 a Bott element induced by  the primitive root of unity $e^{2 \pi i/q} = \zeta_q \in \mathcal{O}_q$. 
 With this notation, we have previously verified:

 \begin{itemize}
 \item[(a)] Under the composite map, the images of  elements
 $
  \beta^{2k-1} [(L,b)]$ generate the codomain $\KSp_{4k-2}$ (proof of Proposition 
\ref{KqexhaustsKSp}). 

 \item[(b)] (Proposition \ref{prop: hodge map surjects}
 and Corollary~\ref{cor:Soule-formula}) Explicitly, the image of $
  \beta^{2k-1} [(L,b)]$ is
  \begin{equation} \label{HSrecall}
   \begin{tikzcd}
     &   \big(\sum_{a: j_a \in \Phi} a^{2k-1} \big) \Bott^{2k-1} 
      \in \pi_{4k-2}(ku;\Z/q) 
     \\
     \beta^{2k-1}[(L,b)] \arrow[ru,mapsto, "c_H"] \arrow[rd,mapsto,"c_B"] \\
     &   \mathrm{tr}(\beta^{2k-1} ([L]-1))  \in  K_{4k-2}(\Z;\Z/q)^{(+)}
   \end{tikzcd}
  \end{equation}
  where:
  \begin{itemize}
  \item $\mathrm{tr}: K_*(\mathcal{O}_q;\Z/q) \to K_*(\Z;\Z/q)$ is the transfer,
 \item $\Phi = \Phi_{(L,b)} \subset \Hom(K_q, \C)$ is the CM type
  associated to $(L, b)$ by  \eqref{PhiLbdef}, and $j_a \in \Hom(K_q,\C)$ is the embedding that sends $\zeta_q \mapsto e^{2 \pi i a/q}$, for $a \in (\Z/q)^\times$,
  \item  On the right, $\Bott \in \pi_2(ku;\Z/q)$ is the mod $q$ reduction of the Bott element.

  \end{itemize}
  \item[(c)] (By \eqref{betalinear} and surrounding discussion): The action of $\sigma \in \Aut(\C)$  on     $ \pi_{4k-2}^s(|\PPqminus|;\Z/q) $  sends
 \begin{equation} \label{MainThmRecall}
  \beta^{2k-1} [(L,b)] \mapsto  \beta^{2k-1} [(L,b) \otimes (X,q)]
\end{equation}
for a certain $(X,q)  \in  \PPqplus$ depending only on the CM type $\Phi_{(L, b)}$ and $\sigma$;
the image of $X$ under the Artin map  was described in  Theorem  \ref{MilneMainTheorem}  
and depends only on the restriction of $\sigma$ to $H_q$.
  \end{itemize}

\begin{lemma}\label{lem: 7.6}
  In the extension
  \begin{equation*}
    \Ker(c_H) \to \KSp_{4k-2}(\Z;\Z/q) \xrightarrow{c_H} \mu_q^{\otimes(2k-1)},
  \end{equation*}
  the action of $\Aut(\C)$ factors through $\Gal(H_q/\Q)$.
  The map $c_H$ is equivariant for this action, and 
  the action on $\Ker(c_H)$ it is trivial.  Finally, 
  the kernel of $c_B$ maps isomorphically to $\mu_q^{\otimes(2k-1)}$
  under $c_H$, splitting the above sequence equivariantly for $\langle c\rangle$. 
\end{lemma}

\begin{proof}

By point (a) above, the action of $\Aut(\C)$ 
is determined by its action on classes $\beta^{2k-1}[(L,b)]$
and by point (c) this action indeed factors through $\Gal(H_q/\Q)$. 

Morally speaking, the equivariance of $c_H$ arises simply from the fact that one can define the Hodge class via algebraic geometry. 
We give a formal argument by a direct computation, using the explicit formula in (b) above. By Corollary \ref{cor: explicit determination of KSp} and point (a) above, we know the images of $\beta^{2k-1}[(L,b)]$ under $\pi_{4k-2}^s(|\PPqminus|;\Z/q)   \rightarrow  \KSp_{4k-2}(\Z;\Z/q)  $ generate all of $\KSp_{4k-2}(\Z; \Z/q)$. Therefore,  it suffices to check the equivariance  for $\Aut(\C)$ acting on $\beta^{2k-1}[(L,b)]$. 

According Proposition \ref{prop: hodge map surjects}, $c_H$ sends
\[
\beta^{2k-1}[(L,b)] \mapsto \left(\sum_{a: j_a \in \Phi} a^{2k-1}\right) \mrm{Bott}^{2k-1} \in \pi_{4k-2}(ku; \Z/q).
\]
Evidently this only depends on the CM type of $(L,b)$, which can be described as the set of characters by which $K_q^{\times}$ acts on the tangent space of the associated abelian variety $\ST(L,b)$ (cf. \S \ref{ssec:constr-abel-vari}). 

Now, consider the action of $K_q^{\times}$ on the tangent space of $\sigma \ST(L,b)$. This is the same underlying scheme as $\ST(L,b)$ but with its structure map to $\Spec(\CC)$ twisted by $\Spec(\sigma^{-1})$, so that as a $\CC$-vector space, 
\[
T_e (\sigma \ST(L,b)) = T_e (\ST(L,b)) \otimes_{\CC, \sigma} \CC.
\]
Therefore, $\sigma \in \Aut(\C/\Q)$ acts on $\Phi$ by post-composition with $\sigma$. Under the identification $\mrm{Emb}(K_q, \CC) \cong (\Z/q)^\times$, this is identified with multiplication by $\chi_{\cyc}(\sigma) \in (\Z/q)^\times$,
the cyclotomic character of $\sigma$.  Hence we find that 
\begin{align*}
  c_H(\sigma \cdot \beta^{2k-1}[(L,b)]) &= \left(\sum_{a: j_a \in \sigma\Phi} a^{2k-1}\right)
                                          \mrm{Bott}^{2k-1} \\
                                        &= \left(\sum_{a: j_a \in \Phi} 
                                          \chi_{\cyc}(\sigma)^{2k-1} a^{2k-1} \right)  \mrm{Bott}^{2k-1}  \\
&= \chi_{\cyc}(\sigma)^{2k-1} c_H(\beta^{2k-1}[(L,b)]).
\end{align*}
This shows the equivariance of $c_H$, as desired. 

Now we check that the Galois action on $\ker(c_H)$ is trivial. In the course of proving Proposition \ref{prop:generators-for-KSp} we have seen -- see \eqref{Phigen} -- that, as $(L,b)$ ranges over objects in $\PPqminus$ inducing a fixed CM structure $\Phi_{(L,b)} = \Phi \subset \mathrm{Emb}(K_q,\C)$, the values of $c_B([(L,b)]) = \mathrm{tr}(\beta^{2k-1} \cdot ([L]-1)) \in K_{4k-2}(\Z;\Z/q)^{(+)}$ exhaust that group.  Therefore it suffices to see that if $[(L,b)]$ and $[(L',b')]$ induce the same CM structure on $K_q$, then $\sigma \in \Aut(\C)$ acts trivially on the element
  \begin{equation} \label{X33}
      \mathrm{tr}(\beta^{2k-1}\cdot ([L]-1)) - \mathrm{tr}(\beta^{2k-1}\cdot ([L']-1)).
  \end{equation}
  
  According to point (c) above, $\sigma \in \Aut(\CC/\Q)$ takes $\beta^{2k-1} [(L,b)] \mapsto  \beta^{2k-1} [(L,b) \otimes (X,q)]$ where $(X,q)$ depends on $(L,b)$ only through the CM type $\Phi_{(L,b)}$. Using the formula
 $    ([L \otimes_{\OO_q} X]-1) = ([L] - 1) + ([X]-1) \in K_0(\OO_q)$ (which is seen by noting that both sides
 having same rank and determinant) we get an equality inside $K_{4k-2}(\Z;\Z/q)$
    \begin{equation} \label{eq:cv0}
    c_B(\beta^{2k-1} \cdot [(L,b) \otimes (X,q)]) = c_B(\beta^{2k-1} \cdot [(L,b)]) + \mathrm{tr}(\beta^{2k-1}([X]-1)),
  \end{equation}
  where $[X]$ depends on $(L,b)$ only through its CM type. Therefore, $\Aut(\C)$ acts trivially on the expression \eqref{X33} in which $(L,b)$ and $(L', b')$ have the same CM type, as desired. 
  
  The last part, about the equivariance of the splitting for the subgroup $\langle c\rangle$ generated by conjugation, follows from Lemma \ref{lem:Soule-complex-conjugation}.
\end{proof}

This concludes the proof of parts (i) and (ii) of Theorem \ref{mt2}, as
well as the statements of (ii') about splitting.
 
 \subsection{Proof of (iii) of the main theorem} 
 \label{sec:universality}
It remains to prove (iii) of Theorem \ref{mt2}.
 The properties verified in Lemma \ref{lem: 7.6} show that  in the sequence
\begin{equation} \label{local exact}     \Ker(c_H)  \rightarrow  \KSp_{4k-2}(\Z;\Z/q)  \stackrel{c_H}{\rightarrow}   \mu_q^{\otimes(2k-1)}\end{equation}
 defines an object of $\mathcal{C}_{\Z/q}(\Gal(H_q/\Q),\langle c\rangle;\mu_q^{\otimes(2k-1)})$.  Our final task is to prove that it is an initial object in this category. 
    This will prove (iii') of Theorem \ref{mt2}, from which (iii) follows by  Lemma \ref{lem: vanishing homology implies same universal object}.

    Let us denote ``the'' initial object of $\mathcal{C}_{\Z/q}(\Gal(H_q/\Q),\langle c\rangle;\mu_q^{\otimes(2k-1)})$ by $T^\mathrm{univ} \rightarrow V^\mathrm{univ}\rightarrow \mu_q^{\otimes(2k-1)}$.  
Now 
 Lemma \ref{lem: initial object exists}
 gives an abstract isomorphism of $\Ker(c_H)$ with $T^{\mathrm{univ}}$, via the isomorphisms:
 \begin{equation}
\Ker(c_H) \stackrel{c_B}{\rightarrow} K_{4k-2}(\Z; \Z/q)^{(+)} \stackrel{\eqref{eq:50}}{\longrightarrow} H_1(\Gal(H_q/\Q),\langle c\rangle;\mu_q^{\otimes(2k-1)}).\label{eq:84}
\end{equation}
(In the case at hand $H_1(\Gal(H_q/\Q), \langle c \rangle; \mu_q^{\otimes 2k-1}) = H_1(\Gal(H_q/\Q); \mu_q^{\otimes 2k-1})$
since
the homology of $\langle c \rangle$ on $\mu_q^{\otimes 2k-1}$ is trivial in all degrees.)
  We shall show that, with reference to this identification, the $1$-cocycle 
  \begin{equation}\label{alphacocycle}
    \alpha: \Gal(H_q/\Q) \times \mu_q^{\otimes(2k-1)} \to \Ker(c_H) 
  \end{equation}
  (arising from \eqref{local exact} and its splitting via $c_B$)
  is identified with the tautological $1$-cocycle valued in $H_1(\Gal(H_q/\Q),\langle c\rangle;\mu_q^{\otimes(2k-1)})$. This will complete the proof of Theorem \ref{mt2} (iii') by Lemma \ref{lem: initial object exists} and the discussion preceding it. 
  
Denote by $\mathrm{Pr}$ the projection of $\KSp_{4k-2}(\Z;\Z/q)$ to $\Ker(c_H)$
with kernel $\Ker(c_B)$.
For $\sigma \in \Gal(H_q/\Q)$ and $m \in \mu_q^{\otimes 2k-1}$,
we have  in the notation of  \S \ref{sec:cocycl-univ-extens} 
the equality $\alpha(\sigma,m) = \Pr \circ (g- \mrm{id}) (\widetilde{m})$ where  $\widetilde{m} \in \KSp_{4k-2}(\Z;\Z/q)$ is any element with $c_H(\tilde{m}) = m$.
Therefore, the value of the cocycle $\alpha$ on $\sigma \in \Gal(H_q/\Q)$ and $c_H(\beta^{2k-1} [(L,b)]) \in \mu_q^{\otimes(2k-1)}$ is given 
by     $$ \alpha(\sigma, c_H(\beta^{2k-1} [(L,b)]) ) = \mathrm{Pr}  \circ (\sigma - \mathrm{id}) \circ \beta^{2k-1}[(L,b)] \in \Ker(c_H).$$

By Theorem \ref{MilneMainTheorem}, we have $\sigma (\beta^{2k-1}[(L,b)]) = \beta^{2k-1}\cdot [(L,b) \otimes (X,q)]$ where $(X,q)$ is determined explicitly by the CM type of $(L,b)$. Hence 
\[
\mathrm{Pr}  \circ (\sigma - \mathrm{id}) \circ \beta^{2k-1}[(L,b)]= \beta^{2k-1}[(X,q)].
\]
Under the identification~\eqref{eq:84}, 
the class $\beta^{k-1}[(X,q)]$ is sent to the Artin class of $X$ pushed forward via $\Gal(H_q/K_q) \rightarrow \Gal(H_q/\Q)$. 
 In detail, there is a diagram:
 \begin{equation*}
   \begin{tikzcd}[remember picture]
     \mathrm{Pr}  \circ (\sigma - \mathrm{id}) \circ \beta^{2k-1}[(L,b)] \ar[r,mapsto]   &   \tr(\beta^{2k-1} ([X]-1))  \ar[r,mapsto] &    \iota_*( \Art(X) \otimes \zeta_q^{2k-1}) \\
     \Ker(c_H) \arrow[r,"c_B", "\sim"'] &  K_{4k-2}(\Z; \Z/q)^{(+)}  \arrow[r, "\eqref{eq:50}", "\sim"'] & H_1(\Gal(H_q/\Q),\langle c\rangle;\mu_q^{\otimes(2k-1)}).
   \end{tikzcd}
   \begin{tikzpicture}[overlay,remember picture]
     \path (\tikzcdmatrixname-1-1) to node[midway,sloped]{$\in$}
     (\tikzcdmatrixname-2-1);
     \path (\tikzcdmatrixname-1-2) to node[midway,sloped]{$\in$}
     (\tikzcdmatrixname-2-2);
     \path (\tikzcdmatrixname-1-3) to node[midway,sloped]{$\in$}
     (\tikzcdmatrixname-2-3);
   \end{tikzpicture}
 \end{equation*}
\begin{itemize}
\item  In the middle, we used  \eqref{MainThmRecall} and \eqref{eq:cv0}; $\tr$ is the $K$-theoretic trace from $\mathcal{O}_q$ to $\Z$;
 \item     On the right, we used Proposition \ref{shtuka};  $\mathrm{Art}(X)$ is the Artin map
   applied to the class of $X$ in the Picard group of $\mathcal{O}_q$, and 
  $\iota_*$ is induced on homology
   by the inclusion $\iota: \Gal(H_q/K_q) \rightarrow \Gal(H_q/\Q)$. 
\end{itemize}

Therefore,  using the explicit formula in the Main Theorem of CM given in ~\eqref{eq: taniyama cocycle}, we find: 
  \begin{multline*}
    \alpha(\sigma,c_H(\beta^{2k-1}[(L,b)])) = \left( \prod_{\varphi \in \Phi(L,b)} w_{\sigma \varphi}^{-1} \sigma w_{\varphi}  \right) \otimes \zeta_q^{\otimes(2k-1)}\\
    \in C_1(\Gal(H_q/\Q);\mu_q^{\otimes(2k-1)}) \twoheadrightarrow H_1(\Gal(H_q/\Q),\langle c\rangle;\mu_q^{\otimes(2k-1)}).
  \end{multline*}

  Now we manipulate this expression using that we are allowed to change this expression by elements of $C_1(\langle c\rangle;\mu_q^{\otimes(2k-1)})$ and $\partial C_2(\Gal(H_q/\Q);\mu_q^{\otimes(2k-1)})$, without affecting the homology class.  The latter gives a relation
  (cf. \eqref{eq:19})  
  \begin{equation}\label{eq: 1-coboundary} 
    (g_1g_2) \otimes m \sim g_1 \otimes g_2 m   + g_2 \otimes m, \quad g_1, g_2 \in \Gal(H_q/\Q),\; m \in \mu_q^{\otimes(2k-1)}
  \end{equation}
  from which we also deduce
  \begin{equation}\label{eq: relation 2}
    0 = (g g^{-1}) \otimes m = g \otimes g^{-1} m + g^{-1} \otimes m, \quad g \in \Gal(H_q/\Q),\; m \in \mu_q^{\otimes(2k-1)}.
  \end{equation}
  By repeated application of~\eqref{eq: 1-coboundary} we get
  \begin{equation}\label{eq: taniyama 2}
    \left( \prod_{\varphi \in \Phi} w_{\sigma \varphi}^{-1} \sigma w_{\varphi}  \right) \otimes \zeta_q^{\otimes (2k-1)} = \sum_{\varphi \in \Phi} (w_{\sigma \varphi}^{-1} \sigma w_{\varphi} \otimes \zeta_q^{\otimes (2k-1)}) \in H_1(\Gal(H_q/\Q), \langle c \rangle; \mu_q^{\otimes (2k-1)}).
  \end{equation}
  Similarly, by combining~\eqref{eq: 1-coboundary} and~\eqref{eq: relation 2} we get
  \begin{equation}\label{eq: relation 1}
    \begin{aligned}
      w_{\sigma \varphi}^{-1} \sigma w_{\varphi}  \otimes \zeta_q^{\otimes (2k-1)} & = w_{\sigma \varphi}^{-1} \otimes {\sigma \varphi}  (\zeta_q^{\otimes (2k-1)})+ \sigma \otimes  \varphi (\zeta_q^{\otimes (2k-1)})  + w_{\varphi} \otimes \zeta_q^{\otimes (2k-1)}\\
      & =
      - w_{\sigma \varphi} \otimes \zeta_q^{\otimes (2k-1)} +  \sigma \otimes \varphi (\zeta_q^{\otimes (2k-1)}) + w_{\varphi} \otimes \zeta_q^{\otimes (2k-1)}.
    \end{aligned}
  \end{equation}

  Finally, observe that 
  \begin{equation*}
    \sum_{\varphi \in \Phi} w_{\sigma \varphi} \otimes \zeta_q^{\otimes (2k-1)} = \sum_{\varphi \in \Phi} w_{\varphi}  \otimes \zeta_q^{ \otimes (2k-1)} \in C_1(\Gal(H_q/\Q),\langle c\rangle; \mu_q^{\otimes(2k-1)}) 
  \end{equation*}
  for $q$ odd. Indeed, $\sigma \Phi$ is another CM type, which contains exactly one representative from each conjugate pair of embeddings $E \hookrightarrow \CC$, and also, by construction $w_{c \varphi} = c w_{\varphi}$, 
  and $c w_{\varphi} \otimes \zeta_q^{\otimes(2k-1)} - w_{\varphi} \otimes \zeta_q^{\otimes(2k-1)}$ belongs to  $C_1(\langle c\rangle;\mu_q^{\otimes(2k-1)})$.

  We deduce the formula 
  \begin{align*}
    \alpha(\sigma,c_H(\beta^{2k-1}[(L,b)])) & = \sum_{\varphi \in \Phi(L,b)} \sigma \otimes \varphi (\zeta_q^{\otimes (2k-1)}) \\
    & = \sigma \otimes \big(\sum_{\varphi \in \Phi(L,b)}  \varphi(\zeta_q^{\otimes (2k-1)}) \big) \\
    &= \sigma \otimes c_H(\beta^{2k-1}[(L,b)]).
  \end{align*}
  Since this holds for any $(L,b) \in \PPqminus$, which generate under $c_H$ by Proposition \ref{prop: hodge map surjects}, we deduce the simple formula
  \begin{equation*}
    \alpha(\sigma,x) = [\sigma \otimes x] \in H_1(\Gal(H_q/\Q),\langle c\rangle;\mu_q^{\otimes(2k-1)}),
  \end{equation*}
  which verifies the claim made after \eqref{alphacocycle} and thereby completes the proof.  \qed

\subsection{Universal property of symplectic $K$-theory with $\Z_p$ coefficients} \label{univ1}
We have finished the proof of the universal property characterizing the $\Aut(\C)$-action on $\KSp_{4k-2}(\Z;\Z/q)$ for all $k$ and all odd prime powers $q = p^n$.  By taking inverse limit over $n$, this also determines the action on $\KSp_{4k-2}(\Z;\Z_p)$.
We shall now formulate a universal property adapted to this limit. 

We need some generalities on profinite group homology. 
This has no real depth in our case 
as it only serves as a notation to keep track of inverse limits
of {\em finite} group homology.
Let $G$ be a profinite group. 
 Let $\Lambda$ be a coefficient ring,  complete for the $p$-adic
 topology.  
 A topological $\Lambda$-module will be a  $\Lambda$-module
$M$ such that each $M/p^n$ is finite and the induced 
map $M \rightarrow \varprojlim M/p^n$ is an isomorphism;
we always regard $M$ as being endowed with the $p$-adic topology.  These assumptions are not maximally general, cf.~\cite{Zalesskii}: among profinite abelian groups  our assumptions on $M$ are equivalent to it being a finitely generated $\Z_p$-module.

Define the completed group algebra
$$\Lambda[[G]] := \varprojlim_{n,U} \frac{\Lambda}{p^n\Lambda}[G/U]$$
the limit ranging over open subgroups $U$ and positive integers $n$. 
If $M$ is a topological $\Lambda$-module
with a continuous action of $G$,  then
$M$ carries a canonical structure of $\Lambda[[G]]$ module,
since the $G$-action on each $M/p^nM$
factors through the quotient by some open subgroup $U_n$.

We define the profinite group homology with $\Lambda$ coefficients by tensoring
$M$ with the bar complex $(\Lambda[[G^m]])_m$,
where the tensor product is now completed tensor product,
and taking homology.  That is to say:  
 $$\mbox{$m$-chains for $(G, M)$} =  \varprojlim_{U,n: U \subset U_n}  \frac{M}{p^n M}[ (G/U)^m]$$
 (we refer to \S5--6 of \cite{Zalesskii} for a more complete discussion). 
 
Since this complex is the inverse limit of the complexes computing  
homology of $G/U_n$ acting on $M/p^n$, and taking an inverse limit of a system of
{\em profinite groups} preserves exactness, we have
\begin{equation} \label{hgmdef} H_*(G,M)=\varprojlim_{U, n: U \subset U_n} H_*(G/U, M/p^n).\end{equation}
Finally, we can similarly define relative group homology $H_1(G, H; M)$
for $H \leqslant G$ using the induced map on chain complexes.

One verifies
that the contents of \S \ref{sec:cocycl-univ-extens}
go through when $G, H, \Lambda$ are as just described.
Namely, one has a category $\mathcal{C}_{\Lambda}(G, H; M)$
of extensions   $\begin{tikzcd}
    T \arrow{r} & V \arrow{r}{\pi} & M. \arrow[bend left=33]{l}{s}
  \end{tikzcd}$ where:
  \begin{itemize}
 \item $T, V, M$ are topological
 $\Lambda$-modules with continuous $G$-action (the maps are automatically continuous by definition of the topology). 
 \item $s$ is equivariant for $H$.
 \end{itemize}
  
For any object in this category, there is a  map $H_1(G, H; M) \rightarrow T$ 
which may be constructed in a similar fashion to \eqref{eq:57} (although the kernel of $V/p^n \rightarrow M/p^n$
need not be $T/p^n$, this becomes true after passing to the inverse limit). 
One verifies, as before, that an object is universal if and only if this map $H_1(G, H; M) \rightarrow T$
is an isomorphism.

\begin{theorem} \label{Zpuniversal}
Let $\Gamma = \Gal(H_{\infty}/\Q)$ be the Galois group of $H_\infty = \bigcup H_{p^n}$ over $\Q$,
and $c \in \Gal(H_{\infty}/\Q)$ the conjugation.  
The sequence 
\begin{equation} \label{Zpsequence2} \Ker(c_H) \rightarrow \KSp_{4k-2}(\Z; \Z_p) \xrightarrow{c_H} \Z_p(2k-1)\end{equation} 
of $G$-modules is uniquely split equivariantly for $\langle c \rangle \subset G$. 
 The resulting sequence is initial
 both in the category $C_{\Z_p}(\Gamma, \langle c \rangle ; \Z_p(2k-1))$
 and in the category $C_{\Z_p}(\Gamma; \Z_p(2k-1))$. 
\end{theorem}

\begin{proof}
The induced map
$$ H_1(\Gamma, \langle c\rangle ; \Z_p(2k-1)) \rightarrow \Ker(c_H)$$
is an isomorphism, because
it is an inverse limit of corresponding isomorphisms 
for the sequences \eqref{lemma2seq}.  
That we can ignore $c$ follows from (the profinite group analogue of) Lemma 
\ref{lem: vanishing homology implies same universal object}.
\end{proof}

\subsection{Universal property using full unramified Galois group} \label{univ2} 
We now reformulate the universal property of the extension
with reference to the {\'e}tale fundamental group of $\Z[1/p]$,
or, in Galois-theoretic terms, 
$$G := \mbox{Galois extension of largest algebraic extension $\Q^{(p)}$ of $\Q$ unramified $p$.}$$

For the results that involve explicit splittings of the sequence, we need to carefully choose a decomposition group for $G$. Let
$$ \wp(q) = \{ \mbox{prime ideals of $H_q$ above $p$}\}.$$
The subset $\wp(q)^c$ fixed by complex conjugation is nonempty, 
because $\wp(q)$ has odd cardinality $[H_q:K_q]$.
The sets $\wp(q)^c$ form an inverse system of  nonempty finite sets as one varies $q$ through powers of $p$; since the inverse
limit of such is nonempty, there exists a prime $\mathfrak{p}$
of $H_{\infty} = \bigcup_q H_q$ inducing an element of $\wp(q)^c$
on each $H_{q}$. 
We extend $\mathfrak{p}$ as above to $\Q^{(p)}$ in an arbitrary way.  
 Let 
 $$G_{\mathfrak{p}} \leqslant G$$ be the decomposition group at   $\mathfrak{p}$.
 
 \begin{remark}  In fact, Vandiver's conjecture is {\em equivalent} (for any $n$) to the statement
 that $\wp(q)$ is a singleton.  Indeed, if $c$ fixes two different primes $\mf{p}$ and $\mf{p}'$ in $H_q$ lying over $p$, then conjugation by $c$ preserves the subset $\mrm{Trans}(\mf{p},\mf{p}') \subset \Gal(H_q/K_q)$ which transports $\mf{p}$ to $\mf{p}'$. But since $p$ is totally split in $H_q/K_q$, $\mrm{Trans}(\mf{p},\mf{p}') $ consists of a single element, so conjugation by $c$ must fix a nontrivial element of $\Gal(H_q/K_q)$. Equivalently, by class field theory, $c$ must fix a non-trivial element of the $p$-part of $\Pic(K_q)$, i.e.\ the $p$-part of $\Pic(K_q^+)$ is non-trivial. But Vandiver's conjecture predicts exactly that the $p$-part of $\Pic(K_p^+)$ is trivial, which is equivalent to the statement that the $p$-part of $\Pic(K_q^+)$ is trivial for all $q = p^n$ by \cite[Corollary 10.7]{Wash}.
 \end{remark}

\begin{theorem} \label{thm: Zpuniversal}
 Let $\mathfrak{p}$ be chosen as above. 
The exact sequence 
\begin{equation} \label{Zpsequence} \Ker(c_H) \rightarrow \KSp_{4k-2}(\Z; \Z_p) \xrightarrow{c_H} \Z_p(2k-1)\end{equation} 
of $G$-modules is uniquely split for $G_{\mathfrak{p}}$; the 
kernel of the Betti map maps isomorphically to $\Z_p(2k-1)$ and furnishes this unique splitting.
 The resulting sequence is initial  
 in the category $C_{\Z_p}(G, G_{\mathfrak{p}}; \Z_p(2k-1))$
 and in the category
 $C_{\Z_p}(G, \Z_p(2k-1))^{G_{\mathfrak{p}}\text{-split}}$
 (see Lemma \ref{lem: vanishing homology implies same universal object}).   \end{theorem}

\begin{remark}Let us rephrase this in geometric terms. By virtue of its Galois action $\KSp(\Z; \Z_p)$   
can be considered as (the $\C$-fiber of) an {\'e}tale sheaf over $\Z[1/p]$. 
This structure arises eventually from the fact that the moduli space
of abelian varieties has a structure of $\Z[1/p]$-scheme. The last
assertion of the Theorem can then be reformulated:
\begin{quote}  The {\'e}tale sheaf on $\Z[1/p]$ defined by $\KSp(\Z; \Z_p)$ is the universal extension of $\Z_p(2k-1)$
by a trivial {\'e}tale sheaf which splits when restricted to the spectrum of $\Q_p$.
\end{quote}
More formally we consider the category whose
objects are {\'e}tale sheaves $\mathcal{F}$ over $\Z[1/p]$ equipped with $\pi: \mathcal{F} \twoheadrightarrow \Z_p(2k-1)$
whose kernel is a trivial sheaf, and with the {\em property}
that $\pi$ splits when restricted to $\Spec \ \Q_p$.  Our assertion is that the sheaf defined by $\KSp$,
together with its Hodge morphism to $\Z_p(2k-1)$, is initial in this category.
\end{remark}

We deduce Theorem \ref{thm: Zpuniversal} from Theorem \ref{mt2} in stages. First (Lemma \ref{decomp group})
we replace the role of complex conjugation by a decomposition group. 
Next (Lemma \ref{decomp group2}) we pass from $\Gal(H_q/\Q)$ to $G$. 
Finally   we pass from $\Z/q$ coefficients to $\Z_p$.

\begin{lemma} \label{decomp group}
The sequence 
\[
\Ker(c_H) \rightarrow \KSp_{4k-2}(\Z; \Z/q) \xrightarrow{c_H} \mu_q^{\otimes(2k-1)}
\]
of $\Gal(H_q/\Q)$-modules is uniquely split for the decomposition group $\Gal(H_q/\Q)_{\mathfrak{p}}$,
where $\mathfrak{p} \in \wp(q)^c$ is any prime fixed by complex conjugation.
The resulting sequence with splitting 
is  universal in $\Cal{C}_{\Z/q}(\Gal(H_q/\Q),  \Gal(H_q/\Q)_{\mathfrak{p}}; \mu_q^{\otimes(2k-1)})$.
 \end{lemma}
 Note that, in particular,  any  splitting that is invariant
 by this decomposition group is also invariant by $\langle c\rangle$;
 so the unique splitting referenced in the Lemma is in fact provided by the Betti map. 

\proof 
  The cyclotomic character $\Gal(H_q/\Q) \rightarrow (\Z/q)^\times$
 restricts to an isomorphism on the decomposition group at $\mathfrak{p}$,
  and in particular this decomposition group is abelian;
 thus $c \in \Gal(H_q/\Q)_{\mathfrak{p}}$ is central, and so
   $H_0( \Gal(H_q/\Q)_{\mathfrak{p}}; \mu_q^{\otimes (2k-1)}) = H_1( \Gal(H_q/\Q)_{\mathfrak{p}};
   \mu_q^{\otimes (2k-1)}) = 0$, which
permits us to apply Lemma \ref{lem: vanishing homology implies same universal object}.
\qed

\begin{lemma}\label{decomp group2}
 The sequence 
\begin{equation} \label{lemma2seq} \Ker(c_H) \rightarrow \KSp_{4k-2}(\Z; \Z/q) \xrightarrow{c_H} \mu_q^{\otimes(2k-1)}\end{equation}
 now considered as $G$-modules, is  uniquely split for $G_{\mathfrak{p}}$ by the kernel of the Betti map. 
 The resulting sequence with splitting is  universal in $\Cal{C}_{\Z/q}(G, G_{\mathfrak{p}};  \mu_q^{\otimes(2k-1)})$, defined as in \S \ref{univ1}.
\end{lemma}

\proof 
That the sequence is uniquely split follows from the same property for $\Gal(H_q/\Q)$,
and that this unique splitting comes from $\mathrm{ker}(c_B)$ is 
as argued after Lemma \ref{decomp group}.
  As in \eqref{eq:57} one gets
$H_1(G, G_{\mathfrak{p}}; \mu_q^{\otimes (2k-1)}) \rightarrow \Ker(c_H)$
which we must prove to be an isomorphism. This map factors through the similar map for $\Gal(H_q/\Q)$, and 
so it is enough to show that
the natural map $f$ of pairs of groups:
\begin{equation} \label{nat} (G,G_{\mathfrak{p}}) \stackrel{f}\rightarrow (\Gal(H_q/\Q), \Gal(H_q/\Q)_{\mathfrak{p}}). \end{equation}
 induces an isomorphism on relative $H_1$ with coefficients in $\mu_{q}^{\otimes(2k-1)}$.

 The action on $\Q(\zeta_q)$ gives a surjection $G \twoheadrightarrow (\Z/q)^\times$,
 which factors through $\Gal(H_q/\Q)$ and 
restricts there to an isomorphism $\Gal(H_q/\Q)_{\mathfrak{p}} \cong (\Z/q)^\times$. 
Write $G^0$ for the kernel, and similarly define $$G^0_{\mathfrak{p}} \leqslant G^0, \quad \{e\} = \Gal(H_q/\Q)^{0}_{\mathfrak{p}} \leqslant \Gal(H_q/\Q)^0.$$
From the morphism to $(\Z/q)^\times$ we obtain (as in
the proof of Proposition \ref{propcor:transfer-surjective})
   compatible spectral sequences computing $H_*(G,G_{\mathfrak{p}})$ in terms of $H_*(G^{0}, G^0_{\mathfrak{p}})$ and similarly for $H_*(\Gal(H_q/\Q)^{0}, \Gal(H_q/\Q)_{\mathfrak{p}}^{0})$.

By the same argument as in \eqref{HS5} the maps 
 $$
  (\Z/q)^\times \mbox{ coinvariants on }H_1(G^{0}, G^{0}_{\mathfrak{p}}; \mu_q^{\otimes(2k-1)})  \to H_1(G, G_{\mathfrak{p}}; \mu_q^{\otimes(2k-1)}) $$
  is an isomorphism, and the same for $\Gal(H_q/\Q)$.
  (Here the flanking terms of \eqref{HS5} vanish for even simpler reasons,
  because relative group $H_0$ always vanishes.)
    
Therefore, it is sufficient to verify that
$$f^0: (G^{0}, G^{0}_{\mathfrak{p}}) \rightarrow (\Gal(H_q/\Q)^{0}, \Gal(H_q/\Q)_{\mathfrak{p}}^{0})$$
induces an isomorphism on first homology with $\mu_{q}^{\otimes (2k-1)}$ coefficients.  The  coefficients have trivial action by definition of the groups, and it suffices to consider $\Z_p$ coefficients because relative $H_0$ vanishes.
  But $H_1(G^{0}, G^{0}_{\mathfrak{p}}) \otimes \Z_p$ is the Galois group of the maximal abelian $p$-power extension of $K_q$
  that is unramified everywhere and split at  $\mf{p}$; this coincides with $H_q$
  because  $H_q/K_q$ is already split at $\mf{p}$. 
\qed

\begin{proof}[Proof of of Theorem \ref{thm: Zpuniversal}]
The sequence \eqref{Zpsequence} is the inverse limit of the
sequences \eqref{lemma2seq},
and the existence of a splitting follows from this.  Uniqueness follows from
the fact that $G_{\mathfrak{p}}$ surjects to $(\Z_p)^\times$,
and thus contains an element acting by $-1$ on $\Z_p(2k-1)$ and trivially on $\Ker(c_H)$. 
Finally the induced map
$$ H_1(G, G_{\mathfrak{p}}; \Z_p(2k-1)) \rightarrow \Ker(c_H)$$
is an isomorphism, because
it is an inverse limit of corresponding isomorphisms 
for the sequences \eqref{lemma2seq}. 
 \end{proof}

\subsection{Universal properties of Bott-inverted $K$-theory}\label{subsec:main-theorem-KSPBott}
We have seen that symplectic $K$-theory realizes certain universal extensions of $\mu_q^{\otimes 2k-1}$ as Galois modules, for $k$ a positive integer. It is natural to ask if the universal extensions of other cyclotomic powers is realized in a similar way. Here we explain that for negative $k$, the Bott-inverted symplectic $K$-theory provides such a realization. 

By Corollary \ref{cor: bott inverted symplectic K}, we have short exact sequences for Bott-inverted symplectic $K$-theory (discussed in \ref{sec:bott-invert-sympl}):
  \begin{equation}\label{eq:72}
    0 \to K^{(\beta)}_{4k-2}(\Z;\Z/q) \to \KSp_{4k-2}^{(\beta)}(\Z;\Z/q) \to \mu_q(\C)^{\otimes 2k-1} \to 0
  \end{equation}
  and
  \begin{equation}\label{eq:73}
    0 \to K^{(\beta)}_{4k-2}(\Z;\Z_p) \to \KSp_{4k-2}^{(\beta)}(\Z;\Z_p) \to \Z_p(2k-1) \to 0.
  \end{equation}
Our main theorems have analogues for Bott inverted symplectic $K$-theory:

 \begin{theorem}
 Let $k$ be any (possibly negative!) integer. 
\begin{enumerate}
\item Let notation be as in Theorem \ref{mt2}. The extension \eqref{eq:72} is initial in $\Cal{C}_{\Z/q}(\Gal(H_q/\Q), \langle c \rangle; \mu_q^{\otimes (2k-1)})$. 
\item Let notation be as in Theorem \ref{Zpuniversal}. The extension \eqref{eq:73} is initial in both $C_{\Z_p}(\Gamma, \langle c \rangle ; \Z_p(2k-1))$
 and in the category $C_{\Z_p}(\Gamma, \Z_p(2k-1))$. 
 \item Let notation be as in Theorem \ref{thm: Zpuniversal}. The extension \eqref{eq:73} is initial in $C_{\Z_p}(G, G_{\mathfrak{p}}; \Z_p(2k-1))$.
\end{enumerate}
\end{theorem}

\begin{proof} As above, parts (2) and (3) follow formally from (1) by an inverse limit argument, so it suffices to prove (1). By Proposition \ref{normresidueconsequence}, for positive $k$ these short exact sequences agree with the ones where $\beta$ is not inverted, and hence of course enjoys the same universal property. For non-positive $k$ the universal property for the short exact sequence~\eqref{eq:72} is deduced immediately by periodicity in $k$.
    \end{proof}

 \begin{remark} \label{BottInvertedKtheory} The  natural map
  \begin{equation*}
    \KSp(\Z;\Z_p) \to \KSp_i^{(\beta)}(\Z;\Z_p)
  \end{equation*}
  is an isomorphism whenever $i \geq 0$, but from a conceptual point of view it may be preferable to work entirely with $\KSp^{(\beta)}_*(\Z;\Z_p)$.  For one thing, the universal property for~\eqref{eq:73} is in some ways more interesting, in that we see universal extensions of $\Z_p(2i-1)$ for all $i \in \Z$, not only $i > 0$.  Secondly, the relationship between $K_*^{(\beta)}(\Z;\Z_p)$ and \'etale cohomology of $\Spec(\Z')$ does not depend on the work of Voevodsky and Rost, and therefore not on any motivic homotopy theory.
  \end{remark}
  
\subsection{Degree $4k-1$}
\label{sec:degree-4k-1}
For odd $p$ the  homotopy groups of $\KSp_*(\Z;\Z_p)$ are non-zero only in degrees $* \equiv 2 \mod 4$ and $*\equiv 3 \mod 4$.   We shall prove that the Galois action is trivial in the latter case.

\begin{proof}
There is a homomorphism
  \begin{equation} \label{Quillen}
    \pi_*^s(\text{point}; \Z_p) \to K_*(\Z; \Z_p)
  \end{equation}
  induced by the natural functor $S \mapsto \Z[S]$
  from the symmetric monoidal category of sets (under disjoint union)
  to the symmetric monoidal category of free $\Z$-modules (under direct sum).

 The work of   Quillen in \cite{MR0482758}
 implies that this  map is surjective in degree $4k-1$. 
 More precisely, if we choose an auxiliary prime $\ell$
 for which the  class of $\ell$ topologically generates $\Z_p^{\times}$, 
 then Quillen's work implies that
 the composite map
 $$ \pi_*^s(\text{point}; \Z_p) \rightarrow K_*(\Z; \Z_p) \rightarrow K_*(\F_{\ell}; \Z_p)$$
 is surjective; on the other hand, the latter map is an isomorphism
 by the norm residue theorem\footnote{Appealing to the norm residue theorem  reduces this to checking that the restriction map $H^1(\Z'; \Z_p(2k)) \rightarrow H^1(\F_{\ell}; \Z_p(2k))$
 in {\'e}tale cohomology is an isomorphism.  Since these groups are finite, this map is identified via the connecting homomorphism with the map $H^0(\Z', \Q_p/\Z_p(2k)) \rightarrow H^0(\F_{\ell}, \Q_p/\Z_p(2k))$. Now, $H^0(\Z', \Q_p/\Z_p(2k))$ is invariants of $\Q_p/\Z_p(2k)$ for the $\Z_p^\times$-action via $\chi_{\cyc}^{2k-1}$, and $H^0(\F_{\ell}, \Q_p/\Z_p(2k))$ is the invariants for the subgroup of $\Z_p^\times$ generated by the element $\ell$.
}.

There is also a natural map
  \begin{equation} \label{Quillen2}
    \pi_{*}^s(\text{point}) \to \KSp_{*}(\Z;\Z_p)
  \end{equation}
arising from the
functor of symmetric monoidal categories 
  sending a finite set $S$ to $\Z[S] \otimes (\Z e \oplus \Z f)$, 
 equipped with the symplectic form   $\langle s \otimes e, s' \otimes f \rangle = \delta_{ss'}$. 
 When composed with $c_B$, the map \eqref{Quillen2} 
 recovers twice \eqref{Quillen};  but $p$ is odd and  $c_B$ is injective by Theorem 
  \ref{eq:30}, so it follows that \eqref{Quillen2} is also surjective in degree $4k-1$.

  The proof may now be finished by showing that $\pi_{4k-1}^s(\text{point}) \to \KSp_{4k-1}(\Z;\Z_p)$ is equivariant for the trivial action on the domain.  Indeed, it is induced by $\{\ast\} \to \mathcal{A}_1(\C)$, sending the point to some chosen elliptic curve $E \to \Spec(\C)$.  Since we may choose $E$ to be defined over $\Q$, the map is indeed equivariant for the trivial action on $\pi_{4k-1}^s(\text{point})$.
\end{proof}

\begin{remark} Let us sketch an alternative argument: we will show that the map $\KSp_{4k-1}(\Z;\Z_p) \to K_{4k-1}(\F_\ell;\Z_p)$ (which is a group isomorphism by the Norm Residue Theorem, as in the first proof, plus Theorem \ref{thm:KSp-of-Z}) is equivariant for the trivial action on $K_{4k-1}(\F_\ell;\Z_p)$.   

This map comes from the functor
\[
\Cal{A}_g(\CC) \rightarrow \Cal{SP}(\Z) \rightarrow \proj{\F_\ell}
\]
sending a complex abelian variety $A \to \Spec(\C)$ to the $\F_\ell$-module $H_1(A;\F_\ell)$.  The latter is canonically identified with $A[\ell] \subset A(\C)$, the $\ell$-torsion points.  These are defined purely algebraically and hence the $\ell$-torsion points of $A$ and of $\sigma A$ are \emph{equal} $\F_\ell$-modules for $\sigma \in \Aut(\C)$.  Therefore this composite functor intertwines the natural action of $\Aut(\CC)$ on the \'etale homotopy type of $\Cal{A}_{g,\CC} $ with the \emph{trivial} action on $|\proj{\F_\ell}|$, from which it may be deduced that $\KSp_*(\Z;\Z_p) \to K_*(\F_\ell;\Z_p)$ is indeed equivariant for the trivial action on $K_*(\F_\ell;\Z_p)$.
\end{remark}
 
 \begin{remark}
   Poitou--Tate duality implies the isomorphism 
   \[
   H_2(G,G_{\mathfrak{p}};\Z_p(2k-1)) \cong K_{4k-1}(\Z;\Z_p) = \KSp_{4k-1}(\Z;\Z_p).
   \]
   One may wonder whether the homotopy groups $\KSp_{4k-2}(\Z;\Z_p)$ and $\KSp_{4k-1}(\Z;\Z_p)$ are shadows of one ``derived universal extension'' of $\Z_p(2k-1)$ as a continuous $\Z_p[[G]]$-module split over $G_\mathfrak{p}$.  Or better yet, whether there is a ($G_\mathfrak{p}$-split) sequence of spectra
   \begin{equation*}
     L_{K(1)} K(\Z)^{(+)} \to L_{K(1)} \KSp(\Z) \to K^{(-)}
   \end{equation*}
   in a suitable category of spectra with continuous $G$-actions, characterized by a universal property.
   Here $K = L_{K(1)} ku$ denotes the $p$-completed periodic complex $K$-theory spectrum.
 \end{remark}
 

\section{Families of abelian varieties and stable homology}\label{sec: stable homology}

In this section, we give a more precise version of \eqref{bwmc} from the introduction.
The proof also illustrates a technique of passing to homology from homotopy.

Suppose $\pi: A \rightarrow X$ is a principally polarized abelian scheme over a smooth, $n$-dimensional, projective  
complex variety $X$.  
The Hodge bundle $\omega = \mathrm{Lie}(A)^*$ 
defines a vector bundle $\omega_X$ on $X$, and we obtain characteristic numbers of the family 
by integrating Chern classes of this Hodge bundle. In particular, if $\dim X = n$ then for any partition $\ul{n} = (n_1, \ldots, n_r)$ of $n$ 
with each $n_i$ odd, we have a Chern number

$$s_{\ul{n}}(A/X) := \int_{X} \mathrm{ch}_{n_1} (\omega_X)  \smile \ldots \smile \mrm{ch}_{n_r}(\omega_X) \in \Q.$$
Our results then imply divisibility constraints for the characteristic numbers of such families where  
 $A/X$ is defined over $\Q$ (that is: $A$, $X$ and the morphism $A \rightarrow X$
 are all defined over $\Q$):
 
\begin{theorem}\label{thm: Chern divisiblity} Suppose that $A/X$ is defined over $\Q$.  For each
partition $\ul{n}$ of $n$ as above, the characteristic number $s_{\ul{n}}(A/X)$ is divisible by 
 each prime $p \geq \max_j(n_j)$ such that, for some $i$, $p$ divides the numerator of the Bernoulli number $B_{n_i+1}$ .
\end{theorem}

\proof (Outline): 
  In what follows we shall freely make use of {\'e}tale {\em homology} of varieties and algebraic stacks, which can be defined as the compactly supported cohomology of the dualizing sheaf. 

The assumption $p \geq n_i$ implies that $\mathrm{ch}_{n_i}$ lifts to a $\Z_{(p)}$-integral class.  In particular we have universal $\mathrm{ch}_{n_i}(\omega) \in H^{2n_i}(\mathcal{A}_{g,\C};\Z_p(n_i))$.
The family $A/X$ induces a classifying map $f \co X \rightarrow \Cal{A}_{g,\C}$, hence a cycle class in  $H_{2n}(\Cal{A}_{g,\C}; \Z_p)$ transforming according to
the $n$th power of the cyclotomic character; more intrinsically we get an equivariant $\Z_p(n) \rightarrow H_{2n}(\Cal{A}_{g,\C}; \Z_p)$.  Taking cap product with $\prod_{j \neq i} \mathrm{ch}_{n_j}(\omega)$ gives
 $\Z_p(n_i) \longrightarrow  H_{2n_i}(\Cal{A}_{g,\C}; \Z_p)$,
whose pairing with $\mrm{ch}_{n_i}(\omega)$ is the Chern number $s_{\ul{n}}(A/X) \in \Z_{(p)} = \Q \cap \Z_p$.
Assuming for a contradiction that this number is not divisible by $p$, the  morphism 
\begin{equation}
\label{Zpn1}
H_{2n_i}(\Cal{A}_{g,\C}, \Z_p) \xrightarrow{\mathrm{ch}_{n_i}(\omega)} \Z_p(n_i)
\end{equation}
 {\em splits} as a morphism of Galois modules. 

We wish to pass from this homological statement to a $K$-theoretic one. To do so
 we  use some facts about stable homology
 \[
H_i(\Sp_{\infty}(\Z);  \Z_p) = \varinjlim_g H_i(\Sp_{2g}(\Z); \Z_p)
\]
(the limit stabilizes  for $i < (g - 5)/2$ by  \cite[Corollary 4.5]{Cha87})
that will be explained in the next subsection. 
 This stable homology 
carries a Pontryagin product, arising from the natural maps
$\Sp_{2a} \times \Sp_{2b} \rightarrow \Sp_{2a+2b}$; 
in particular we  can define the ``decomposable elements'' of $H_i$
as the $\Z_p$-span of all products $x_1 \cdot x_2$ where $x,y$ have strictly  positive  degree, 
and a corresponding quotient space of ``indecomposables.'' 

We will be interested
in a variant that is better adapted to $\Z_p$ coefficients. 
Define {\em integral  decomposables}  in $H_*(\Sp_{\infty}(\Z); \Z_p)$ as
the $\Z_p$-span of all $x_1 \cdot x_2$ {\em and }$\beta(x_1' \cdot x_2')$ where $x_i \in H_*(\Sp_{\infty}(\Z); \Z_p),  x_i' \in H_*(\Sp_{\infty}(\Z); \Z/p^k)$
have positive degree and $\beta$ is the Bockstein induced  from the sequence
$
0 \rightarrow \Z_p \rightarrow \Z_p \rightarrow \Z/p^k \Z \rightarrow 0$.  Correspondingly
this permits
us to define an indecomposable quotient
 $$H_i(\Sp_{\infty}(\Z); \Z_p)_{\mathrm{Ind}} = \frac{ H_i(\Sp_{\infty}(\Z); \Z_p)}{\mbox{integral decomposables}}.$$

Then the fact that we shall use (generalizing a familiar
property of {\em rational} $K$-theory,  see  \cite[Theorem 1.4]{Nov66}))
is that the composite of the Hurewicz map and the quotient map 
\begin{equation}  \label{eq: K = ind stable homology} \KSp_i(\Z) \otimes \Z_p \longrightarrow  H_i(\Sp_{\infty}(\Z); \Z_p)_{\mathrm{Ind}}\end{equation}
is an {\em isomorphism} for $i \leq 2p-2$.  We sketch the proof of this fact in \S \ref{sec:stable-homology-its} below.
In particular, there is a map $H_{2n_i}(\Sp_{2g}(\Z); \Z_p) \rightarrow \KSp_{2n_i}(\Z; \Z_p)$
(by mapping to stable homology followed by $\text{\eqref{eq: K = ind stable homology}}^{-1}$);
this map is Galois equivariant and  intertwines \eqref{Zpn1}
with the Hodge map $c_H$, and therefore
 the sequence \eqref{Zpsequence2} is also split. 

 But this gives a contradiction:
By Theorem \ref{Zpuniversal}, the sequence \eqref{Zpsequence2} is non-split  as long as 
$\ker(c_H)$ is nonzero, which by Theorem \ref{thm:KSp-of-Z}
is the case precisely when
$H_{\et}^2(\Z[1/p]; \Z_p(n_i+1)) \neq 0$,
which  by Iwasawa theory (see
\cite[Cor 4.2]{Kolster})  
is the case precisely when $p$ divides  the numerator of $B_{n_i+1}$.
\qed

\begin{remark}
One can explicitly construct various examples of this situation, e.g.:
\begin{itemize}
\item[(i)] We can take $X$ to be a projective Shimura variety of PEL type; the simplest example
is a Shimura curve parameterizing abelian varieties with quaternionic multiplication. 

\item[(ii)] There exist many such families of curves, i.e.\ 
embeddings of smooth proper $X$ into $\mathcal{M}_g$, and then
the Jacobians form a (canonically principally polarized) abelian scheme over $X$.
 
\item[(iii)] The (projective) Baily-Borel compactification of $\mathcal{A}_g$ has 
a boundary of codimension $g$; consequently, thus, a generic $(g-1)$-dimensional hyperplane section 
gives a variety $X$ as above. 
 \end{itemize} 
It should be possible to dirctly verify the divisibility at
least in examples (i) and (ii), where it is related to (respectively)
divisibility in the cohomology of the Torelli map   $\mathcal{M}_g \rightarrow \mathcal{A}_g$
and the occurrence of $\zeta$-values in volumes of Shimura varieties. 
\end{remark}

\subsection{Stable homology and its indecomposable quotient}
\label{sec:stable-homology-its}

The proof of \eqref{eq: K = ind stable homology} is a consequence of a more general fact about infinite loop spaces,
formulated and proved in Theorem \ref{primitive homology to indecomposable isomorphism}.
Let $E$ be a $p$-complete connected spectrum (all homotopy in strictly
positive degree) and let $X = \Omega^\infty E$ be the corresponding infinite loop space.  We consider the composition
\begin{equation}
  \label{eq:the-map}
  \pi_i(E) = \pi_i(X) \to H_i(X;\Z_p) \to H_i(X;\Z_p)_{\mathrm{Ind}}
\end{equation}
of the Hurewicz homomorphism and the quotient by ``integral indecomposables.''
As above, the latter space is defined as 
 $I^2 + \sum \beta_k I_k^2$
 where:
 \begin{itemize}
 \item $I$ is the kernel of the augmentation  $H_*(X;\Z_p)  \to \Z_p = H_*(\pt;\Z_p)$;
 \item $I_k$ is the kernel of the similarly defined $H_*(X;\Z/p^k\Z) \to \Z/p^k\Z$;
 \item $\beta_k \co H_*(X; \Z/p^k\Z) \rightarrow H_{*-1}(X; \Z_p)$
 is the Bockstein operator associated to the short exact sequence
$
0 \rightarrow \Z_p \rightarrow \Z_p \rightarrow \Z/p^k \Z \rightarrow 0$.
 \end{itemize}

\begin{example}
  Let $E$ be the Eilenberg--MacLane spectrum with
  $E_k = K(\Z/2\Z, k+1)$, so that $X = \mbf{RP}^\infty = K(\Z/2\Z,1)$.  As is well known, $H_*(X;\Z/2\Z)$
  is a divided power algebra over $\Z/2\Z$ and $H_*(X;\Z_2)$ is
  additively $\Z/2\Z$ in each odd degree.  Hence $I^2 = 0$ for degree
  reasons so $I/I^2 = I$ is the entire positive-degree homology.  In
  contrast, $I_1 = H_{>0}(X;\Z/2\Z)$ is one-dimensional in all positive
  degrees whereas $I_1^2 \subset I_1$ is one-dimensional when the
  degree is not a power of two, but zero when the degree is a power of
  two.

  Since $\beta_1: H_*(X;\Z/2\Z) \to H_*(X;\Z_2)$ is surjective in
  positive degrees we may deduce that in this case the integral
  indecomposables
  \begin{equation*}
    I/(I^2 + \beta_1(I_1^2))
  \end{equation*}
  is $\Z/2\Z$ in degrees of the form $2^i - 1$ and vanishes in all other
  degrees.
\end{example}

 \begin{theorem}\label{primitive homology to indecomposable isomorphism}
  For $X = \Omega^\infty E$ as above, the homomorphism~(\ref{eq:the-map}) is
  an isomorphism in degrees $* \leq 2p-2$.
\end{theorem}

\begin{proof}[Proof of Theorem \ref{primitive homology to indecomposable isomorphism}]
  Let us first sketch why the result is true when $X$ is a connected Eilenberg--MacLane space, i.e.\ $X = K(\Z_p,n)$ or $K(\Z/p^k\Z,n)$ for $n \geq 1$.  In that case $X$ has the structure of a topological abelian group, and the singular chains $C_*(X;\Z_p)$ form a graded-commutative differential graded algebra (cdga).  In the case $X = K(\Z_p,n)$ there is a cdga morphism
  \begin{equation*}
    A = \Z_p[x] \to C_*(X;\Z_p)
  \end{equation*}
  from the free cdga on a generator $x$ in degree $n$, and in the case $X = K(\Z/p^k\Z,n)$ a cdga morphism
  \begin{equation*}
    A = \Z_p[x,y \mid \partial y = x] \to C_*(X;\Z_p)
  \end{equation*}
  where $x$ has degree $n$ and $y$ has degree $n+1$, in both cases inducing an isomorphism $H_n(A) \to H_n(X;\Z_p)$.  In both cases the mapping cone is acyclic in degrees $* \leq 2p-1$, as follows either by inspecting the explicitly known $H^*(X;\F_p)$, see \cite[Theorems 4,5,6]{Car54} (also stated in e.g.\ \cite[Theorem 6.19]{McCleary}), or by an induction argument (the case $n=1$ is a calculation of mod $p$ group homology of $\Z_p$ or $\Z/p^k$, the induction step is by the Serre spectral sequence).  The map therefore induces an isomorphism in homology in degrees $* \leq 2p-2$, and it is clear that $H_*(A)_{\mathrm{Ind}}$ is generated by the class of $x$.

  A general $X = \Omega^\infty E$ may be replaced by its Postnikov truncation $\tau_{\leq 2p-2} X$, which splits as a product of Eilenberg--MacLane spaces, up to homotopy equivalence of loop spaces.  Indeed, the deloop $\Omega^\infty \Sigma E$ can have non-vanishing homotopy in degrees $* \geq 2$ only, and the shortest possible $k$-invariant is $\mathcal{P}^1: K(\Z/p,2) \to K(\Z/p,2p)$.  Hence $\tau_{\leq 2p-2} X$ splits as
  a product of Eilenberg--MacLane spaces, and this splitting respects
  the $H$-space structure.  Hence the result follows by induction from
  the final Lemma \ref{indecomposables of product}.\end{proof}
  
\begin{lemma}\label{indecomposables of product}
If $X$ and $Y$ are connected $H$ spaces of finite type, then the natural
  map 
  \[
  H_*(X; \Z_p)_{\Ind} \oplus H_*(Y; \Z_p)_{\Ind} \to H_*(X\times Y; \Z_p)_{\Ind}
  \]
  is an isomorphism.
\end{lemma}

\begin{proof}
  There are ``natural maps'' induced by $X \to X \times Y$,
  $Y \to X \times Y$, and the two projections, all of which are $H$-space maps.  A formal argument shows that one composition gives the
  identity map of $H_*(X; \Z_p)_{\Ind} \oplus H_*(Y; \Z_p)_{\Ind}$ and hence that
  $H_*(X; \Z_p)_{\Ind} \oplus H_*(Y; \Z_p)_{\Ind} \to H_*(X \times Y; \Z_p)_{\Ind}$ is injective.  (This much would
  also be true if we only take quotient by $I^2$ and not all the
  $\beta_k(I_k^2)$.)

  The algebra $H_*(X \times Y)$ may be calculated additively by the
  Kunneth formula, and the main issue in this lemma is to deal with
  non-vanishing Tor terms.  By the finite type assumption, the
  homology of $X$ and $Y$ will be direct sums of groups of the form
  $\Z_p$ and $\Z/p^k\Z$.  Pick such a direct sum decomposition.  Then
  each $\Z/p^k\Z$ summand in $H_*(X)$ and $\Z/p^l\Z$ summand in $H_*(Y)$
  pair to give a $\Z/p^d\Z$ summand in the Tor term, where
  $d = \min(k,l)$.  It is not hard to see that this summand must be in
  the $\beta_d(I_d^2)$.  Indeed, a generator may be chosen as
  $\beta_d(xy)$, where $\beta_d(x)$ and $\beta_d(y)$ are generators
  for the $p^d$ torsion in the $\Z/p^k\Z$ and $\Z/p^l\Z$ summands of
  $H_*(X)$ and $H_*(Y)$ respectively.

  We have shown that the multiplication map
  $H_*(X) \otimes H_*(Y) \to H_*(X \times Y)$ becomes surjective after
  taking quotient by $\sum \beta_k(I_k^2)$, but then it must remain
  surjective after passing to augmentation ideals and taking further
  quotients.
\end{proof} 
 
\appendix  

\newcommand{\Map}{\mathrm{Map}}

 \newcommand{\Var}{\mathrm{Var}}

\section{Construction of the Galois action on symplectic $K$-theory}
\label{sec:stable-homot-theory2}

The goal of this Appendix is to supply details for an argument sketched in the main text, viz.\ the construction of the Galois action on symplectic $K$-theory in the proof of Proposition \ref{prop:5.2}.
The contents are as follows:
 
 In Subsection \S \ref{sec:gamma-spaces} we review notation and basic definition concerning $\Gamma$-spaces, and in Subsection~\ref{sec:homot-types-compl} we review two ways to to extract a space from a simplicial scheme quasiprojective over $\Spec(\C)$. One might be called ``Betti realization'' and the other ``\'etale realization''.  Then in \S \ref{sec:etale-homotopy-type} we explain how to relate Betti realization with \'etale realization after completing at a prime $p$.  As usual, the point is that the \'etale realization of objects base changed from $\Spec(\Q)$ inherits an action of the group $\Aut(\C)$ of all field automorphisms of the complex numbers.  

The main construction happens in \S \ref{sec:gamma-objects-simpl}, where a certain $\Gamma$-object $Z$ in simplicial schemes quasi-projective over $\Spec(\Q)$ is constructed.  We prove that the $\Gamma$-space resulting from base changing $Z$ from $\Q$ to $\C$ and taking Betti realization gives a model for $\KSp(\Z)$, the symplectic $K$-theory spectrum studied in this paper.  This eventually boils down to the Betti realization of $\mathcal{A}_g(\C)^\mathrm{an}$ being a model for $B\Sp_{2g}(\Z)$ ``in the orbifold sense'',  which, in turn, is deduced from uniformization of principally polarized abelian varieties over $\C$ and the contractibility of Siegel upper half-space $\mathbb{H}_g$, as we discuss in\S \ref{sec:betti-real-mathc}.  The result is a model for the $p$-completion of the spectrum $\KSp(\Z)$ on which $\Aut(\C)$ acts by spectrum maps, as we conclude in \S \ref{sec:galo-acti-sympl}.

\subsection{Gamma spaces and deloopings of algebraic $K$-theory spaces}
\label{sec:gamma-spaces}

\newcommand{\Sets}{\mathrm{Sets}}

We summarize a convenient formalism for constructing infinite loop structures on certain spaces, and to promote certain maps to infinite loop maps, introduced by G.\ Segal (\cite{MR353298}) and further developed by Bousfield--Friedlander (\cite{BousfieldFriedlander}) and others.
\begin{definition}
Let $\Gamma^\mathrm{op}$ denote a skeleton of the category whose objects are finite pointed sets and whose morphisms are pointed maps.    Let $s\Sets_\ast$ denote the category of pointed simplicial sets.
  A \emph{$\Gamma$-space} is a functor $X: \Gamma^\mathrm{op} \to s\Sets_\ast$ sending the terminal object $\{\ast\}$ to a terminal simplicial set (one-point set in each simplicial degree).  A morphism of $\Gamma$-spaces is a natural transformation of such functors.
\end{definition}
There is then a functor
\begin{equation} \label{Binftydef}
  B^\infty: \text{$\Gamma$-spaces} \to \text{connective spectra}.
\end{equation}
Under extra assumptions on the $\Gamma$-space $X$, there is also a way to recognize  $\Omega^\infty B^\infty X$ in terms of $X(S^0)$, the value of the functor $X$ on the pointed set $S^0 := \{0, \infty\}$ with basepoint $\infty$.

The ``infinite delooping'' functor $B^\infty$ is easy to define.  Following \cite{BousfieldFriedlander}, we first extend $X: \Gamma^\mathrm{op} \to s\Sets_*$ to a functor
$
  X: s\Sets_* \to s\Sets_*
$ which preserves filtered colimits and geometric realization.  Such an extension is unique up to unique isomorphism, and automatically preserves pointed weak equivalences.  There are canonical maps
$X(S^n) \to \Omega X(S^{n+1})$
and hence
\begin{equation}
  \label{eq:38}
  |X(S^n)| \to \Omega |X(S^{n+1})|,
\end{equation}
where $S^1$ denotes the simplicial circle, and $S^n = (S^1)^{\wedge n}$ the simplicial $n$-sphere.  See e.g.\ \cite[Section 4]{BousfieldFriedlander} for more details.  These maps let us functorially associate a spectrum to each $\Gamma$-space $X$, and the spectra arising this way are automatically connective.

\begin{definition}
  The coproduct of two pointed sets $S$ and $T$ is denoted $S \vee T$ and traditionally called the wedge sum.  $\vee$ gives a symmetric monoidal structure on $\Gamma^\mathrm{op}$, and any object is isomorphic to a finite wedge sum $S^0 \vee \dots \vee S^0$.

  The $\Gamma$-space $X$ is \emph{special} if for any two objects $S,T$ the canonical map
  \begin{equation*}
    X(S \vee T) \to X(S) \times X(T),
  \end{equation*}
  is a weak equivalence.
\end{definition}
When $X$ is a special $\Gamma$-space, the pointed simplicial set $X(S^0)$ may be thought of as the underlying space of $X$.  The fold map $S^0 \vee S^0 \to S^0$ induces a diagram
\begin{equation}\label{eq:40}
  X(S^0) \times X(S^0) \xleftarrow{\simeq} X(S^0 \vee S^0) \to X(S^0),
\end{equation}
which makes $|X(S^0)|$ into an $H$-space, which is unital, associative, and commutative up to homotopy.  In particular the pointed set $\pi_0(|X(S^0)|)$ inherits the structure of a commutative monoid.  As shown by Segal, the maps~(\ref{eq:38}) are weak equivalences for $n \geq 1$ when $X$ is special, so in that case $B^\infty X$ is equivalent to an $\Omega$-spectrum with 0th space $\Omega |X(S^1)|$ and $n$th space $|X(S^n)|$ for $n \geq 1$.  We then have a map of $H$-spaces
\begin{equation}\label{eq:76}
  |X(S^0)| \to \Omega |X(S^1)| \xrightarrow{\simeq} \Omega^\infty B^\infty X
\end{equation}
which is a ``group completion'', in the sense that it induces an isomorphism
\begin{equation*}
  H_*(X(S^0))[\pi_0(X(S^0))^{-1}] \xrightarrow{\sim} H_*(\Omega |X(S^1)|),
\end{equation*}
whose domain is $H_*(X(S^0))$, made into graded-commutative ring using~(\ref{eq:40}), and localized at the multiplicative subset $\pi_0(X(S^0))$.  A similar localization holds with (local) coefficients in any $\Z[\pi_2(X(S^1))]$-module.

Many spectra may be constructed this way.  We list some examples relevant for this paper.
\begin{example}
  For any pointed simplicial set $M$, consider the $\Gamma$-space
  \begin{equation*}
    S \mapsto S \wedge (M_+),
  \end{equation*}
  where $M_+$ denotes $M$ with a disjoint basepoint added.  The corresponding spectrum is then the (unbased) suspension spectrum $\Sigma^\infty_+ M$ mentioned earlier.

There is a natural map of spectra
  \begin{equation}\label{eq:39}
    \Sigma^\infty_+ |X(S^0)| \to B^\infty X,
  \end{equation}
  natural in the $\Gamma$-space $X$, constructed as follows. For any finite pointed set $S$ and any $s \in S$ we have a map $S^0 \to \{s,\ast\}$ sending the non-basepoint to $s$.  If $X$ is a $\Gamma$-space we may apply $X$ to the composition $S^0 \to \{s,\ast\} \subset S$ to get a map $\{s\} \times X(S^0) \to X(S)$ for each $s \in S$.  These assemble to a canonical map from $S \times X(S^0)$ which factors as
  \begin{equation*}
    S \wedge X(S^0)_+ \to X(S).
  \end{equation*}
  This map is natural in $S \in \Gamma^\mathrm{op}$, i.e.,\ defines a map of $\Gamma$-spaces and hence gives rise to a map of spectra.  On homotopy groups it induces a map from the stable homotopy groups of $|X(S^0)|$ to the homotopy groups of $B^\infty X$.
\end{example}

\newcommand{\projj}[2]{\mathcal{P}_{#2}(#1)}  

\begin{example}[Constructing the algebraic $K$-theory spectrum]\label{ex:deloop-K-of-R}
  Following Segal, let us explain how to use $\Gamma$-space machinery to construct algebraic $K$-theory spectra $K(R)$ for a ring $R$.  The idea is to construct a special $\Gamma$-space whose value on $S^0$ is equivalent to $|\proj{R}|$, the classifying space of the groupoid of finitely generated projective $R$-modules.  Its value on $\{\ast, 1,\dots, n\}$ should be a classifying space for a groupoid of finitely generated projective modules equipped with a splitting into $n$ many direct summands.  

Let $S \in \Gamma^\mathrm{op}$ and let $R^S$ denote the ring of all functions $f: S \to R$ under pointwise ring operations.  The diagonal $R \to R^S$ makes any $R^S$-module into an $R$-module.  Let us for $s \in S$ write $e_s \in R^S$ for the idempotent with $e_s(s) = 1$ and $e_s(S \setminus\{s\}) = \{0\}$.  Then for projective $R^S$-module $M$ has submodules $e_s M \subset M$ and the canonical map $\oplus_{s \in S} M_S \to M$ is an isomorphism.  Hence each $M_s$ is a projective $R$-module (for the diagonal $R$-structure).  Let us write $e = 1 - e_{\ast} = \sum_{s \in S\setminus\{\ast\}} e_s \in R^S$ so that $eM = \sum_{s \in S \setminus \{\ast\}} e_s M$, and let
$\projj{R}{S}$ be the category whose objects are pairs $(n,\phi)$ with $n \in \N$ and $\phi: R^S \to M_n(R)$ an $R$-algebra homomorphism, and whose morphisms $(n,\phi) \to (n',\phi')$ are $R^S$-linear isomorphisms $\phi(e)R^n \to \phi'(e) R^{n'}$.    The forgetful functor
  \begin{align*}
    \projj{R}{S^0} & \to \proj{R}\\
    (n,\phi) & \mapsto \phi(e) R^n
  \end{align*}
  is then an equivalence of categories, since any finitely generated projective module is isomorphic to a retract of $R^n$ for some $n$.  Moreover the association
  \begin{equation*}
    S \mapsto \projj{R}{S}
  \end{equation*}
  extends to a functor from $\Gamma^\mathrm{op}$ to groupoids:  a morphism $f: S \to T$ is sent to the functor $\projj{R}{S} \to \projj{R}{T}$ which on objects sends $(n,\phi) \to (n,\phi \circ (f^*))$, where $f^*: R^T \to R^S$ is precomposing with $f$.  We emphasize that composition of morphisms in $\Gamma^\mathrm{op}$ is carried to composition of functors \emph{on the nose} (not just up to preferred isomorphism). That is, $S \mapsto \projj{R}{S}$ is a functor to the 1-category  of small groupoids.

  For $S = \{\ast,1,\dots, n\}$ the restriction functors induce an equivalence of groupoids
  \begin{equation*}
    \projj{R}{S} \to \prod_{i = 1}^n \projj{R}{\{\ast,i\}} \simeq \big(\proj{R}\big)^n.
  \end{equation*}
  It follows that $S \mapsto N(\projj{R}{S})$ is a special $\Gamma$-space and the corresponding spectrum is a model for $K(R)$.  The map~\eqref{eq:76} is a model for the canonical group-completion map
$
    |\proj{R}| \to \Omega^\infty K(R)$ 
  mentioned in Subsection~\ref{sec:algebraic-k-theory-1}.
\end{example}

\begin{example}[Constructing the symplectic $K$-theory spectrum]\label{ex:KSp-as-Gamma-space}
  Finally, let us discuss the spectrum $\KSp(\Z)$, where we are looking for a $\Gamma$-space with $X(S^0) \simeq N(\SP(\Z))$.  The idea is similar to $S \mapsto \projj{\Z}{S}$.  Recall that the objects of $\projj{\Z}{S}$ are $\Z^S$-modules $M$ whose underlying $\Z$-module is equal to $\Z^n$ for some $n \in \N$.  Let objects of $\SP_S(\Z)$ be pairs of an object $M \in \projj{\Z}{S}$ and a symplectic form $b: M \times M \to \Z$ for 
  which the action of $\Z^S$ is by symmetric endomorphisms, i.e. $b(rm_1, m_2) = b(m_1, r m_2)$ for $r \in \Z^S$.
 
  This defines a functor from $\Gamma^\mathrm{op}$ to the 1-category of small groupoids, as before.  We obtain a $\Gamma$-space $S \mapsto N(\SP_S(\Z))$, whose associated spectrum is $\KSp(\Z)$ and infinite loop space is a model for $\Z \times B\Sp_\infty(\Z)^+$.
\end{example}

\subsection{Homotopy types of complex varieties}
\label{sec:homot-types-compl}

Let 
us review various ``realization functors'' assigning a complex scheme $X \to \Spec (\C)$.  We shall mostly assume that $X$ is a \emph{variety}, which we define as follows.
\begin{definition}
  Let $\Var_\C$ denote the category of schemes over $\Spec(\C)$ which are coproducts of quasi-projective schemes.
\end{definition}

The realization functors we need may be summarized in a diagram of simplicial sets
\begin{equation}\label{eq:77}
  \begin{tikzcd}
    X(\C) = \Sing^\mathrm{an}_0(X) \rar & \Sing^\mathrm{an}(X) \ar[r,dashed] & \Et_p(X),
  \end{tikzcd}
\end{equation}
where the dashed arrow indicates a zig-zag of the form
\begin{equation*}
  \Sing^\mathrm{an}(X) \xleftarrow{\simeq} \dots \to \Et_p(X).
\end{equation*}
As we shall explain in more detail below, the ``Betti realization'' has $n$-simplices $\Sing^\mathrm{an}_n(X)$ the set of maps $\Delta^n \to X(\C)$ which are continuous in the analytic topology on $X(\C)$.  Therefore the homotopy type of $\Sing^\mathrm{an}(X)$ encodes the weak homotopy type of the space $X(\C)$ equipped with its \emph{analytic topology}.  Less interestingly, $X(\C) = \Sing^\mathrm{an}_0(X)$ is the set of complex points regarded as a constant simplicial set, encoding the homotopy type of $X(\C)$ in the discrete topology.  Finally, the ``$p$-completed \'etale realization'' $\Et_p(X)$ is a model for the \emph{\'etale homotopy type} of $X$, introduced by Artin and Mazur \cite{ArtinMazur}, or rather its $p$-completion.

We obtain similar realization functors when $X \in s\Var_\C$ is a \emph{simplicial} complex variety, i.e.\ a functor $\Delta^\mathrm{op} \to \Var_\C$.  We will make use of the following properties of these realization functors.
\begin{enumerate}[(i)]
\item\label{item:9} $\Sing^\mathrm{an}_0(X)$ and $\Et_p(X)$ are functorial with respect to commutative diagrams
  \begin{equation}\label{eq:78}
    \begin{tikzcd}
      X \dar\rar["f"] & X' \dar\\
      \Spec(\C) \rar["\sigma"] & \Spec(\C),
    \end{tikzcd}
  \end{equation}
  in which $\sigma$ is any automorphism of $\C$, and the composition~\eqref{eq:77} is a natural transformation of such functors.
\item\label{item:10} $\Sing^\mathrm{an}(X)$ is functorial with respect to diagrams of the form~\eqref{eq:78} where $\sigma \in \Aut(\C)$ is a \emph{continuous} field automorphism (that is, $\sigma$ is either the identity or complex conjugation), and all arrows in~\eqref{eq:77} are natural transformation of such functors.
\item\label{item:11} The map $\Sing^\mathrm{an}(X) \to \Et_p(X)$ induces an isomorphism in mod $p$ homology, at least when $H^1(\Sing^\mathrm{an}(X);\F_p) = 0$.
\item\label{item:12} If $X_{g,\C} \to \Spec(\C)$ is the simplicial variety arising from an atlas $U \to \mathcal{A}_{g,\C}$, then $\Sing^\mathrm{an}(X_g) \simeq B\Sp_{2g}(\Z)$.  Moreover, under this equivalence the maps $\mathcal{A}_g \times \mathcal{A}_{g'} \to \mathcal{A}_{g + g'}$ defined by taking product of principally polarized abelian varieties correspond to the symmetric monoidal structure on $\SP(\Z)$ given by orthogonal direct sum.
\end{enumerate}

It is essentially well known that realization functors with these properties exist.  In particular, the isomorphism between mod $p$ cohomology of $\Sing^\mathrm{an}(X)$ and $\Et_p(X)$ is a combination of Artin's comparison theorem relating \'etale cohomology with finite constant coefficients to Cech cohomology with finite constant coefficients, and the isomorphism between Cech cohomology and singular cohomology.  We shall use two aspects which are perhaps slightly less standard, so we outline the constructions in subsection~\ref{sec:homot-types-compl} below.  Firstly, the \'etale homotopy type usually outputs a pro-object, but it is convenient for us to have a genuine simplicial set.  Secondly, as stated in~(\ref{item:9}), we shall make usage of the fact that $X(\C) = \Sing^\mathrm{an}_0(X)$ is more functorial than the entire $\Sing^\mathrm{an}(X)$.  This last property is used only for the verification of commutativity of~\eqref{eq:12}.

The reader willing to accept on faith (or knowledge) that realization functors with these properties exist may skip ahead to \ref{sec:gamma-objects-simpl} to see how to complete the proof of Proposition~\ref{prop:5.2}.

\subsubsection{Betti realization}
\label{sec:betti-realization}

A complex scheme $X \to \Spec(\C)$ is \emph{quasi-projective} if it is isomorphic (as a scheme over $\Spec(\C)$) to an intersection of a Zariski open and a Zariski closed subset of $\PS^N_\C$ for some $N$.  The resulting embedding $X \to \PS^N_\C$ induces an injection of complex points
\begin{equation*}
  X(\C) \hookrightarrow \PS^N_\C(\C) = \C P^N,
\end{equation*}
and the set of complex points $X(\C)$ inherits the \emph{analytic topology} as a subspace of $\C P^N$, itself the quotient topology from the Euclidean topology on $\C^{N+1} \setminus \{0\}$.  We shall write $X(\C)^\mathrm{an}$ for this topological space, which is Hausdorff and locally compact, and also locally contractible (as follows from Hironaka's theorem that it is triangulable \cite{Hir75}).

For any compact Hausdorff space $\Delta$ we have have a $\C$-algebra $\C^\Delta$ of functions $\Delta \to \C$ that are continuous in the Euclidean topology.  
There is a canonical   function \begin{equation*}
 e_{\Delta}: \Delta \hookrightarrow \Spec(\C^\Delta)(\C),
\end{equation*}
sending a point of $\Delta$ to the point corresponding to the evaluation homomorphism $\C^\Delta \to \C$.  

\begin{lemma}
For any scheme $X$ over $\C$, and any compact Hausdorff space $\Delta$,  the map
\[
\begin{tikzcd}
  \text{maps $\Spec(\C^\Delta) \to X$ of schemes over $\Spec(\C)$} \ar[d]\\
  \text{maps $\Delta \to X(\C)$ continuous in the analytic topology} 
  \end{tikzcd}
  \]
induced by precomposition with $e_{\Delta}$ is a bijection.
  \end{lemma}
  
  \begin{proof}
  We describe the inverse.
Take $f: \Delta \rightarrow X(\C)$ and choose an affine cover $U_i$
  of $X$, and take $V_i = f^{-1}(U_i)$; choose a partition of unity $1=\sum g_i$
  on $\Delta$ where $\mathrm{supp}(g_i) \subset V_i$. 
  The $g_i$ generate the unit ideal of $\C^{\Delta}$, i.e.\ the spectrum of $\C^{\Delta}$ is the union of the open
  affines corresponding to rings $\C^{\Delta}[g_i^{-1}]$. 
We obtain
  $$ \mbox{regular functions on $U_i$} \rightarrow \mbox{continuous functions on $V_i$} \rightarrow \C^{\Delta}[g_i^{-1}].$$
  where the last map sends a continuous function $h$ on $V_i$ to $(h g_i) \cdot g_i^{-1}$,
  where we extend by zero off $V_i$ to make $hg_i$ a function on $\Delta$. Dually we obtain
  $$ \Spec \C^{\Delta}[g_i^{-1}] \longrightarrow U_i,$$
These morphisms glue to the desired map $\Spec \C^{\Delta} \rightarrow X.$ \end{proof}
 
In particular, the simplicial set $\Sing(X(\C)^\mathrm{an})$ may be written in terms of the functor $X: \text{$\C$-algebras} \to \Sets$ as
\begin{equation*}
  \Sing_n(X(\C)) = X(\C^{\Delta^n}) = \mathrm{Maps}_\text{$\C$-schemes} (\Spec(\C^{\Delta^n}),X).
\end{equation*}
where $\Delta^n$ is as usual the (topological) $n$-simplex. 
Motivated by this observation, we make the following more general definition.
\begin{definition}
  Let $X$ be a simplicial complex variety, or more generally any functor from $\C$-algebras to simplicial sets.  The \emph{analytic homotopy type} (or ``Betti realization'') of $X$ is the simplicial set $\Sing^\mathrm{an}(X)$ defined by
  \begin{equation*}
    \Sing^\mathrm{an}_n(X) = X(\C^{\Delta^n})_n. 
  \end{equation*}
  In other words, $\Sing^\mathrm{an}(X)$ is the diagonal of the simplicial set $([n],[m]) \mapsto \Sing_n(X_m(\C)^\mathrm{an})$. 
\end{definition}

\subsubsection{Etale homotopy type and $p$-adic comparison}
\label{sec:etale-homotopy-type}
 
The theory of \'etale homotopy type assigns a   pro-simplicial set\footnote{The original approach of Artin and Mazur \cite{ArtinMazur}  assigns to $X$ a pro-object in the homotopy category of simplicial sets, which was rigidified in later approaches \cite{Fried82} to output a pro-object in simplicial sets.} $\Et(X)$ functorially to any (locally Noetherian) scheme $X$, 
where 
\begin{equation} \label{cohid} H^*(\Et(X), A) \simeq H^*_{\et}(X, A)\end{equation}
for finite abelian groups $A$.  We will
outline how to modify this construction so as to assign an actual simplicial set $\Et_p(X)$ to such a scheme,
maintaining  the validity of \eqref{cohid} for $p$-torsion $A$.  We shall also make the zig-zag of 
 \eqref{eq:77}.

Let $s\Sets^{(p)}$ be the category of $p$-finite simplicial sets: those simplicial sets $X$ where $\pi_0(X)$ is a finite set, and $\pi_i(X,x)$ is a finite $p$-group for all $x \in X_0$ and all $i > 0$, which is trivial for sufficiently large $i$.  
 We  define $\Et_p$ as the composition of three functors: the  {\'e}tale homotopy type, $p$-completion,
and homotopy limit, each of which we review in turn:
\begin{equation}\label{eq:79}
  s\Var_\C \xrightarrow{\Et} \text{pro-}s\Sets \xrightarrow{\text{pro-$p$ completion}} \text{pro-}s\Sets^{(p)} \xrightarrow{\mathrm{holim}} s\Sets,
\end{equation}

As the first functor
\begin{equation} \label{Friedlander}
    s\Var_\C \xrightarrow{\Et} \text{pro-}s\Sets
\end{equation}
we shall take Friedlander's rigid \'etale homotopy type \cite{Fried82}. To each hypercover $U_\bullet \to X$ of a locally Noetherian schemes $X$ one gets a simplicial set $[n] \mapsto \pi_0(U_n)$, where $\pi_0(U_n)$ denotes the set of connected components of $U_n$ (denoted just $\pi(U_n)$ in \cite{Fried82}).  If $X_\bullet$ is a simplicial object in locally Noetherian scheme, we may similarly consider bisimplicial object $U_{\bullet,\bullet}$ forming a hypercover $U_{s,\bullet} \to X_s$ for each $s$: to this situation we associate the diagonal simplicial set $[n] \mapsto \pi_0(U_{n,n})$.  Friedlander then defines the \'etale homotopy type of $X_\bullet$ as a functor
\begin{align*}
  \mathrm{HRR}(X_\bullet)^\mathrm{op} & \to s\Sets\\
  (U_{\bullet,\bullet} \to X) & \mapsto ([n] \mapsto \pi_0(U_{n,n})),
\end{align*}
where $\mathrm{HRR}(X_\bullet)$ is a suitable category of \emph{rigid hypercovers}.  These are actually hypercovers equipped with extra data, making $\mathrm{HRR}(X_\bullet)$ into a filtered category. The details of how $\mathrm{HRR}(X_\bullet)$ is defined shall not matter for our applications. Let us emphasize however, that this construction outputs an inverse system of simplicial sets functorially in $X$, which is slightly stronger than outputting a pro-object\footnote{The point of this sentence is that the last step in~\eqref{eq:79}, taking homotopy limit, is not strictly functorial in pro-objects, only ``functorial up to weak equivalence''.  This objection is dismissed by upgrading to preferred inverse systems.}.

The $p$-profinite completion was introduced in \cite{MR1414543}, see also \cite{MR2136405}.  It associates to an inverse system $Y: j \mapsto Y_j$ in the category of simplicial sets another inverse system $Y^\wedge_p$ in the category of $p$-finite simplicial sets, and a map
\begin{equation*}
  Y \to Y^\wedge_p
\end{equation*}
inducing an isomorphism in ``continuous'' mod $p$ cohomology, defined as $H^*(Y;\F_p) = \colim_j H^*(Y_j;\F_p)$.  There is again an explicit construction which outputs an inverse system of $p$-finite simplicial sets, for instance one can for each $j$ consider all quotients $Y_j \to Z$ which are finite sets in each simplicial degree, then take a Postnikov truncation of a stage in the totalization of the Bousfield--Kan cosimplicial resolution of $Y$.  The pro-object $Y^\wedge_p$ is obtained by letting these stages vary over the natural numbers, $Z$ vary over finite quotients of $Y_j$, and $j$ vary over the indexing category of $Y$.

Combining these two constructions assigns to any $X \in s\Var_\C$ an $(\Et(X))^\wedge_p$ which is an inverse system of $p$-finite simplicial sets, together with a canonical isomorphism
\begin{equation*}
  H^*_\mathrm{et}(X;\F_p) \cong \colim H^*(\Et(X)^\wedge_p;\F_p),
\end{equation*}
where the left hand side is \'etale cohomology of $X$ with coefficients in the constant sheaf $X$, and the right hand side is the colimit of cohomology of the levels in the inverse system $(\Et(X))^\wedge_p$.

The last step is to replace the inverse system by its homotopy limit, which is a more subtle thing to do.  For any inverse system $j \mapsto Y_j$ of simplicial sets, there is a canonical map
\begin{equation*}
  \colim_j H^*(Y_j;\F_p) \to H^*(\holim_j Y;\F_p),
\end{equation*}
but there is no formal reason for this map to be an isomorphism, and in general it may well not be.  
It is known to be an isomorphism however, when the domain is finite-dimensional in each cohomological degree and vanishes in degree 1.  A recent reference for this statement in the form used here is \cite[Proposition 3.3.8 and Theorem 3.4.2]{DAGXIII}, but the insight that such pro-objects may sometimes be replaced with their homotopy limits without too much loss of information goes back to Sullivan's ``MIT notes'', see for instance \cite[Theorem 3.9]{MR2162361} for a profinite version.

We therefore define
\begin{equation*}
  \Et_p(X) := \holim (\Et(X) ^\wedge_p),
\end{equation*}
and have a canonical map
\begin{equation*}
  H^*_\mathrm{et}(X;\F_p) \to H^*(\Et_p(X);\F_p)
\end{equation*}
which is an isomorphism when the domain is finite-dimensional in each degree and vanishes in degree 1.

\begin{remark}
  Presumably the explicit construction given here could be replaced with any of the recent constructions leading to a pro-space in the $\infty$-categorical sense, e.g.\ \cite{barnea2016projective}, \cite{hoyois2018higher}, \cite{carchedietale}, or \cite[Section 12]{barwick2018exodromy}.  
  In  particular, some readers may prefer an approach based on the notion of the ``shape of an $\infty$-topos'', assigning a pro-space to any $\infty$-topos and hence to any site in the usual sense.  When $X$ is a scheme, Friedlander's explicit construction would then be replaced by the shape of the \'etale site of $X$, and the comparison maps constructed below should come from morphisms of sites
 $
    X(\C)^\mathrm{disc} \to X(\C)^\mathrm{an} \to X_{\mathrm{et}},
$
  where $X(\C)^\mathrm{disc}$ denotes the site corresponding to the set $X(\C)$ in the discrete topology.
\end{remark}

\subsubsection{Comparison map}
\label{sec:comparison-map}

When $X \in s\Var_\C$, the Artin comparison gives a canonical isomorphism between $H^*_\mathrm{et}(X;\F_p)$ and the Cech cohomology of $X(\C)^\mathrm{an}$, the complex points in the analytic topology.  Since complex varieties are paracompact and locally contractible in the analytic topology (since they are triangulable), Cech cohomology with constant coefficients is also isomorphic to singular cohomology.  In total we obtain an isomorphism
\begin{equation*}
  H^*(\Sing^\mathrm{an}(X);\F_p) \cong H^*_{\mathrm{et}}(X;\F_p).
\end{equation*}
Above we explained how \'etale cohomology is calculated by the space $\Et_p(X)$ in good cases, we now finally explain how to define a comparison map $\Sing^\mathrm{an} \to \Et_p(X)$, or at least a zig-zag.

Let $U_{\bullet, \bullet}$ be a levelwise hypercover as after \eqref{Friedlander}. 
The scheme $\Spec(\C^{\Delta^n})$ is connected, so that all maps to $U_{s,t}$ land in the same connected component.  Therefore we obtain well defined maps $\Sing^\mathrm{an}_n(U_{s,t}) \to \pi_0(U_{s,t})$ which are invariant under simplicial operations in the $n$-direction, and hence induce continuous maps
\begin{equation*}
  |\Sing^\mathrm{an}(U_{s,t})| \to \pi_0(U_{s,t})
\end{equation*}
for all $s,t$.  Moreover $U_{s,\bullet}^\mathrm{an} \to X_s^\mathrm{an}$ is a topological hypercover, which implies that $|U_{s,\bullet}^\mathrm{an}| \to X_s^\mathrm{an}$ is a weak equivalence.  Therefore the natural map
\begin{equation*}
  \Sing^\mathrm{an}(U_{s,\bullet}) \to \Sing^\mathrm{an}(X_s)
\end{equation*}
is a weak equivalence of simplicial sets for all $s$, where in the domain we implicitly pass to diagonal simplicial set.  Combining all this, and taking geometric realization in the $s$-direction, we obtain a zig-zag of maps of simplicial sets
\begin{equation*}
  \Sing^\mathrm{an}(X_\bullet) \xleftarrow{\simeq} \Sing^\mathrm{an}(U_{\bullet,\bullet}) \to \big([n] \mapsto \pi_0(U_{n,n})\big),
\end{equation*}
natural in the hypercover $U_{\bullet,\bullet} \to X_\bullet$.  Composing with the canonical map to the $p$-completion and taking homotopy limit over hypercovers of $X$, we obtain the desired zig-zag as
\begin{equation*}
  \Sing^\mathrm{an}(X) \xleftarrow{\simeq} \bigg(\holim_{U \in \mathrm{HRR}(X)} \Sing^\mathrm{an}(U)\bigg) \longrightarrow \Et_p(X).
\end{equation*}
Together with the canonical map $\Sing^\mathrm{an}_0(X) \to \Sing^\mathrm{an}(X)$, this finishes the construction of the diagram~\eqref{eq:77} of realization functors.

\subsection{Betti realization of $\mathcal{A}_{g,\C}$}
\label{sec:betti-real-mathc}

Let us finally establish the last desideratum, item~(\ref{item:12}), asserting that the Betti realization of the simplicial variety arising from an atlas $U \to \mathcal{A}_{g,\C}$ is a model for $B\Sp_{2g}(\Z)$.

\begin{example}\label{ex:analytic-fiber-products}
  Let $U,V \in \Var_\C$ and let $f: U \to V$ be a smooth surjection.  Then $U(\C)^\mathrm{an}$ and $V(\C)^\mathrm{an}$ are smooth manifolds and $f^\mathrm{an}: U(\C)^\mathrm{an} \to V(\C)^\mathrm{an}$ is a surjective submersion in the differential geometric sense.  Then we can form an object $U_\bullet \in s\Var_\C$ by letting $U_n$ be the $n$-fold fiber product of $U$ over $V$.  Taking analytic space commutes with fiber products, so $(U_\bullet(\C))^\mathrm{an} \to V(\C)^\mathrm{an}$ is also the simplicial object arising from iterated fiber products of the surjective submersion $U(\C)^\mathrm{an} \to V(\C)^\mathrm{an}$.  It follows that
  \begin{equation*}
    |U_\bullet(\C)^\mathrm{an}| \to V(\C)^\mathrm{an}
  \end{equation*}
  has contractible point fibers, and standard arguments show that it is a Serre fibration.  Hence
  \begin{equation*}
    \Sing^\mathrm{an}(U_\bullet) \to \Sing^\mathrm{an}(V)
  \end{equation*}
  is a weak equivalence, too.  
 
  \end{example}

\begin{example}\label{ex:an-actual-map}  
  Let $X_g$ be the simplicial variety arising from an atlas $U \to \mathcal{A}_g$, or even just a smooth surjective map, i.e.\ $X_g([n])$ is the $(n+1)$-fold iterated fiber product of $U$ over $\mathcal{A}_g$.  If $U' \to \mathcal{A}_g$ is another smooth surjection, then they may be compared using the bisimplicial variety $([n],[m]) \mapsto X_g([n]) \times_{\mathcal{A}_g} X'_g([m])$.  By Example~\ref{ex:analytic-fiber-products}, the projection
  \begin{equation*}
    \Sing^\mathrm{an}(X_g \times_{\mathcal{A}_g} X'_g([m])) \to \Sing^\mathrm{an}(X'_g([m]))
  \end{equation*}
  is a weak equivalence, and hence the same holds after taking geometric realization in the $m$-direction.  We deduce
  \begin{equation*}
    \Sing^\mathrm{an}(X_g) \xleftarrow{\simeq} \Sing^\mathrm{an}(X''_g) \xrightarrow{\simeq} \Sing^\mathrm{an}(X'_g),
  \end{equation*}
  where $X''_g$ is the simplicial variety obtained by iterated fiber products of $U_g \times_{\mathcal{A}_g} U'_g \to \mathcal{A}_g$.  

  Then $\mathrm{Sing}^\mathrm{an}(X_g)$ is a model for $B\Sp_{2g}(\Z)$. Indeed, we may use the quasiprojective variety $\mathcal{A}_g(N)$ (the $\Gamma_g(N) := \ker(\Sp_{2g}(\Z) \rightarrow \Sp_{2g}(\Z/N))$-cover of $\Cal{A}_g$, which parametrizes a trivialization of the $N$-torsion) as atlas for $N \geq 4$.  The simplicial variety arising from the atlas $\mathcal{A}_g(N) \to \mathcal{A}_g$ is isomorphic to the Borel construction of $\Sp_{2g}(\Z/N)$ acting on $\mathcal{A}_g(N)$.  The action of $\Sp_{2g}(\Z/N)$ on the space $(\mathcal{A}_g(N))^\mathrm{an} \cong \mathbb{H}_g/\Gamma_g(N)$ is the canonical one arising from the extension of the action of $\Gamma_g(N) < \Sp_{2g}(\Z)$, so we get 
  \begin{align*}
    \Sing^\mathrm{an}(X_g) & = \Sing^\mathrm{an}(\mathcal{A}_g(N) \moddd \Sp_{2g}(\Z/N)) = \Sing^\mathrm{an}(\mathcal{A}_g(N)) \moddd \Sp_{2g}(\Z/N) \\
    & = (\Sing(\mathbb{H}_g) / \Gamma_g(N)) \moddd \Sp_{2g}(\Z/N).
  \end{align*}
  Here ``$\moddd$'' denotes the Borel construction (homotopy orbits): explicitly, when group $G$ acts on a $X$, we write $X \moddd G$ for the usual simplicial object with $n$-simplices $G^n \times X$.
  At the last step we used the fact that, since the quotient $\mathbb{H}_g \to \mathcal{A}_g(N)$ is a covering map, there is an isomorphism of simplicial sets
 $    \Sing^\mathrm{an}(\mathcal{A}_g(N)) \cong (\Sing(\mathbb{H}_g))/\Gamma_g(N)$.

 We may then finally
   use that $\mathbb{H}_g$ is contractible, replace it by a point and the quotient by $\Gamma_g(N)$ by the homotopy quotient: 
   \begin{equation*}
    (\Sing(\mathbb{H}_g) / \Gamma_g(N)) \moddd \Sp_{2g}(\Z/n) \xleftarrow{\simeq}
    (E\Sp_{2g}(\Z) \times \Sing(\mathbb{H}_g)) \moddd \Sp_{2g}(\Z) \xrightarrow{\simeq} B\Sp_{2g}(\Z).
  \end{equation*}
\end{example}

\begin{example}\label{ex:Siegel-discrete}
  For later use, restricting the above discussion to zero-simplices yields an equivalence of groupoids  
  \begin{equation}\label{eq:43}
    \Sing^\mathrm{an}_0(X_g) \xrightarrow{\simeq} N(\mathcal{A}_g(\C)),
  \end{equation}
where  the domain is the simplicial set obtained by taking $\C$ points levelwise in the simplicial variety  $X_g$, and the codomain denotes the nerve of the groupoid whose objects are rank $g$ principally polarized abelian varieties $(A,\mathcal{L})$ over $\Spec(\C)$ and whose morphisms are isomorphisms of such.  By uniformization, we also have an equivalence
  \begin{equation*}
    \mathbb{H}_g^\delta \moddd \Sp_{2g}(\Z) \xrightarrow{\simeq} N(\mathcal{A}_g(\C)),
  \end{equation*}
  where $\mathbb{H}_g^\delta$ denotes the Siegel upper half space in the discrete topology.  The equivalence is induced by the usual construction, sending a symmetric matrix $\Omega$ with positive imaginary part to the abelian variety $\C^g/(\Z^g + \Omega \Z^g)$ in the usual principal polarization.

  To summarize, the diagram~\eqref{eq:77} for $X = X_g$ becomes a model for the evident maps
  \begin{equation*}
    \mathbb{H}_g^\delta \moddd \Sp_{2g}(\Z) \to \mathbb{H}_g \moddd \Sp_{2g}(\Z) \to \big( \mathbb{H}_g \moddd \Sp_{2g}(\Z) \big)^\wedge_p,
  \end{equation*}
  where the first map is induced by the identity map of Siegel upper half space, from the discrete to the Euclidean topology.  The composition is our $\Aut(\C)$-equivariant model for
  \begin{equation*}
    |N(\mathcal{A}_g(\C))| \to (B\Sp_{2g}(\Z))^\wedge_p.
  \end{equation*}  
\end{example} As explained above, we may use $\mathcal{A}_g$ to exhibit $B\Sp_{2g}(\Z)$ as the Betti realization of a simplicial variety defined over $\Q$, and hence construct an $\Aut(\C)$-action on its $p$-completion (at least for $g \geq 3$ where $\Sp_{2g}(\Z)$ is perfect).  It remains to see that this structure is compatible with the structure which constructs the spectrum $\KSp(\Z)$ out of the $B\Sp_{2g}(\Z)$, i.e.\ the $\Gamma$-space structure.

\subsection{A Gamma-object in simplicial varieties}
\label{sec:gamma-objects-simpl}

In this section we use the moduli stacks $\mathcal{A}_g$ to define a functor from $\Gamma^\mathrm{op}$ to simplicial complex varieties, such that the composition  
\begin{equation*}
  Z\colon \Gamma^\mathrm{op} \to s\mathrm{Var}_\C \xrightarrow{\Sing^\mathrm{an}} s\Sets
\end{equation*}
is naturally homotopy equivalent to $T \mapsto |\SP_T(\Z)|$.  We first discuss how to construct a functor $T \mapsto \mathcal{A}(T) \simeq (\coprod_{g \geq 0} \mathcal{A}_g)^{T \setminus \{\ast\}}$ from $\Gamma^\mathrm{op}$ to groupoids, modeled on how we defined $T \mapsto \SP_T(\Z)$.

To avoid excessive notation, let us agree that for a scheme $S$ we denote objects of $\mathcal{A}_g(S)$ like $(A,\mathcal{L})$, where $A$ is an abelian scheme over $S$ and $\mathcal{L}$ is a principal polarization.  On the set level, $A$ is an abbreviation for a scheme $A$ and maps of schemes $\pi: A \to S$ and $e: S \to A$, with the property that they make $A$ into a rank $g$ abelian scheme over $S$ with identity section $e$.  Similarly, $\mathcal{L}$ is an abbreviation for a line bundle $\mathcal{L}$ on $A \times_S A$, rigidified by non-zero section $i$ of $\mathcal{L}$ over $A \times_S \{e\} \hookrightarrow A \times_S A$ and $i'$ over $\{e\} \times_S A \hookrightarrow A \times_S A$ agreeing with $i$ over $(e,e): S \to A \times_S A$, with the property that $(\mathcal{L},i,i')$ is symmetric under swapping the two factors of $A$, the restriction $\Delta^*\mathcal{L}$ along the diagonal $\Delta : A \to A \times_S A$ is ample, and the morphism $A \to A^\vee$ induced by $\mathcal{L}$ is an isomorphism.  We shall say ``$(A,\mathcal{L})$ is a principally polarized abelian variety over $S$'' to mean that we are given all this data for some $g \geq 0$.

For each finite pointed set $T$ we let $\mathcal{A}(T)$ denote the category whose objects  are $(A,\mathcal{L},\phi)$ where $(A,\mathcal{L})$ is a principally polarized abelian variety over $S$, which is a scheme over $\Spec(\Q)$, and $\phi: \Z^T \to \mathrm{End}(A)$ is a ring homomorphism, 
 with the property that the image of $\phi$ is symmetric with respect to the Rosati involution defined by $\phi$.
 In particular, $\mathcal{L}$ restricts to a principal polarization on the abelian subvarieties $A_t \subset A$, defined as $A_t = \Ker(1-\phi(e_t)) \subset A$ for all $t \in T$.  For $e = \sum_{t \in T \setminus \{\ast\}} e_t$ we similarly have $\Ker(1 - \phi(e)) \subset A$, which we shall denote $eA$.  Addition in the group structure on $A$ defines an isomorphism of abelian varieties $\oplus_{t \neq \ast} A_t \to eA$.  We now define morphisms in $\mathcal{A}(T)$ to be isomorphisms of abelian schemes $eA \to eA'$ restricting to isomorphisms between the $A_t$ for all $t \in T$ and preserving polarizations.  Forgetting everything but $S$ makes this category $\mathcal{A}(T)$ fibered in groupoids over the category of schemes over $\Spec(\Q)$, and the forgetful map
\begin{equation}\label{eq:41}
  \begin{aligned}
    \mathcal{A}(T) & \to \prod_{t \in T \setminus \{\ast\}} \bigg(\coprod_{g = 0}^\infty \mathcal{A}_g \bigg)\\
    (A, \mathcal{L},\phi) & \mapsto \big((A_t,\mathcal{L}_{\vert A_t \times A_t})\big)_{t \in T \setminus \{\ast\}}
  \end{aligned}
\end{equation}
defines an equivalence of stacks over $\Spec(\Q)$.  Moreover, the association $T \mapsto \mathcal{A}(T)$ defines a functor from $\Gamma^\mathrm{op}$ to (the 1-category of) such fibered categories: functoriality is again by precomposing the map $\Z^T \to \mathrm{End}(A)$.

To turn $\mathcal{A}(T)$ into a simplicial scheme we rigidify the objects.  To be specific, let us take $U(T)$ to be a scheme classifying the functor which sends $(S \to \Spec(\Q))$ to the set of tuples $(A,\mathcal{L},\phi,j)$, where $(A,\mathcal{A},\phi) \in \mathcal{A}(T)$ as above, and $j: A \hookrightarrow \mathbb{P}^{N-1}_S$ is an embedding such that $\mathcal{O}(1)$ restricts to $3\Delta^*(\mathcal{L})$ on $A$.  This functor is represented by a locally closed subscheme of a finite product of Hilbert schemes, and hence is quasi-projective over $\Spec(\Q)$, as in  \cite[Chapter 6]{GIT}.  Finally, we extract a simplicial scheme $Z(T)$ from the map $U(T) \to \mathcal{A}(T)$ by taking iterated fiber products.  Then $n$th space classifies $(n+1)$-tuples $(A_0, \dots, A_n)$ of abelian schemes over $S$, each equipped principal polarizations $\mathcal{L}_i$ and with embeddings $j_i: A_i \subset \mathbf{P}^{N_i-1}_S$ as above and ring homomorphisms $\phi_i: \Z^T \to \mathrm{End}(A_i)$, defining principally polarized abelian subvarieties $eA_i = \Ker(1 - \phi_i(e)) \subset A_i$, as well as isomorphisms of abelian varieties
\begin{equation*}
  e A_0 \xrightarrow{\cong} e A_1 \xrightarrow{\cong} \dots \xrightarrow{\cong} eA_n
\end{equation*}
preserving polarizations (but no compatibility imposed on projective embeddings).

\begin{proposition}
  There is a zig-zag of weak equivalences of simplicial sets
  \begin{equation*}
    \Sing^\mathrm{an}(\Spec(\C) \times_{\Spec(\Q)} Z(T)) \xleftarrow{\simeq} \dots \xrightarrow{\simeq} N_\bullet(\SP_T(\Z))
  \end{equation*}
  natural in $T \in \Gamma^\mathrm{op}$.

  In particular, $\Sing^\mathrm{an}(\Spec(\C) \times_{\Spec(\Q)} Z(S^0)) \simeq \coprod_g N \Sp_{2g}(\Z)$.
\end{proposition}
\begin{proof}[Proof sketch]
  We have explained a smooth surjection $U(T) \to \mathcal{A}(T) \xrightarrow{\simeq} (\coprod_g \mathcal{A}_g)^{T \setminus \{\ast\}}$, which up to equivalence may be rewritten as a coproduct of smooth surjections into stacks of the form $\mathcal{A}_{g_1} \times \dots \times \mathcal{A}_{g_m}$.  After base changing to $\Spec(\C)$ all simplicial varieties arising are quasi-projective over $\Spec(\C)$.  The weak equivalence now follows
 by an argument similar to Example~\ref{ex:an-actual-map}, which can also be used to produce an explicit zig-zag.  Since all constructions are strictly functorial in $T \in \Gamma^\mathrm{op}$, so is the resulting $T \mapsto Z(T)$.
\end{proof}

Taking complex points (in the discrete topology) of $\mathcal{A}_g$ gives the groupoid $\mathcal{A}_g(\C)$ whose objects are $(A,\mathcal{L})$, principally polarized abelian varieties over $\Spec(\C)$, and whose morphisms are isomorphisms of such.  In this groupoid all automorphism groups are finite, but it has continuum many isomorphism classes of objects for $g > 0$.  Hence $|\mathcal{A}_g(\C)|$ is a disjoint union of continuum many $K(\pi,1)$'s for finite groups.  The equivalence~(\ref{eq:41}) for $T = S^0$ implies a weak equivalence of simplicial sets
\begin{equation}\label{eq:42}
  \Map(\Spec(\C),Z(S^0)) \xrightarrow{\simeq} \coprod_{g \geq 0} N(\mathcal{A}_g(\C))
\end{equation}
as in Example~\ref{ex:Siegel-discrete}.  This, and the $\Aut(\C)$-equivariant map to $\Et_p(Z(S^0))$, will eventually lead to commutativity of the diagram~\eqref{eq:12}.

\subsection{Galois action on symplectic $K$-theory}
\label{sec:galo-acti-sympl}

We finally construct the promised action of the group $\Aut(\C)$ on the spectrum $\KSp(\Z;\Z_p)$.  For simplicity we first assume $p > 3$, which has the convenient effect that $H^1_\mathrm{et}(\mathcal{A}_g;\F_p) = H^1(B\Sp_{2g}(\Z);\F_p) = 0$ for all $g$.  (For $p=3$ this fails for $g=1$ and for $p=2$ it fails for $g = 1$ and $g=2$.  A mild variation of the argument applies also in those two cases; see below.)

First, recall that we described a composite functor
\begin{equation*}
  \Gamma^\mathrm{op} \xrightarrow{Z} s\Var_\Q \xrightarrow{-\otimes_\Q \C} s\Var_\C,
\end{equation*}
whose composition with $\Sing^\mathrm{an}$ is equivalent to the $\Gamma$-space delooping $\KSp(\Z)$.  We get a zig-zag of maps between $\Gamma$-spaces
\begin{equation*}
  \Sing^\mathrm{an}(Z) \xleftarrow{\simeq} \dots \xrightarrow{\text{comparison}} \Et_p(Z)
\end{equation*}
which has the property that evaluated on any object $S \in \Gamma$ it induces an isomorphism in mod $p$ cohomology.  In addition to vanishing $H^1(-;\F_p)$, this requires that $H^*(B\Sp_{2g}(\Z);\F_p)$ is finite-dimensional in each degree, which is well known.

Now choose a simplicial set modeling a Moore space $M(\Z/p^k,2)$ for all $k$, and choose maps $M(\Z/p^{k+1},2) \to M(\Z/p^k,2)$ corresponding to reduction modulo $p^k$.  Then we get an induced functor  
\begin{equation*}
  \Gamma^{\mathrm{op}} \xrightarrow{M(\Z/p^k,2) \wedge Z} s\Var_\Q \xrightarrow{- \otimes_\Q \C} s\Var_\C,
\end{equation*}
giving rise to two $\Gamma$-spaces by applying $\Sing^\mathrm{an}$ or $\Et_p$, and a zig-zag
\begin{equation*}
  B^\infty(M(\Z/p^k,2) \wedge \Sing^\mathrm{an} (Z \otimes_\Q \C)) \xleftarrow{\simeq} \dots \xrightarrow{\text{comparison}} B^\infty(M(\Z/p^k,2) \wedge \Et_p (Z \otimes_\Q \C)).
\end{equation*}
where $B^{\infty}$ is as in \eqref{Binftydef}. Since any mod $p$ homology isomorphism becomes a weak equivalence after smashing with $M(\Z/p^k,2)$, and since the comparison map comes from a map of $\Gamma$-spaces which is a mod $p$ homology equivalence when evaluated on any $S \in \Gamma^\mathrm{op}$ as long as $p > 3$, we have produced a weak equivalence of spectra
\begin{equation*}
  M(\Z/p^k,2) \wedge \KSp(\Z) \simeq B^\infty(M(\Z/p^k,2) \wedge \Sing^\mathrm{an}(Z\otimes_\Q \C))) \xleftarrow{\simeq} \dots \xrightarrow{\simeq} B^\infty(M(\Z/p^k,2) \wedge \Et_p(Z\otimes_\Q \C))).
\end{equation*}
Desuspending twice and taking homotopy inverse limit over $k$, we get
\begin{equation*}
  \KSp(\Z;\Z_p) \simeq \holim_k (\mathbb{S}/p^k) \wedge \KSp(\Z) \simeq \holim_k M(\Z/p^k) \wedge B^\infty(\Et_p(Z\otimes_\Q \C))).
\end{equation*}
But by functoriality of the delooping machine, the group $\Aut(\C)$ manifestly acts by spectrum maps on $B^\infty(\Et_p(Z \otimes_\Q \C))$.  Hence this equivalence can be viewed as a homotopy action on the $p$-completed symplectic $K$-theory spectrum, and in particular it constructs an action on homotopy groups
\begin{align*}
  \KSp_n(\Z;\Z/p^k) & \cong \pi_n((\mathbb{S}/p^k) \wedge B^\infty(\Et_p(Z\otimes_\Q \C))))\\
  \KSp_n(\Z;\Z_p) & \cong \varprojlim_{k} \pi_n((\mathbb{S}/p^k) \wedge B^\infty(\Et_p(Z\otimes_\Q \C)))).
\end{align*}

\begin{remark}
  For $p \leq 3$ a mild variant of the argument works.  The only problem with those small primes was that $\mathcal{A}_g$ has non-trivial mod $p$ cohomology for small $g$, which prevents us from controlling the mod $p$ cohomology of the inverse limit involved in forming $\Et_p$.  But smashing with $M(\Z/p^k,2)$ makes any simplicial set be simply connected, so if we do that operation before taking $\Et_p$ there is nothing special about the small primes.  As a side-effect, this version of the argument will not make use of the calculation $H_1(B\Sp_{2g}(\Z))$.
\end{remark}

We finally discuss the compatibility of actions mentioned in Proposition~\ref{prop:5.2}.  Namely, as an instance of the spectrum map~(\ref{eq:39}) we have 
\begin{equation*}
  \Sigma^\infty (M(\Z/p^k,2) \wedge \mathrm{Et}_p(Z(S^0)\otimes_\Q \C))) \to B^\infty(M(\Z/p^k,2) \wedge \mathrm{Et}_p(Z\otimes_\Q \C))),
\end{equation*}
extracted functorially from the $\Gamma$-space $T \mapsto M(\Z/p^k,2) \wedge \mathrm{Et}_p(Z(T)\otimes_\Q \C))$, and hence equivariant for the $\Aut(\C)$-action.
By the argument of Example~\ref{ex:Siegel-discrete} above, the map of simplicial sets
\begin{equation*}
  \Sing^\mathrm{an}_0(Z(S^0)\otimes_\Q \C)) \to \mathrm{Et}_p(Z(S^0)\otimes_\Q \C))
\end{equation*}
is also equivariant.  Hence we get an equivariant map of spectra
\begin{equation*}
  \Sigma^\infty (M(\Z/p^k,2) \wedge \Sing^\mathrm{an}_0(Z(S^0)\otimes_\Q \C))) \to B^\infty (M(\Z/p^k,2) \wedge \mathrm{Et}_p(Z\otimes_\Q \C))).
\end{equation*}
Shifting degrees by 2 and taking homotopy groups we get a homomorphism
\begin{equation*}
  \pi_n^s(\Sing^\mathrm{an}_0(Z(S^0)\otimes_\Q \C));\Z/p^k) \to \KSp_n(\Z/\Z/p^k),
\end{equation*}
which is equivariant for the action constructed above.  Now finally, the equivalence~(\ref{eq:42}) is also $\Aut(\C)$-equivariant for the evident action on $\mathcal{A}_g(\C)$, i.e.\ the one changing reference maps $\pi: A \to \Spec(\C)$ of abelian schemes over $\Spec(\C)$.  Restricting attention to the path component corresponding to abelian varieties of rank $g$, we have shown that the homomorphism
\begin{equation*}
  \pi_n^s(|\mathcal{A}_g(\C)|;\Z/p^k) \to \KSp_n(\Z/\Z/p^k),
\end{equation*}
induced from mapping $N(\mathcal{A}_g(\C)) \simeq \mathbb{H}_g^\delta \moddd \Sp_{2g}(\Z) \to B\Sp_{2g}(\Z)$, is equivariant for $\Aut(\C)$.  This is the commutativity of the diagram~\eqref{eq:12}.

\bibliographystyle{amsalpha}
\bibliography{Bibliography}

\end{document}